\documentclass{amsart}%
\usepackage{amssymb}
\usepackage{amsfonts}
\usepackage{geometry}
\usepackage{color}
\usepackage{amsmath}
\usepackage{verbatim}
\usepackage{graphicx}%
\usepackage{tcolorbox}
\usepackage{mathtools}
\usepackage{caption}
\usepackage{subcaption}
\usepackage{hyperref}
\providecommand{\U}[1]{\protect\rule{.1in}{.1in}}

\renewcommand{\hat}{\widehat}
\renewcommand{\tilde}{\widetilde}
\renewcommand{\bar}{\overline}
\renewcommand{\epsilon}{\varepsilon}

\newcommand{\roof}{\mathcal{R}} 
\newcommand{\supr}{\mathcal{SU}} 
\newcommand{\stopp}{\mathcal{S}} 

\DeclareMathOperator{\arcsinh}{arcsinh}

\newtheorem{theorem}{Theorem}
\theoremstyle{plain}

\newtheorem{corollary}[theorem]{Corollary}

\newtheorem{definition}[theorem]{Definition}

\newtheorem{lemma}[theorem]{Lemma}

\newtheorem{proposition}[theorem]{Proposition}
\newtheorem{remark}[theorem]{Remark}

\newtheorem*{statement*}{Statement $\mathbf{U_t ( \{k_j\}_{j \geq 0}, \epsilon , \Gamma, p )}$}
\newtheorem*{statement2*}{Statement $\mathbf{V}_t ^{\mathcal{F}}$}

\numberwithin{equation}{section}
\begin{document}
\title[The T1 theorem for pairs of doubling measures fails when p not 2]{The scalar $T1$ theorem for pairs of doubling measures fails for Riesz transforms when $p$ not $2$}
\author[M. Alexis]{Michel Alexis}
\address{Department of Mathematics \& Statistics, McMaster University, 1280 Main Street
West, Hamilton, Ontario, Canada L8S 4K1}
\email{micalexis.math@gmail.com}
\author[J. L. Luna-Garcia]{Jos\'e Luis Luna-Garcia}
\address{Department of Mathematics \& Statistics, McMaster University, 1280 Main Street
West, Hamilton, Ontario, Canada L8S 4K1}
\email{lunagaj@mcmaster.ca}
\author[E.T. Sawyer]{Eric T. Sawyer}
\address{Department of Mathematics \& Statistics, McMaster University, 1280 Main Street
West, Hamilton, Ontario, Canada L8S 4K1 }
\email{Sawyer@mcmaster.ca}
\author[I. Uriarte-Tuero]{Ignacio Uriarte-Tuero}
\address{Department of Mathematics, University of Toronto\\
Room 6290, 40 St. George Street, Toronto, Ontario, Canada M5S 2E4\\
(Adjunct appointment)\\
Department of Mathematics\\
619 Red Cedar Rd., room C212\\
Michigan State University\\
East Lansing, MI 48824 USA}
\email{ignacio.uriartetuero@utoronto.ca}
\thanks{E. Sawyer is partially supported by a grant from the National Research Council
of Canada}
\thanks{I. Uriarte-Tuero has been partially supported by grant MTM2015-65792-P
(MINECO, Spain), and is partially supported by a grant from the National
Research Council of Canada}
\thanks{The authors thank Max Gie\ss ler noticing an error in the exponent of Lemma \ref{reduction_extra} in the previous version of this paper.}

\begin{abstract}

We show that for an individual Riesz transform in the setting of doubling measures, the \emph{scalar} $T1$ theorem fails when $p \neq 2$: for each $ p \in (1, \infty) \setminus \{2\}$, we construct a pair of doubling measures $(\sigma, \omega)$ on $\mathbb{R}^2$ with doubling constant close to that of Lebesgue measure that also satisfy the scalar $\mathcal{A}_p$ condition and the full scalar $L^p$-testing conditions for an individual Riesz transform $R_j$, and yet $\left ( R_j \right )_{\sigma} : L^p (\sigma) \not \to L^p (\omega)$.

 On the other hand, we improve upon the \emph{quadratic}, or vector-valued, $T1$ theorem of \cite{SaWi} when $p \neq 2$ on pairs of doubling measures: we dispense with their vector-valued weak boundedness property to show that for pairs of doubling measures, the two-weight $L^p$ norm inequality for the vector Riesz transform is characterized by a quadratic Muckenhoupt condition $A_{p} ^{\ell^2, \operatorname{local}}$, and a quadratic testing condition.

Finally, in the appendix, we use constructions of \cite{KaTr} to show that the two-weight norm inequality for the maximal function cannot be characterized solely by the $A_p$ condition when the measures are doubling, contrary to reports in the literature.

\end{abstract}
\maketitle

\tableofcontents

\section{Introduction}
The $T1$ conjecture of Nazarov, Treil, and Volberg on the boundedness
of the Hilbert transform $H$ between weighted spaces $L^{2}\left(  \sigma\right)
$ and $L^{2}\left(  \omega\right)  $ was proved in the
two part paper \cite{LaSaShUr3,Lac}, with the assumption on common
point masses subsequently removed in \cite{Hyt2}. In the case that $\sigma$ and $\omega$ are doubling measures,  this result was extended to vector Riesz transforms in \cite{LaWi}, and  partial progress towards a $T1$ theorem for Calder\'on-Zygmund operators was made in \cite{SaShUr7,AlSaUr}. The natural conjecture for the $L^p$ analog is as follows: given a Calder\'on-Zygmund operator $T$ with kernel $K$ and doubling measures $\sigma$ and $\omega$, one has
\begin{equation}\label{eq:naive_T1}
	\mathfrak{N}_{T, p } (\sigma, \omega) \lesssim \mathfrak{T}_{T, p } ^{\operatorname{local}} (\sigma, \omega) + \mathfrak{T}_{T ^* ,p'} ^{\operatorname{local}} (\omega , \sigma) + A_p(\sigma, \omega) \ ,
\end{equation}
where the scalar Muckenhoupt $A_p$ characteristic is given by
\[
	A_p ( \sigma, \omega) \equiv \sup\limits_{I} \left ( \frac{1}{|I|} \int\limits_I \sigma  \right )^{\frac{1}{p'}} \left ( \frac{1}{|I|} \int\limits_I \omega \right ) ^{\frac{1}{p}} \, ,
\]
the scalar local testing constant is given by
\[
	\left ( \mathfrak{T}_{T,p} ^{\operatorname{local}} (\sigma, \omega)  \right )^p \equiv \sup\limits_{I} \frac{1}{|I|_{\sigma}} \int\limits_I | T_{\sigma} \mathbf{1}_I (x)|^p d \omega(x) \, ,
 \]
 and $\mathfrak{N}_{T, p } (\sigma, \omega)$ denotes the $L^p (\sigma) \to L^p (\omega)$ norm of the operator 
 \[
   T_{\sigma} f (x) \equiv T (f \sigma) (x) = \int K(x,y) f(y) d \sigma (y) \, .
 \]
See Section \ref{subsection:norm} for how we make sense of $T_{\sigma}$, the norm inequality and the testing characteristics. 

When $T$ is the vector Riesz transform $\mathbf{R}$, \cite{SaWi} proved a $T1$ theorem on $L^p$ by replacing the scalar conditions above by their stronger \emph{quadratic}, or vector-valued, analogs introduced by Hyt\" onen and Vuorinen.
The quadratic $A_{p} ^{\ell^2, \operatorname{local}}$ condition, a variant of one introduced by \cite{HyVu}, is given by the existence of a finite constant $C$ for which
\begin{align}\label{forward_Ap_local}
  \left \| \left ( \sum\limits_{J} a_J ^2  \left (\frac{\left |J \right |_{\sigma}}{\left | J \right |} \right )^2 \mathbf{1}_{J} \right )^{\frac{1}{2}} \right \|_{L^p (\omega)} \leq C  \left \|\left ( \sum\limits_{J}  a_J ^2 \mathbf{1}_{J} \right )^{\frac{1}{2}} \right \|_{L^p (\sigma)} ,
\end{align}
where \eqref{forward_Ap_local} holds over all countable collections of cubes $\{J\}$ and sequences $\{a_J\}$ indexed by the sequence of cubes. The best constant $C$ in \eqref{forward_Ap_local} is denoted by $A_{p} ^{\ell^2, \operatorname{local}} (\sigma, \omega)$. Similarly, the quadratic local testing condition $\mathfrak{T}_{T, p} ^{\ell^2, \operatorname{local}} (\sigma, \omega)$ is defined as
\begin{equation}\label{eq:quad_testing_local_0}
\left \| \left ( \sum\limits_{J} a_J ^2  \left (T_{\sigma} \mathbf{1}_J \right )^2 \mathbf{1}_{J} \right )^{\frac{1}{2}} \right \|_{L^p (\omega)} \leq C  \left \|\left ( \sum\limits_{J}  a_J ^2 \mathbf{1}_{J} \right )^{\frac{1}{2}} \right \|_{L^p (\sigma)} \, ,
\end{equation}
where again the best constant $C$ above is denoted by $\mathfrak{T}_{T,p} ^{\ell^2, \operatorname{local}} (\sigma, \omega)$. In particular, \cite[Theorem 1, part (2)]{SaWi} required that the measure pair $(\sigma, \omega)$ satisfies the quadratic local testing conditions $\mathfrak{T}_{T, p} ^{\ell^2, \operatorname{local}} (\sigma, \omega)$ and $\mathfrak{T}_{T, p'} ^{\ell^2, \operatorname{local}} (\omega ,\sigma  )$ defined in \eqref{eq:quad_testing_local_0}, the quadratic Muckenhoupt condition $A_{p} ^{\ell^2, \operatorname{local}} (\sigma, \omega)$ defined in \eqref{forward_Ap_local}, and a
 quadratic weak-boundedness property defined in  \eqref{eq:wbp_quadratic}. 
Our first theorem below dispenses with the last condition, improving upon \cite[Theorem 1, part (2)]{SaWi}.
\begin{theorem}\label{thm:new_testing_thm_Lp_vector}
  Let $\sigma, \omega$ be doubling measures on $\mathbb{R}^n$ and $1<p < \infty$. Then
    \[
      \mathfrak{N}_{\mathbf{R},p} (\sigma, \omega) \lesssim \mathfrak{T}_{ \mathbf{R},p} ^{ \ell^2,\operatorname{local}} (\sigma, \omega) + \mathfrak{T}_{\mathbf{R},p'} ^{\ell^2 , \operatorname{local}}  (\omega, \sigma) + A_{p} ^{\ell^2, \operatorname{local}} (\sigma, \omega) + A_{p'} ^{\ell^2, \operatorname{local}} (\omega, \sigma) \, , 
    \]
    where the implicit constant only depends on $p$ and the doubling constants of $\sigma$ and $\omega$.
  \end{theorem}

\begin{remark}
    If the measures are doubling as in Theorem \ref{thm:new_testing_thm_Lp_vector}, both of the local quadratic Muckenhoupt conditions are comparable for \emph{doubling measures} by Lemma \ref{lem:duality_quadratic} part (\ref{lem:duality_quadratic_doubling}), and so only one of the two is needed.
\end{remark}

Aside from one instance where one needs the extreme energy reversal \cite[(6.14)]{SaWi}, which is implied by quadratic testing for the vector Riesz transform, the proof in \cite{SaWi} otherwise holds for general smooth Calder\'on-Zygmund operators $T$, and one can replace the quadratic testing conditions by the scalar local testing conditions for those parts of the proof. This is \emph{surprising} and suggests that even when $p \neq 2$ a $T1$ theorem only involving scalar conditions \emph{might} hold on pairs of doubling measures. However, such a result is only known for a much smaller class of measure pairs. Here is one such result, due to \cite{NTV2, Neu}.  
\begin{theorem}[\cite{NTV2, Neu}]\label{thm:new_testing_thm_Lp}
  Let $1<p<\infty$ and let $\sigma, \omega$ be measures on $\mathbb{R}^n$ for which $\sigma = \omega$ or $\sigma, \omega \in A_{\infty} (\mathbb{R}^n)$. If $T$ is a smooth Calder\'on-Zygmund operator, then
    \[
      \mathfrak{N}_{T,p} (\sigma, \omega) \lesssim \mathfrak{T}_{T,p} ^{\operatorname{local}} (\sigma, \omega) + \mathfrak{T}_{T^*,p'} ^{\operatorname{local}}  (\omega, \sigma) + A_{p}  (\sigma, \omega) + A_{p'}  (\omega, \sigma) \, , 
    \]
    where the implicit constant only depends on $p$, the Calder\'on-Zygmund data of $T$, and additionally the $A_{\infty}$ characteristics of $\sigma$ and $\omega$ in the second case.

  \end{theorem}
\begin{proof}[Proof sketch]
    If $\sigma = \omega$, this follows from \cite[Theorem 0.3]{NTV2}, noting that the required growth condition on $\sigma$ follows from the $A_p$ condition.
    
    If $\sigma, \omega \in A_{\infty} (\mathbb{R}^n)$, then both $\sigma$ and $\omega$ are absolutely continuous and satisfy a reverse H\" older inequality \cite{CoFe}. Thus if in addition we have $A_p  (\sigma, \omega) < \infty$, then there exists $\epsilon > 0$ for which we have
    \[
    \sup\limits_{Q} \left ( \frac{1}{|Q|} \int\limits_{Q} \sigma ^{1 + \epsilon} \right )^{\frac 1 p} \left ( \frac{1}{|Q|} \int\limits_{Q} \omega ^{1 + \epsilon} \right )^{\frac 1 p'} \lesssim A_p  (\sigma, \omega)^{1+\epsilon} < \infty \, .
    \]
    Thus the pair $(\sigma, \omega)$ satisfies the bump condition of \cite{Neu}. Hence there exists $W \in A_p (\mathbb{R}^n)$ for which
    \[
    \omega \lesssim W \lesssim \sigma^{1-p} 
    \]
    pointwise almost everywhere. Thus for all $f \in L^p (\sigma)$, the norm inequality holds since
    \[
   \int\limits_{\mathbb{R}^n} |T_{\sigma} f |^p \omega \lesssim \int\limits_{\mathbb{R}^n} |T (f \sigma) |^p W \lesssim   \int\limits_{\mathbb{R}^n} |f \sigma|^p W \lesssim \int\limits_{\mathbb{R}^n} |f|^p \sigma \, . \qedhere
    \]
\end{proof}
However when $p \neq 2$, the simple-minded $T1$ theorem philosophy breaks down in several instances.	In an unpublished note \cite{Naz}, Nazarov first showed that a scalar $T1$ theorem on $L^p$ fails for Haar multipliers, a class of well-localized operators. Then, \cite{AlLuSaUr2} showed that the scalar $T1$ theorem on $L^p$ fails for the Hilbert transform and a pair of measures $\sigma$ and $\omega$, neither of which is doubling. Now, our main Theorem \ref{thm:counter_example} below proves that when $p \neq 2$, the scalar two-weight $T1$ theorem cannot hold even for  ``nice'' doubling measures with doubling constant close to that of Lebesgue measure. Namely, in this setting with $T$ equal to any individual Riesz transform $R_j$, we show the scalar $T1$ theorem fails even if we replace the scalar local $A_p$ condition by the stronger scalar two-tailed $\mathcal{A}_p$ condition
  \begin{equation}
    \mathcal{A}_{p}\left(  \sigma,\omega\right)  \equiv \sup\limits_I \left( \frac{1}{ \left\vert I\right\vert} \int_{\mathbb{R}^n%
      }\frac{\left\vert I\right\vert ^{p}}{\left[  \ell \left ( I\right )
+\operatorname*{dist}\left(  x,I\right)  \right]  ^{n p}}d\omega\right)
^{\frac{1}{p}}\left( \frac{1}{ \left\vert I\right\vert}  \int_{\mathbb{R}^n}  \frac{\left\vert I\right\vert
  ^{p^{\prime}}}{\left[  \ell \left (  I\right ) +\operatorname*{dist}\left(
x,I\right)  \right]  ^{ n p^{\prime}}}d\sigma\right)  ^{\frac{1}{p^{\prime}}%
}<\infty  \label{Muck conditions} \, ,
\end{equation}
and we replace the scalar local testing constant by the bigger full testing constant
\[
	\left ( \mathfrak{T}_{T, p} ^{\operatorname{full}}  (\sigma, \omega) \right )^p \equiv \sup\limits_I \frac{1}{|I|_{\sigma}} \int\limits_{\mathbb{R}^n} \left | T_{\sigma} \mathbf{1}_I \right |^p d \omega \, . 
  \]
  Thus for doubling measures with doubling constant close to that of Lebesgue measure, inequality \eqref{eq:naive_T1} as well as any other like-minded conjecture involving classical scalar conditions, fails. 
In what follows, recall that the doubling constant of Lebesgue measure in the plane $\mathbb{R}^2$ is $4$, and $R_2$ denotes the vertical Riesz transform in the plane. A rotation extends the result below to the horizontal Riesz transform $R_1$.
  \begin{theorem}\label{thm:counter_example}
	  Let $p \in \left ( 1, \infty \right ) \setminus \left \{ 2 \right \}$. For every $\Gamma > 1$ and $\tau > 0$, there exist doubling measures $\sigma, \omega$ on $\mathbb{R}^2$ with doubling constant at most $4+\tau$, for which 
    \[
      \mathfrak{T}_{R_2, p} ^{\operatorname{full}}  (\sigma, \omega) \, , \mathfrak{T}_{R_2 , p'} ^{\operatorname{full}}  (\omega, \sigma) \lesssim 1 \text{ and } \mathcal{A}_p (\sigma, \omega) \, , \mathcal{A}_{p'} (\omega, \sigma) \lesssim 1 \, , 
    \]
    and yet
  \[
    \mathfrak{N}_{R_2, p} (\sigma, \omega) > \Gamma \, .
  \]
  Here the implicit constant only depends on $p$.
  \end{theorem}
  \begin{remark} Theorem \ref{thm:counter_example} also extends to $\mathbb{R}^n$. The correct ideas to extend the proofs of the propositions and lemmas in Section \ref{section:testing} crucial to the proof of Theorem \ref{thm:counter_example} can be found in \cite[Sections 4-5]{AlLuSaUr}. 
  \end{remark}

On one hand, our two results Theorems \ref{thm:new_testing_thm_Lp_vector} and \ref{thm:counter_example} suggest that if $(\sigma, \omega)$ are doubling, the $L^p$ norm inequality cannot be characterized without some sort of quadratic condition, whether it be quadratic testing or a quadratic Muckenhoupt condition. On the other hand, if the measure pair is too ``nice,'' say a pair of $A_{\infty}$ weights, then the proof of Theorem \ref{thm:new_testing_thm_Lp} shows the scalar $A_p$ condition characterizes the two-weight norm inequality on $L^p$ and there is no need for a testing condition whatsoever \cite{Neu, HyLa}. Like in \cite{AlLuSaUr}, this provides another operator-theoretic distinction between the class of doubling measures and $A_{\infty}$ weights.
  
  \begin{remark}\label{rmk:no_tails}
    Since $\sigma$ and $\omega$ can be constructed to have doubling constant at most $4+\epsilon$, then as we show in Proposition \ref{prop:small_doubling_tailed_Ap} below, if $\epsilon$ is small, the two-tailed scalar $\mathcal{A}_{p}(\sigma, \omega)$ characteristic is $\lesssim A_p (\sigma, \omega)$, and similarly for $\mathcal{A}_{p'} (\omega, \sigma)$. Furthermore, if we define the triple testing constant
    \[
	    \left ( \mathfrak{T}_{T, q} ^{\operatorname{triple}} (\sigma, \omega) \right )^q \equiv \sup\limits_Q \frac{1}{|Q|_{\sigma}} \int\limits_{3 Q} \left | T_{\sigma} \mathbf{1}_Q \right |^q d \omega \, , 
    \]
    then by Proposition \ref{prop:full_testing_triple}  it follows that $\mathfrak{T}_{T, p} ^{\operatorname{full}} (\sigma, \omega) \lesssim \mathcal{A}_{p} (\sigma, \omega) + \mathfrak{T}_{T,p} ^{\operatorname{triple}} (\sigma, \omega)$. Thus to show Theorem \ref{thm:counter_example}, it suffices to construct $\sigma$ and $\omega$ with doubling constant arbitrarily close to $4$ such that 
    \[
      \mathfrak{T}_{R_2, p} ^{\operatorname{triple}}  (\sigma, \omega) \, , \mathfrak{T}_{R_2 , p'}  ^{\operatorname{triple}} (\omega, \sigma) \leq C_p \text{ and } A_p (\sigma, \omega) \leq C_p \text{ and } \mathfrak{N}_{T, p} (\sigma, \omega) > \Gamma \, ,
    \]
    for all $\Gamma >1$, and where $C_p$ is a constant independent of $\Gamma$, only depending on $p$.
\end{remark}

To prove Theorem \ref{thm:counter_example}, we construct doubling measures $\sigma$ and $\omega$ on $\mathbb{R}^2$ for which the norm inequality for $R_2$ fails along a horizontal configuration, and yet the scalar testing constants for $R_2$ and these measures are \emph{as small as $A_p (\sigma, \omega)$}. To obtain the last point, our measures $\sigma$ and $\omega$  will only vary along the $e_1$ direction.  

More precisely, given  $\Gamma > 1$ arbitrarily large, the proof outline of Theorem \ref{thm:counter_example} is as follows:
  \begin{itemize}
    \item in Section \ref{section:Eric_example_weights}, we begin with positive weights $\sigma, \omega$ on $[0,1) \subset \mathbb{R}$ from \cite[Section 7.1]{SaWi}, which satisfy $A_p ^{\mathcal{D}} (\sigma, \omega) \lesssim 1$, but $A_{p} ^{\ell^2, \operatorname{local}, \mathcal{D}} (\sigma, \omega) > \Gamma$, and with large dyadic doubling ratio depending on $\Gamma$. 
    \item in Section \ref{section:KaTr}, we disarrange $\sigma$ and $\omega$ using the ``small step'' argument of \cite{KaTr} into weights $\tilde{\sigma}$ and $\tilde{\omega}$ with dyadic doubling constant arbitrarily close to $4$ (i.e., that of Lebesgue measure). None of the dyadic Muckenhoupt characteristic estimates we began with will be affected. 
    \item in Section \ref{section:transplantion}, we disarrange $\tilde{\sigma}$ and $\tilde{\omega}$ using Nazarov's remodeling \cite{NaVo, Naz2}\footnote{\cite{Naz2} is merely an online note. See the published paper \cite[Section 4]{NaVo} for full details.}, also called transplantation in \cite{AlLuSaUr2}, into new weights $\hat{\sigma}$ and $\hat{\omega}$ on $[0,1)$, which we then extend periodically to $\mathbb{R}$ and finally tensor with $1$ to get weights $\sigma$ and $\omega$ on $\mathbb{R}^2$. Remodeling affects the weights as follows:
      \begin{itemize}
	\item using Nazarov's transition intervals technique, we obtain doubling measures, i.e., not just dyadically doubling measures, with doubling constant arbitrarily close to $4$. See Section \ref{subsection:transplantation_blue_cubes}.
    \item the Muckenhoupt conditions will essentially be preserved, i.e., $A_p (\sigma, \omega) \lesssim 1$, while the quadratic Muckenhoupt characteristic $A_{p} ^{\ell^2, \operatorname{local}, \mathcal{D}} (\sigma, \omega) > \Gamma$. This will be done in Sections \ref{subsection:remodeling_Muckenhoupt} and \ref{section:Ap_rectangles}.
    \item transplantation localizes the Haar support of the weights to a well-separated set of \emph{horizontal} frequencies. The  testing constants for the \emph{vertical} Riesz transform $R_2$ will be bounded by the $A_p$ condition plus an error term, which can be made small by inductively increasing the separation of the frequencies, see also \cite[p.33-34]{AlLuSaUr2}. We do this in Section \ref{section:testing}.
      \end{itemize}
    \item since $\mathfrak{N}_{R_2, p} (\sigma, \omega) \gtrsim A_{p} ^{\ell^2, \operatorname{local}, \mathcal{D}} (\sigma, \omega) > \Gamma $, then the norm inequality cannot be controlled by scalar $A_p$ and scalar testing, thereby proving Theorem \ref{thm:counter_example}. 
     \end{itemize}
 \begin{remark} Like \cite{AlLuSaUr} did for the  Hilbert transform, a key part of our strategy is to nontrivially estimate the operator norm of $R_2$ from below using the quadratic Muckenhoupt condition.
 \end{remark}

     Finally, we record a straightforward corollary of \cite{KaTr} in Appendix \ref{section:counter_example_maximal_fn}: the scalar $A_p$ condition cannot characterize the two-weight norm inequality of the Hardy-Littlewood maximal function on doubling measures. We have not seen this result anywhere, and have even seen contradictory claims in the literature.

  \section{Preliminaries}

  \subsection{Dependence of implicit constants \texorpdfstring{in the symbol $\lesssim$}{}}\label{subsection:tilde}
If $f_1$ and $f_2$ are nonnegative functions, we write $f_1 \lesssim f_2$ if there is a constant $C$ such that 
\[
  f_1 \leq C f_2 \, ,
\]
for all values of their arguments. We define $\gtrsim$ similarly, and write $f_1 \approx f_2$ if $f_1 \lesssim f_2$ and $f_2 \lesssim f_1$.

For Theorem \ref{thm:new_testing_thm_Lp} and related arguments in Section \ref{section:positive_result}, we allow the constants implicit in $\lesssim$ to depend on $p$, the dimension $n$, the doubling constants of the measures, and the Calder\'on-Zygmund data for the operator.

As for Theorem \ref{thm:counter_example}, and Sections \ref{section:KaTr} and beyond, which are dedicated towards its proof, the constant implicit in $\lesssim$ only depends on $p$, since here we fix the dimension $n=2$, the operator $T = R_2$, and the measures will have bounded doubling constant.

\subsection{Duality for the quadratic Muckenhoupt characteristics}\label{subsection:quadratic_Muck_facts}
Given a measure $\mu$ on $\mathbb{R}^n$ and a collection $\mathcal{S}$ of (axis-parallel) cubes, define the $\mathcal{S}$-maximal function
\begin{equation}\label{eq:defn_M_mu}
  M_{\mu} ^{\mathcal{S}} f (x) \equiv \sup\limits_{I \in \mathcal{S} ~:~ x \in I } \frac{1}{|I|_{\mu}} \int\limits_I \left | f(y) \right | d\mu (y) \, ,
\end{equation}
where $|E|_{\mu}$ denotes the measure of a Borel set $E$, and where if $|I|_{\mu} =0$, the average on the right of \eqref{eq:defn_M_mu} is understood to be $0$. If $\mathcal{S}$ is unspecified, we take $\mathcal{S}$ to be the set of all cubes. In particular, $M_{\mu} f$ is the uncentered Hardy-Littlewood maximal function. Denoting the standard dyadic grid on $\mathbb{R}^n$ by 
\[
\mathcal{D} \equiv \{ 2^{-k}(j+ [0, 1)^n)\}_{j \in \mathbb{Z}^n, k \in \mathbb{Z}} \, , 
\]
we also recover the dyadic maximal function
$M_{\mu} ^{\mathcal{D}} f$ is  with respect to $\mu$.

The following lemma is already well-known for doubling measures \cite{GrLiYa}, and the following short proof of the dyadic case was observed by the second author. More precisely, we
were unable to find this statement explicitly in the literature, and while working on another project the second author noticed the case $p\geq2$
follows from the duality argument in \cite{FeSt}, and that the case $1<p\leq2$
then follows from the weak type $\left(  1,1\right)  $ inequality in
\cite[Theorem A.18 on page 251]{CrMaPe}, together with Marcinkiewicz
interpolation for Banach-space-valued functions, see, e.g., \cite[Theorem 1.18
on page 480]{GaRu}. For the convenience of the reader, we repeat the
short arguments suggested by the second author here. In what follows, 
recall a measure $\mu$ on $\mathbb{R}^n$ is doubling if there exists a constant $C$ such that for all cubes $Q \in \mathbb{R}^n$, we have
  \begin{align}\label{eq:doubling_def}
    |2 Q|_{\mu} \leq C |Q|_{\mu} \, ,
\end{align}
where $2Q$ denotes the cube with same center as $Q$ but whose sidelength is double that of $Q$.
The doubling constant $C_{\mu} ^{\operatorname{doub}}$ associated to $\mu$ is the best constant $C$ in the inequality \eqref{eq:doubling_def}.

\begin{lemma}[Fefferman-Stein type inequalities]\label{lem:FS}
  Let $\mu$ be locally finite Borel measure on $\mathbb{R}^n$, and let $p \in (1,\infty)$.
  \begin{enumerate}
    \item \label{lem:FS_dyadic} We have $M_{\mu} ^{\mathcal{D}} : L^p (\ell^2, \mu) \to L^p (\ell^2, \mu)$, i.e., for every sequence of functions $\left \{ f_i \right \}_i$ we have 
      \[
	\left \| \left \{ \sum\limits_{i} \left | M_{\mu}^{\mathcal{D}} f_i \right |^2 \right \}^{\frac{1}{2}} \right \|_{L^p (\mu)} \lesssim  \left \| \left \{ \sum\limits_{i} \left | f_i \right |^2 \right \}^{\frac{1}{2}} \right \|_{L^p (\mu)} \, . 
      \]
    \item \label{lem:FS_doubling} If $\mu$ is doubling, then $M_{\mu} : L^p (\ell^2, \mu) \to L^p (\ell^2, \mu)$, i.e., for every sequence of functions $f_i$ we have 
      \[
	\left \| \left \{ \sum\limits_{i} \left | M_{\mu} f_i \right |^2 \right \}^{\frac{1}{2}} \right \|_{L^p (\mu)} \lesssim  \left \| \left \{ \sum\limits_{i} \left | f_i \right |^2 \right \}^{\frac{1}{2}} \right \|_{L^p (\mu)} \, . 
      \]
        \end{enumerate} 
\end{lemma}
\begin{remark}
  In (\ref{lem:FS_dyadic}), the implicit constant \emph{only} depends on $p$ and the dimension $n$. In (\ref{lem:FS_doubling}), the implicit constant additionally depends on the doubling constant of $\mu$.
\end{remark}
\begin{proof}
	The proof of part (\ref{lem:FS_doubling}) can be found in \cite{GrLiYa}. We now present a proof sketch of (\ref{lem:FS_dyadic}).

	First, for any nonnegative $w \in L^1 _{loc} (\mu)$, we
claim that
\begin{equation}
	\int_{\mathbb{R}}\left\vert M_{\mu}^{\mathcal{D}}f\left(  x\right)
\right\vert ^{q}w\left(  x\right)  d\mu\left(  x\right)  \leq C_{q}%
\int_{\mathbb{R}}\left\vert f\left(  x\right)  \right\vert ^{q}M_{\mu
}^{\mathcal{D}}w\left(  x\right)  d\mu\left(  x\right)
,\ \ \ \ \ 1<q<\infty.\label{FS 1979}%
\end{equation}
Indeed, let $\lambda>0$ and write $\left\{  M_{\mu
}^{\mathcal{D}}f>\lambda\right\}  =\bigcup\limits
_{j=1}^{\infty}I_{j}$ where $I_{j}$ are the maximal dyadic cubes
satisfying $\frac{1}{\left\vert I_{j}\right\vert _{\mu}}\int_{I_{j}}\left\vert
f\right\vert d\mu>\lambda$. Then we have the weak type $\left(  1,1\right)  $
inequality,%
\[
\left\vert \left\{  M_{\mu}^{\mathcal{D}}f>\lambda\right\}  \right\vert
_{w\mu}  =\sum_{j}\left\vert I_{j}\right\vert _{w\mu}=\sum_{j}\left(
\frac{1}{\left\vert I_{j}\right\vert _{\mu}}\left\vert I_{j}\right\vert
_{w\mu}\right)  \left\vert I_{j}\right\vert _{\mu}
\]
\[\leq\sum_{j}\left(  \frac{1}{\left\vert I_{j}\right\vert _{\mu}}\left\vert
I_{j}\right\vert _{w\mu}\right)  \frac{1}{\lambda}\int_{I_{j}}\left\vert
f\right\vert d\mu\leq\frac{1}{\lambda}\int_{\mathbb{R}}\left\vert f\right\vert
\left(  M_{\mu}^{\mathcal{D}}w\right)  d\mu,
\]
as well as the strong type $\left(  \infty,\infty\right)  $ inequality.
Marcinkiewicz interpolation now gives (\ref{FS 1979}).

	From (\ref{FS 1979}) with $q=\frac{p}{2}\geq1$ we have for any $g \in L^{q'} (\mu)$ that,%
\[ 
\int_{\mathbb{R}}\left(  \sum_{i=1}^{\infty}\left\vert M_{\mu}%
^{\mathcal{D}}f_{i}\right\vert ^{2}\right)  gd\mu  \leq \int_{\mathbb{R}}\sum_{i=1}^{\infty}\left\vert f_{i}\right\vert ^{2}M_{\mu
}^{\mathcal{D}}gd\mu\leq  \left(  \int_{\mathbb{R}}\left(  \sum
_{i=1}^{\infty}\left\vert f_{i}\right\vert ^{2}\right)  ^{q}d\mu\right)
^{\frac{1}{q}}\left(  \int_{\mathbb{R}}\left(  M_{\mu}^{\mathcal{D}%
}g\right)  ^{q^{\prime}}d\mu\right)  ^{\frac{1}{q^{\prime}}}
\]
\[\leq C_{q} \left(  \int_{\mathbb{R}}\left(  \sum_{i=1}^{\infty}\left\vert
f_{i}\right\vert ^{2}\right)  ^{\frac{p}{2}}d\mu\right)  ^{\frac{2}{p}}\left(
\int_{\mathbb{R}}\left\vert g\right\vert ^{q^{\prime}}d\mu\right)  ^{\frac
{1}{q^{\prime}}} \,  .
\]
Taking the supremum over all $g \in L^{q'} (\mu)$ yields (\ref{lem:FS_dyadic}) for $p\geq2$.

The weak type $\left(  1,1\right)  $ inequality in \cite[Theorem A.18 on page
251]{CrMaPe} says that%
\[
\left\vert \left\{  \left\vert M_{\mu}^{\mathcal{D}}\mathbf{f}%
\right\vert _{\ell^{2}}>\lambda\right\}  \right\vert _{\mu}\leq\frac
{C}{\lambda}\int\left\vert \mathbf{f}\right\vert _{\ell^{2}}d\mu,
\]
where given a vector $\mathbf{f} = ( f_i )_i$, we define $M_{\mu} ^{\mathcal{D}} \mathbf{f} \equiv ( M_{\mu} ^{\mathcal{D}} f_i )_i$ and $|\mathbf{f}|_{\ell^2} ^2\equiv \sum\limits_i |f_i|^2$. The Marcinkiewicz interpolation theorem in \cite[Theorem 1.18 on page
480]{GaRu} then completes the proof of (\ref{lem:FS_dyadic}).\end{proof}

Given  a countable collection of cubes $\mathcal{S}$, define the $A_{p} ^{\ell^2, \operatorname{local}, \mathcal{S}}(\sigma, \omega)$ characteristic as the best constant $C$ in the inequality
\begin{align}\label{eq:Ap_quadratic_S}
  \left \| \left ( \sum\limits_{J \in \mathcal{S}} a_J ^2  \left (E_{J} \sigma \right )^2 \mathbf{1}_{J} \right )^{\frac{1}{2}} \right \|_{L^p (\omega)} \leq C \left \|\left ( \sum\limits_{J \in \mathcal{S}}  a_J ^2 \mathbf{1}_{J} \right )^{\frac{1}{2}} \right \|_{L^p (\sigma)} \, ,
\end{align}
where given locally finite measure $\mu$ on a cube $Q$, its expectation on $Q$ is defined as
\[
  E_{Q} \mu \equiv \frac{\left | Q \right |_{\mu}}{ \left | Q \right |} \, . 
\]
And given $f \in L^1 _{\operatorname{loc}} (\mu)$, we also define its $\mu$-expectation over $Q$ by 
\[
  E_{Q} ^{\mu} f \equiv \frac{1}{\left | Q \right |_{\mu} } \int\limits_Q f(x) d \mu (x) \, .
\]
If $\mu$ is not specified, then $\mu$ is understood to be Lebesgue measure, i.e., $E_Q f \equiv \frac{1}{|Q|} \int\limits_{Q} f dx$. 

We have the following duality statement for the quadratic Muckenhoupt conditions.
\begin{lemma}[Duality of Quadratic conditions, {\cite[Lemma 3.1]{Vuo}}]\label{lem:duality_quadratic}
  Let $\sigma$ and $\omega$ be locally finite Borel measures, and let $1<p<\infty$. If $\mathcal{S}$ is a collection of cubes for which 
  \[
	  M_{\sigma} ^{\mathcal{S}} : L^p (\ell^2, \sigma) \to L^p (\ell^2, \sigma) \, , \qquad M_{\omega} ^{\mathcal{S}} : L^{p'} (\ell^2, \omega) \to L^{p'} (\ell^2, \omega) \, , 
  \]
  then $A_{p} ^{\ell^2, \operatorname{local}, \mathcal{S}} (\sigma, \omega) \approx A_{p'} ^{\ell^2 \operatorname{local}, \mathcal{S}} (\omega, \sigma)$. In particular, 
  \begin{enumerate}
    \item \label{lem:duality_quadratic_dyadic}$A_{p} ^{\ell^2, \operatorname{local}, \mathcal{D}} (\sigma, \omega) \approx A_{p'} ^{\ell^2 \operatorname{local}, \mathcal{D}} (\omega, \sigma)$.
    \item \label{lem:duality_quadratic_doubling} If $\sigma$ and $\omega$ are doubling measures, then $A_{p} ^{\ell^2, \operatorname{local}} (\sigma, \omega) \approx A_{p'} ^{\ell^2, \operatorname{local}} (\omega, \sigma)$.
  \end{enumerate} 
\end{lemma}
\begin{proof} Conclusions (\ref{lem:duality_quadratic_dyadic}) and (\ref{lem:duality_quadratic_doubling}) follow from the first ``if\dots then\dots'' statement of the lemma and the Fefferman-Stein inequality Lemma \ref{lem:FS}.

 We claim that \begin{equation}\label{eq:dual_quad_1}
  \left \| \left ( \sum\limits_{J \in \mathcal{S}} \left | \left ( E_{J} ^{\omega}  g_J \right ) \mathbf{1}_J (y) \left (E_J \omega \right ) \right |^2 \right )^{\frac{1}{2}} \right \|_{L^{p'} (\sigma)} \lesssim A_{p'} ^{\ell^2, \operatorname{local}, \mathcal{S}}   \left \| \left ( \sum\limits_{J \in \mathcal{S}} \left | g_J  \right |^2 \right )^{\frac{1}{2}} \right \|_{L^{p'} (\omega)} \,  
\end{equation}
for all sequences $\{g_J\}_J$.
Indeed, by the definition of the $A_{p'}^{\ell^2, \operatorname{local}, \mathcal{S}}(\omega, \sigma)$  characteristic in \eqref{eq:Ap_quadratic_S}, followed by the Fefferman-Stein inequality of Lemma \ref{lem:FS} yields 
      \begin{multline*}
\left \| \left ( \sum\limits_{J \in \mathcal{S}} \left | \left ( E_{J} ^{\omega}  g_J \right ) \mathbf{1}_J (y) \left (E_J \omega \right ) \right |^2 \right )^{\frac{1}{2}} \right \|_{L^{p'} (\sigma)} \leq A_{p'}^{\ell^2, \operatorname{local}, \mathcal{S}}(\omega, \sigma)\left \| \left ( \sum\limits_{J \in \mathcal{S}} \left | \left ( E_{J} ^{\omega}  g_J \right ) \mathbf{1}_J (y)  \right |^2 \right )^{\frac{1}{2}} \right \|_{L^{p'} (\omega)} \\ 
\leq A_{p'}^{\ell^2, \operatorname{local}, \mathcal{S}}(\omega, \sigma)\left \| \left ( \sum\limits_{J \in \mathcal{S}} \left | M_{\omega} ^{\mathcal{S}} g_J  \right |^2 \right )^{\frac{1}{2}} \right \|_{L^{p'} (\omega)} 
\lesssim A_{p'}^{\ell^2, \operatorname{local}, \mathcal{S}}(\omega, \sigma)\left \| \left ( \sum\limits_{J \in \mathcal{S}} \left | g_J  \right |^2 \right )^{\frac{1}{2}} \right \|_{L^{p'} (\omega)} \, .
      \end{multline*}

Then \eqref{eq:dual_quad_1} combined with Cauchy-Schwarz in $\sum\limits_{\mathcal{S}}$ followed by H\"older's inequality in $\int$ shows that \[
	\int \sum\limits_{J \in \mathcal{S}} a_J \mathbf{1}_J (y) \left (E_J \omega \right ) \left ( E_{J} ^{\omega}  g_J \right )  d \sigma (y) \lesssim A_{p'} ^{\ell^2, \operatorname{local}, \mathcal{S}}(\omega, \sigma)  \left \| \left ( \sum\limits_{J \in \mathcal{S}} \left |a_J \mathbf{1}_J  \right |^2 \right )^{\frac{1}{2}} \right \|_{L^p (\sigma)} \left \| \left ( \sum\limits_{J \in \mathcal{S}} \left | g_J  \right |^2 \right )^{\frac{1}{2}} \right \|_{L^{p'} (\omega)} \, , 
      \]
across all choices of coefficients $\{a_J\}_{J}$ and sequences of functions $\{g_J\}_{J}$. Or rather, by Fubini, 
\[
     \int \sum\limits_{J \in \mathcal{S}} a_J \mathbf{1}_J (x) \left (E_J \sigma \right ) g_J (x) d \omega \lesssim A_{p'} ^{\ell^2, \operatorname{local}, \mathcal{S}} (\omega, \sigma)  \left \| \left ( \sum\limits_{J \in \mathcal{S}} \left |a_J \mathbf{1}_J  \right |^2 \right )^{\frac{1}{2}} \right \|_{L^p (\sigma)} \left \| \left ( \sum\limits_{J \in \mathcal{S}} \left | g_J  \right |^2 \right )^{\frac{1}{2}} \right \|_{L^{p'} (\omega)} \, . 
      \]
Because this holds across all sequences $\{g_J\}$ of functions, duality then yields \eqref{eq:Ap_quadratic_S} with ``$\leq C$'' replaced by ``$\lesssim A_{p'} ^{\ell^2, \operatorname{local}, \mathcal{S}} ( \omega , \sigma)$.''
Thus  $A_p ^{\ell^2, \operatorname{local}, \mathcal{S}} (\sigma, \omega) \lesssim A_{p'} ^{\ell^2, \operatorname{local}, \mathcal{S}} ( \omega , \sigma)$. The reverse inequality holds by symmetry. 
    \end{proof}

    When $\sigma$ and $\omega$ are doubling, by the Fefferman-Stein inequality Lemma \ref{lem:FS} combined with the bounds 
    \begin{align}\label{eq:FS_key_application_piece}
	    a_J ^2 \mathbf{1}_{J^*} \lesssim \left (  M_{\sigma} a_J \mathbf{1}_{J} \right )^2 \mathbf{1}_{J^*} \, , \qquad a_J ^2 \mathbf{1}_{J} \lesssim \left (  M_{\sigma} a_J \mathbf{1}_{J^*} \right )^2 \mathbf{1}_{J} \, ,
    \end{align} we have 
    \begin{equation}\label{eq:local_equiv_offset_Muck}
        A_{p} ^{\ell^2, \operatorname{local}} (\sigma, \omega) \approx A_{p} ^{\ell^2, \operatorname{offset}} (\sigma, \omega ; C_0)
    \end{equation} for doubling pairs $(\sigma, \omega)$ where $A_{p} ^{\ell^2, \operatorname{offset}} (\sigma, \omega ; C_0)$ is the best constant in the inequality 
\begin{align}\label{forward_Ap_offset}
  \left \| \left ( \sum\limits_{J} a_J ^2  \left (E_{J^*} \sigma \right )^2 \mathbf{1}_{J} \right )^{\frac{1}{2}} \right \|_{L^p (\omega)} \leq  C \left \|\left ( \sum\limits_{J}  a_J ^2 \mathbf{1}_{J^*} \right )^{\frac{1}{2}} \right \|_{L^p (\sigma)} ,
\end{align}
where the inequality varies over sets $\{J\}$ and $\{a_J\}$ as before, and $J^*$ denotes a dyadic cube in any dyadic grid containing $J$, whose sidelenght $\ell(J^*)$ equals $\ell(J)$ and $\operatorname{dist}(J, J^*) < C_0 \ell(J)$ for some positive constant $C_0$;   see also the explanation in \cite[from (1.6) till end of p.6]{SaWi}. While the implicit constant in \eqref{eq:local_equiv_offset_Muck} will depends on $C_0$, in applications one takes $C_0$ depending only on $p$, the dimension $n$ and the Stein elliptic constant for the operator $T$, making the implicit constant in \eqref{eq:local_equiv_offset_Muck} again only depend on $p$, the doubling constants of $\sigma$ and $\omega$, and the Calder\'on-Zygmund data of the operator.

\subsection{Calder\'on-Zygmund operators, and making sense of the norm inequality}\label{subsection:norm}

For $0\leq \lambda <n$ we define a smooth $\lambda $-fractional Calder\'{o}%
n-Zygmund kernel $K^{\lambda }(x,y)$ to be a function $K^{\lambda }:\mathbb{R%
}^{n}\times \mathbb{R}^{n}\rightarrow \mathbb{R}$ satisfying the following
fractional size and smoothness conditions%
\begin{equation}
\left\vert \nabla _{x}^{j}K^{\lambda }\left( x,y\right) \right\vert
+\left\vert \nabla _{y}^{j}K^{\lambda }\left( x,y\right) \right\vert \leq
C_{\lambda ,j}\left\vert x-y\right\vert ^{\lambda -j-n},\ \ \ \ \ 0\leq
j<\infty ,  \label{sizeandsmoothness'}
\end{equation}%
and we denote by $T^{\lambda }$ the associated $\lambda $-fractional
 Calder\'on-Zygmund operator on $\mathbb{R}^{n}$. When $\lambda = 0$, as in the case of Theorem \ref{thm:counter_example}, we will simply write $T$ and $K$.
An operator $T^{\lambda}$ is  Stein-elliptic if there exists a unit vector $\mathbf{u}$ and a constant $c$ such that for all $x\in \mathbb{R}^n$ and $t \in \mathbb{R}$, we have
\[
  \left | K(x, x+t \mathbf{u}) \right | > \frac{c}{|t|^{n- \lambda}} \, . 
\]
The best such constant $c$ is called the Stein-ellipticity constant for $T$.  

As in \cite[see page 314]{SaShUr9}, we introduce a family $\left\{ \eta
_{\delta ,R}^{\lambda }\right\} _{0<\delta <R<\infty }$ of smooth
nonnegative functions on $\left[ 0,\infty \right) $ so that the truncated
kernels $K_{\delta ,R}^{\lambda }\left( x,y\right) =\eta _{\delta
,R}^{\lambda }\left( \left\vert x-y\right\vert \right) K^{\lambda }\left(
x,y\right) $ are bounded with compact support for fixed $x$ or $y$, and
uniformly satisfy (\ref{sizeandsmoothness'}). Then the truncated operators 
\begin{equation*}
T_{\sigma ,\delta ,R}^{\lambda }f\left( x\right) \equiv \int_{\mathbb{R}%
^{n}}K_{\delta ,R}^{\lambda }\left( x,y\right) f\left( y\right) d\sigma
\left( y\right) ,\ \ \ \ \ x\in \mathbb{R}^{n},
\end{equation*}%
are pointwise well-defined when $f$ is bounded with compact support, and we
will refer to the pair $\left( K^{\lambda },\left\{ \eta _{\delta
,R}^{\lambda }\right\} _{0<\delta <R<\infty }\right) $ as a $\lambda $%
-fractional singular integral operator, which we typically denote by $%
T^{\lambda }$, suppressing the dependence on the truncations. For $%
1<p<\infty $, we say that a $\lambda $-fractional Calder\'on-Zygmund operator 
$T^{\lambda }=\left( K^{\lambda },\left\{ \eta _{\delta ,R}^{\lambda
}\right\} _{0<\delta <R<\infty }\right) $ satisfies the norm inequality%
\begin{equation}
\left\Vert T_{\sigma }^{\lambda }f\right\Vert _{L^{p}\left( \omega \right)
}\leq \mathfrak{N}_{T^{\lambda }}\left( \sigma ,\omega \right) \left\Vert
f\right\Vert _{L^{p}\left( \sigma \right) },\ \ \ \ \ f\in L^{p}\left(
\sigma \right) ,  \label{two weight'}
\end{equation}%
where $\mathfrak{N}_{T^{\lambda }}\left( \sigma ,\omega \right) $ denotes
the best constant in (\ref{two weight'}), provided%
\begin{equation*}
\left\Vert T_{\sigma ,\delta ,R}^{\lambda }f\right\Vert _{L^{p}\left( \omega
\right) }\leq \mathfrak{N}_{T^{\lambda }}\left( \sigma ,\omega \right)
\left\Vert f\right\Vert _{L^{p}\left( \sigma \right) },\ \ \ \ \ f\in
L^{p}\left( \sigma \right) ,0<\delta <R<\infty .
\end{equation*}%
If the classical Muckenhoupt characteristic 
\[
  A_{p}^{\lambda } (\sigma, \omega) \equiv \sup\limits_{I} \left ( \frac{\left |I \right |_{\sigma}}{ \left |I \right | ^{1 - \frac{\lambda}{n}}} \right )^{\frac{1}{p'}} \left ( \frac{\left |I \right |_{\omega}}{ \left |I \right | ^{1 - \frac{\lambda}{n}}} \right )^{\frac{1}{p}}
\] is finite, it
can be easily shown that the norm inequality is independent of the choice of
truncations used - see, e.g., \cite{LaSaShUr3} where rough operators are
treated in the case $p=2$, but the proofs can be modified for general $p$.

Given such an operator $T^{\lambda}$, we use the same abuse of notation as in \eqref{two weight'} in defining the testing conditions: the testing constants are understood as the best constants across all of our truncations. We also define the local testing constant over a collection $\mathcal{S}$ of cubes by 
\[
	\left ( \mathfrak{T}_{T^{\lambda},p} ^{\mathcal{S}, \operatorname{local}} (\sigma, \omega) \right )^{p} \equiv \sup\limits_{Q \in \mathcal{S}} \frac{1}{|Q|_{\sigma}} \int\limits_Q \left | T_{\sigma} ^{\lambda} \mathbf{1}_{Q} (x) \right |^p d \omega(x) \, ;
\]
one may define the $\mathcal{S}$ triple-testing constant $\mathfrak{T}_{T^{\lambda},p} ^{\mathcal{S},\operatorname{triple}}(\sigma, \omega)$ and $\mathcal{S}$ full-testing constant $\mathfrak{T}_{T ^{\lambda},p} ^{\mathcal{S}, \operatorname{full} }(\sigma, \omega)$ similarly. As previously, if $\mathcal{S}$ is not specified, our convention is that the supremum is then taken over all cubes in $\mathbb{R}^n$.
And given an operator $T^{\lambda}$, define the $\mathfrak{T}_{T^{\lambda}, p} ^{\ell^2, \operatorname{local},\mathcal{S} } (\sigma, \omega)$  characteristic as the best constant $C$ in the inequality
\begin{align}\label{eq:quad_testing}
  \left \| \left ( \sum\limits_{J \in \mathcal{S}} a_J ^2  \left (T_{\sigma} ^{\lambda} \mathbf{1}_J \right )^2 \mathbf{1}_{J} \right )^{\frac{1}{2}} \right \|_{L^p (\omega)} \leq C \left \|\left ( \sum\limits_{J \in \mathcal{S}}  a_J ^2 \mathbf{1}_{J} \right )^{\frac{1}{2}} \right \|_{L^p (\sigma)} \, .
\end{align}

Our interest is when $T$ equals an individual Riesz transform $R_j$, $j=1, \ldots, n$,  which are defined on $\mathbb{R}^n$ by
\[
  R_j  f(x) = c_n ~ \mathrm{p.v.} \int\limits_{\mathbb{R}^n} \frac{y_j}{|y|^{n+1}} f(x-y) \, dy \, , \qquad \text{ where } c_n = \frac{\Gamma(\frac{n+1}{2})}{\pi^{\frac{n+1}{2}}} \, . 
\]
Note that each $R_j$ is Stein-elliptic with $\mathbf{u} = e_j$.
We also define the vector Riesz transform on $\mathbb{R}^n$ by
\[
\mathbf{R} f \equiv (R_1 f, \ldots, R_n f) \, .
\]

\subsection{Proof of Remark \ref{rmk:no_tails}}\label{subsection:reduce_tails}
The following propositions justify Remark \ref{rmk:no_tails}; their statements and proofs are common knowledge, but we include them for the convenience of the reader.
\begin{proposition}\label{prop:full_testing_triple} Let $1<p<\infty$, let $\sigma$ and $\omega$ be locally finite measures, and let $T$ be a Calder\'on-Zygmund operator. Then 
  \[
    \mathfrak{T}_{T, p} ^{\operatorname{full}} (\sigma, \omega) \lesssim \mathcal{A}_{p} (\sigma, \omega) + \mathfrak{T}_{T,p} ^{\operatorname{triple}} (\sigma, \omega) \, .
  \]
\end{proposition}
\begin{proof}
We begin by writing
\[
  \frac{1}{\left | Q \right |_{\sigma}}\int\limits_{\mathbb{R}^n} \left | T_{\sigma} \mathbf{1}_Q (x) \right |^p d \omega(x) \leq \frac{1}{\left | Q \right |_{\sigma}}\int\limits_{3Q} \left | T_{\sigma} \mathbf{1}_Q (x) \right |^p d \omega (x) +  \frac{1}{\left | Q \right |_{\sigma}}\int\limits_{(3Q)^c} \left | T_{\sigma} \mathbf{1}_Q (x) \right |^p d \omega (x) \, . 
\]
The first term on the right is  $\lesssim \mathfrak{T}_{T,p} ^{\operatorname{triple}} (\sigma, \omega) ^p$. We show the second term is $\lesssim \mathcal{A}_{p} (\sigma, \omega)^p$. We proceed by annular decomposition and write the second term as
\[
  \frac{1}{\left | Q \right |_{\sigma}}\sum\limits_{\ell=1}^{\infty} \int\limits_{\left ( 3^{\ell+1} Q \right )  \setminus \left (  3^{\ell} Q \right )} \left | T_{\sigma} \mathbf{1}_Q (x) \right |^p d \omega (x) \leq \frac{1}{\left | Q \right |_{\sigma}}\sum\limits_{\ell=1}^{\infty} \int\limits_{\left ( 3^{\ell+1} Q \right )  \setminus \left (  3^{\ell} Q \right )}\left ( \int\limits_{Q}  \left | K(x,y) \right | d \sigma (y)\right )^p d \omega (x) \, . 
\]
By the Calder\'on-Zygmund size estimate \eqref{sizeandsmoothness'} on $K$, this last term is  
 \[
 \lesssim \frac{1}{\left | Q \right |_{\sigma}}\sum\limits_{\ell=1}^{\infty} \int\limits_{\left ( 3^{\ell+1} Q \right )  \setminus \left (  3^{\ell} Q \right )}\left ( \int\limits_{Q}  \frac{1}{\left |x-y \right |^n} d \sigma (y)\right )^p d \omega (x) \approx \left |Q \right |_{\sigma} ^{p-1} \sum\limits_{\ell=1}^{\infty}  \frac{ 1 }{\left |3^{\ell} Q \right |^p}  \int\limits_{\left ( 3^{\ell+1} Q \right )  \setminus \left (  3^{\ell} Q \right )} d \omega (x) 
 \]
 \[
  \approx \left |Q \right |_{\sigma} ^{p-1} \sum\limits_{\ell=1}^{\infty}    \int\limits_{\left ( 3^{\ell+1} Q \right )  \setminus \left (  3^{\ell} Q \right )} \frac{1}{\left [\ell(Q)+ \operatorname{dist}(x,Q) \right ]^{np}} d \omega (x) 
  \approx \left (E_Q \sigma \right ) ^{p-1}    \left ( \frac{1}{|Q|} \int\limits_{(3 Q)^c } \frac{|Q|^p}{\left [\ell(Q)+ \operatorname{dist}(x,Q) \right ]^{np}} d \omega (x) \right ) \, , 
  \]
  which is at most $\mathcal{A}_p (\sigma, \omega)^p$.
\end{proof}

The following proposition assumes for simplicity that $\sigma$ and $\omega$ have doubling constant close to Lebesgue measure. However, the result extends to any pair $(\sigma, \omega)$ of doubling measures, or even any pair of reverse doubling measures. See Proposition \ref{prop:reverse_doubling_tailed_Ap} in Appendix \ref{section:rev_doub_tailed_Ap}, which is based on the proof of \cite[Theorem 4]{Gr}.

\begin{proposition}[{\cite[Theorem 4]{Gr}}]\label{prop:small_doubling_tailed_Ap}
  Let $1<p<\infty$. There exists $\epsilon =\epsilon (p) \in \left (0, \frac{1}{2} \right )$ such that, if $\sigma$ and $\omega$ are doubling measures with doubling constant at most $2^{n(1+ \epsilon)}$, then 
  \[
    \mathcal{A}_{p} (\sigma, \omega) \lesssim A_p (\sigma, \omega) \, .
  \]
\end{proposition}
\begin{proof}
  It suffices to show that if $\mu$ has doubling constant at most $2^{n(1+\epsilon)}$ for $\epsilon$ small enough, then for all cubes $Q$ we have
  \[
     \frac{1}{|Q|} \int \frac{|Q|^p}{\left [\ell(Q)+ \operatorname{dist}(x,Q) \right ]^{np}} d \mu (x)  \lesssim E_Q \mu \, .
  \]
  This follows from an annular decomposition and the doubling property. Indeed, the left side splits as
  \begin{align*}
   &\frac{1}{|Q|} \int\limits_{2 Q} \frac{|Q|^p}{\left [\ell(Q)+ \operatorname{dist}(x,Q) \right ]^{np}} d \mu (x) + \frac{1}{|Q|} \sum\limits_{\ell=1}^{\infty} \int\limits_{2^{\ell+1} Q \setminus 2^{\ell} Q} \frac{|Q|^p}{\left [\ell(Q)+ \operatorname{dist}(x,Q) \right ]^{np}} d \mu (x) \\ 
  \approx &E_Q \mu  + \frac{1}{|Q|} \sum\limits_{\ell=1}^{\infty} 2^{-\ell n p} \left | 2^{\ell+1} Q \right |_{\mu} \leq E_{Q} \mu + \frac{1}{|Q|} \sum\limits_{\ell=1}^{\infty} 2^{n \ell (\epsilon + 1 - p) } \left |  Q \right |_{\mu}   \, .
\end{align*}
If $\epsilon \equiv \frac{p-1}{2} $, then the geometric sum converges and this is  $\lesssim E_{Q} \mu $.
\end{proof}

\section{Proof of Theorem \ref{thm:new_testing_thm_Lp_vector}}\label{section:positive_result}
As mentioned in Section \ref{subsection:tilde}, in this section the implicit constant in $\approx$ and $\lesssim$ will only depend on $p$, the doubling constants of $\sigma$ and $\omega$, and the Calder\'on-Zygmund size, smoothness and ellipticity constants of $T$.

Given a cube $K$, define $\mathcal{D} (K)$ to be the set of all dyadic subcubes of $K$. And given $t \geq 0$, define $\mathcal{D}_t (K)$ to be all dyadic subcubes of $K$ of sidelength $2^{-t } \ell (K)$. Recall that each dyadic cube $I \in \mathcal{D}$ has a parent $\pi I = \pi_{\mathcal{D}} I$ in $\mathcal{D}$, and similarly for $\mathcal{D} (K)$.
The quadratic weak boundedness property of Hyt\"onen-Vuorinen for a measure pair $(\sigma, \omega)$ and an operator $T$ is given by the inequality
  \begin{equation} \label{eq:wbp_quadratic}
	  \sum\limits_{J} \left | \int\limits_{\mathbb{R}^n} a_J \left ( T_{\sigma} \mathbf{1}_{J} \right ) (x) b_J \mathbf{1}_{J ^*} (x)  d \omega(x) \right |\leq C \left | \left ( \sum\limits_{J } |a_J \mathbf{1}_{J} |^2 \right )^{\frac{1}{2}} \right |_{L^p (\sigma)} \left | \left ( \sum\limits_{J} |b_J \mathbf{1}_{J ^*} |^2 \right )^{\frac{1}{2}} \right |_{L^{p'} (\omega)} \, , 
  \end{equation}
  which must hold over all countable sequences of cubes $\{J\}$, sequences $\{a_J\}_J$ and $\{b_J\}_J$, and where $J ^*$ denotes a cube adjacent to $J$ of same sidelength, meaning there exists a cube $K$ such that $J$ and $J^*$ are in $\mathcal{D}_1 (K)$. We let $\mathcal{WBP}^{\ell^2} _{T, p} (\sigma, \omega)$ denote the best constant $C$ in the above inequality. The main results of \cite{SaWi} involved the  quantity $\mathcal{WBP} ^{\ell^2} _{T, p} (\sigma, \omega)$.
\begin{theorem}[{\cite[Theorem 1, part (2)]{SaWi}} ]\label{thm:T1_thm_original_vector}
 Suppose that $\sigma, \omega$ are doubling measures on $\mathbb{R}^n$. Then
    \[
      \mathfrak{N}_{\mathbf{R},p} (\sigma, \omega) \approx \mathfrak{T}_{\mathbf{R},p} ^{\ell^2, \operatorname{local}}  (\sigma, \omega) + \mathfrak{T}_{\mathbf{R},p'} ^{\ell^2, \operatorname{local}}  (\omega, \sigma) + A_{p} ^{\ell^2, \operatorname{local}} (\sigma, \omega) + A_{p'} ^{\ell^2, \operatorname{local}} (\omega, \sigma) + \mathcal{WBP}_{\mathbf{R}, p} ^{\ell^2} (\sigma, \omega) \, . 
    \]
\end{theorem}
We show that for doubling measures, the weak boundedness property follows from the  $A_{p} ^{\ell^2, \operatorname{offset}}$ condition.
  \begin{lemma}\label{lem:wbp_doubling}
    Let $\sigma$ and $\omega$ be doubling measures on $\mathbb{R}^n$ and let $T$ be a Calder\'on-Zygmund operator. Then   \[
      \mathcal{WBP}_{T,p}^{\ell^2} (\sigma, \omega) \lesssim A_{p'} ^{\ell^2, \operatorname{offset}} (\omega, \sigma ; 1) \, .
    \]
	  The upper-bound can also be replaced by   $A_p ^{\ell^2, \operatorname{local}} (\sigma, \omega)$ or $A_{p'} ^{\ell^2, \operatorname{local}} (\omega, \sigma)$ \, .
  \end{lemma}

\begin{remark}
    In Lemma \ref{lem:wbp_doubling}, one can also show
    \[
    \mathcal{WBP}_{T, p} ^{\ell^2} (\sigma, \omega) \lesssim A_p ^{\ell^2, \operatorname{offset}} (\sigma, \omega; C_0 ) \, ,
    \]
    where now the implicit constant depends on $p$, the doubling constants of $\sigma$ and $\omega$, and $C_0$. Indeed, it suffices to take Lemma \ref{lem:wbp_doubling} with the local condition on the right side, and then apply \eqref{eq:local_equiv_offset_Muck} with arbitrary $C_0$.
\end{remark}
  
Theorem \ref{thm:new_testing_thm_Lp_vector} then follows immediately from Theorem  \ref{thm:T1_thm_original_vector} and Lemma \ref{lem:wbp_doubling}.

  \begin{proof}[Proof of Lemma \ref{lem:wbp_doubling}]
    To see that the last statement holds, it suffices to apply \eqref{eq:local_equiv_offset_Muck} and Lemma \ref{lem:duality_quadratic}.  

	  We turn to the first statement, where we abbreviate $A_{p'} ^{\ell^2, \operatorname{offset}}(\omega, \sigma; 1)$ by $A_{p'} ^{\ell^2, \operatorname{offset}}(\omega, \sigma)$. It suffices to show that given adjacent cubes $J$ and $J^*$ of equal sidelength, there exist pairwise disjoint cubes $\{J_{k}\}_k$ and $\{J_{ k} ^* \}$, which are subcubes of $J$ and $J ^*$ with $\operatorname{dist} (J_{k} , J_{ k} ^*)\approx \ell(J_{k})$ and $\ell(J_{k}) = \ell(J_{k} ^*)$, such that 
    \begin{align}\label{eq:wbp_doubling_single_1D}
  \int \int \frac{\mathbf{1}_{J} (y) \mathbf{1}_{J^*} (x)}{|x-y|^n} d \sigma (y) d \omega(x) \lesssim \sum\limits_{k =0}^{\infty} \int \mathbf{1}_{J_{ k} ^*} (x) \left ( E_{J_{k} } \sigma \right )   d \omega (x)  \, .
\end{align}
Indeed, given the above decomposition and \eqref{eq:wbp_doubling_single_1D}, we may then write
\begin{align*}
  \sum\limits_J \left |a_J b_J \int\limits_{\mathbb{R}} \left ( T_{\sigma} \mathbf{1}_{J}  \right )(x) \mathbf{1}_{J ^*} (x) d \omega(x) \right | &\lesssim \sum\limits_J \left |a_J \right | \left | b_J \right |  \int \int \frac{\mathbf{1}_{J} (y) \mathbf{1}_{J^*} (x)}{|x-y|^n} d \sigma (y) d \omega(x) \\
  &\lesssim \int \sum\limits_{J, k} |a_J| |b_J|  \mathbf{1}_{J_{ k} ^*} (x) \left ( E_{J_{ k} } \sigma \right )d \omega (x)  \, ,
\end{align*}
which, by applying Cauchy-Schwarz in $\sum\limits_{J,k}$ and then H\"older's inequality in $\int$, is at most 
\begin{align*}
	&\left \| \left \{ \sum\limits_{J, k} |a_J|^2 \left ( E_{J_{ k} } \sigma \right )^2 \mathbf{1}_{J_{ k} ^*} \right \}^{\frac{1}{2}} \right \|_{L^p (\omega )}  \left \| \left \{ \sum\limits_{J, k} |b_J|^2 \mathbf{1}_{J_{ k} ^*}  \right \}^{\frac{1}{2}} \right \|_{L^{p'} (\omega )} \\
	\lesssim &A_{p} ^{\ell^2, \operatorname{offset}} (\sigma, \omega) \left \| \left \{ \sum\limits_{J, k} |a_J|^2 \mathbf{1}_{J_{k}} \right \}^{\frac{1}{2}} \right \|_{L^p (\sigma )}  \left \| \left \{ \sum\limits_{J, k} |b_J|^2 \mathbf{1}_{J_{ k} ^*} \right \}^{\frac{1}{2}} \right \|_{L^{p'} (\omega )} \\
	\leq &A_{p} ^{\ell^2, \operatorname{offset}} (\sigma, \omega) \left \| \left \{ \sum\limits_{J} |a_J|^2 \mathbf{1}_{J} \right \} ^{\frac{1}{2}} \right \|_{L^p (\sigma )}  \left \| \left \{ \sum\limits_{J} |b_J|^2 \mathbf{1}_{J ^*} \right \}^{\frac{1}{2}} \right \|_{L^{p'} (\omega)} \, .
\end{align*}
Thus \eqref{eq:wbp_doubling_single_1D} implies $\mathcal{WBP}_{T, p} ^{\ell^2} (\sigma, \omega) \lesssim A_{p} ^{\ell^2, \operatorname{offset}} (\sigma, \omega)$. 

\begin{figure}[ht]
  \fbox{\includegraphics[width=0.5\linewidth]{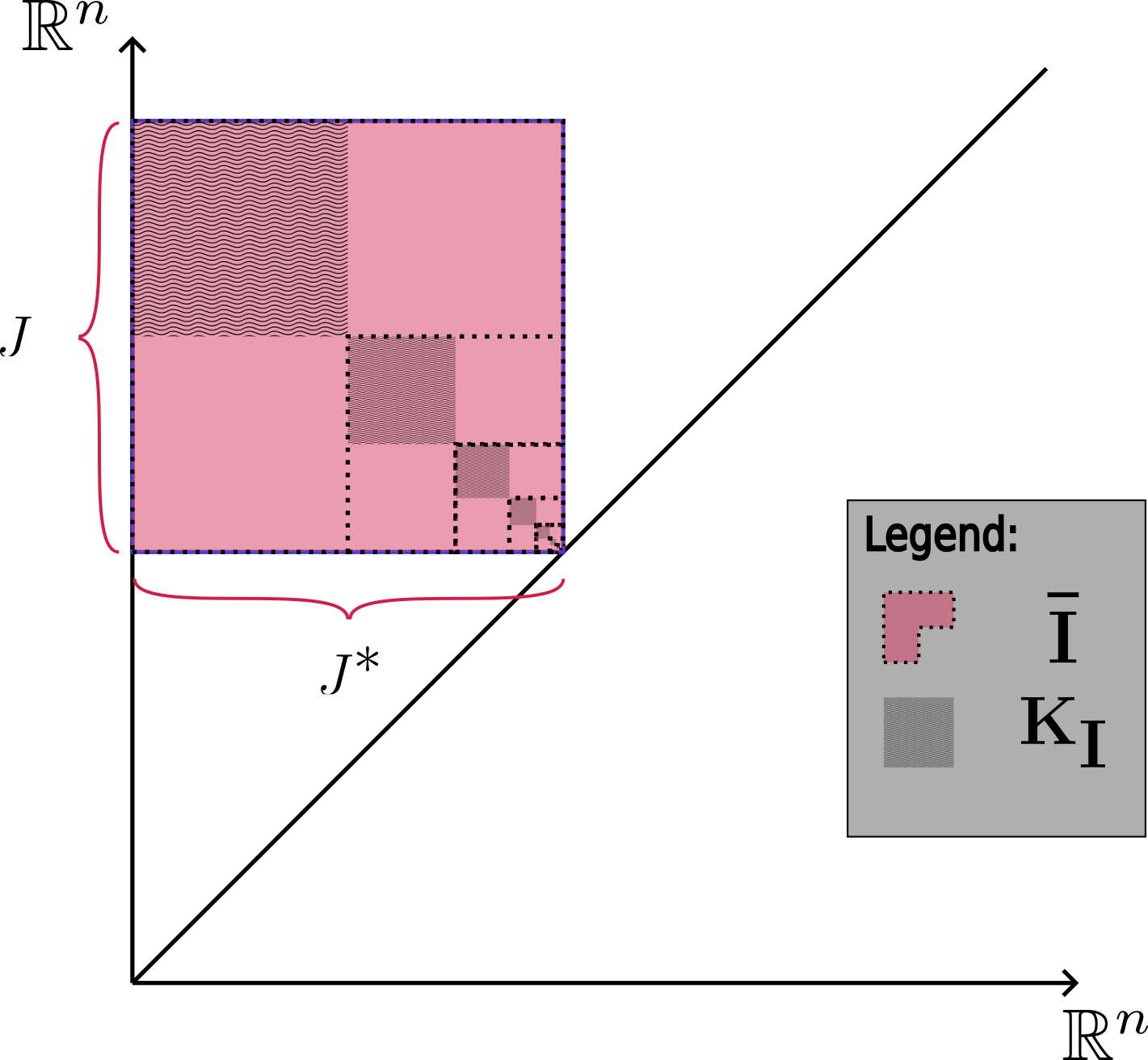}}
	\caption{Picture of $\bar{\mathbf{I}}$ and $\mathbf{K}_{\mathbf{I}}$}
\label{fig:Whitney_cubes}
\end{figure}
	  To show \eqref{eq:wbp_doubling_single_1D}, we will select the cubes $\{J_k\}$ and $\{J_k ^*\}$ from a dyadic Whitney decomposition of 
      \[
      \mathbf{J} \equiv J^* \times J
      \] in the product space $\mathbb{R}^n \times \mathbb{R}^n$ away from the diagonal $D \equiv \{ (x,y) \in \mathbb{R}^{2n}  ~:~ x=y \}$; such a decomposition was used in \cite[Lemma 2.11]{LaSaUr} to prove $A_2 (\sigma, \omega) \lesssim A_2 ^{\operatorname{offset}}(\sigma, \omega)$. More precisely, let $\mathcal{C}$ denote the dyadic subcubes of $\mathbf{J}$ whose boundary intersects the diagonal $D$.
	  Then $\mathcal{C}$ is a grid with natural tree structure and natural parent operator $\pi_{\mathcal{C}}$ inherited from $\mathcal{D} \left ( \mathbf{J} \right )$. We denote the elements of $\mathcal{C}$ by $\mathbf{I} = I^* \times I \subset \mathbb{R}^n \times \mathbb{R}^n$. We next note that $\mathbf{J}$ may be decomposed into the disjoint sets $\bigcup\limits_{\mathbf{I} \in \mathcal{C}} \bar{\mathbf{I}}$, where 
      \[
      \bar{\mathbf{I}} \equiv \mathbf{I} \setminus \bigcup\limits_{\mathbf{K} \in \mathcal{C} ~:~ \mathbf{K} \subsetneq \mathbf{I}} \mathbf{K}
      \]
      is obtained by taking $\mathbf{I} \in \mathcal{C}$ and removing the portions of $\mathbf{I}$ that intersect with other cubes of $\mathcal{C}$ with strictly smaller sidelength. We also note that all points in $\bar{\mathbf{I}}$ are distance $\approx \ell \left (\mathbf{I} \right )$ away from $D$. Thus 
\[
	\int \int \frac{\mathbf{1}_{J} (y) \mathbf{1}_{J^*} (x)}{|x-y|^n} d \sigma (y) d \omega(x) \approx \sum\limits_{\mathbf{I} \in \mathcal{C}} \frac{\left | \bar{\mathbf{I}} \right |_{\omega \otimes \sigma}}{\sqrt{\left |\mathbf{I} \right |}} \, . 
\]
	  Recalling $\mathbf{I} = I ^*  \times I \subset \mathbb{R}^n \times \mathbb{R}^n$, we see that there is a choice of cube 
\[
	\mathbf{K}_{\mathbf{I}} = K_{\mathbf{I}} ^* \times K_{\mathbf{I}}  \, , \text{ satisfying } K_{\mathbf{I}} ^* \in \mathcal{D}_1 \left (I^* \right) \text{ and } K_{\mathbf{I}} \in \mathcal{D}_1 \left (I \right) \, ,
\]
such that the strips
\[
	K_{\mathbf{I}} ^* \times \mathbb{R}^n \, , \qquad \mathbb{R}^n \times K_{\mathbf{I}} \, ,
\]
	  do not intersect any $\mathcal{C}$-descendants of $\mathbf{I}$; see, e.g., Figure \ref{fig:Whitney_cubes}. Then by doubling we have
\[
	\sum\limits_{\mathbf{I} \in \mathcal{C}} \frac{\left | \bar{\mathbf{I}} \right |_{\omega \otimes \sigma}}{\sqrt{\left |\mathbf{I} \right |}} \approx \sum\limits_{\mathbf{I} \in \mathcal{C}} \frac{\left | \mathbf{K}_{\mathbf{I}} \right |_{\omega \otimes \sigma}}{\sqrt{\left |\mathbf{K}_{\mathbf{I}} \right |}} = \sum\limits_{\mathbf{I} \in \mathcal{C}} \int \mathbf{1}_{K_{\mathbf{I}} ^*} (x)  E_{K_{\mathbf{I}}} \left ( \sigma \right ) d \omega(x) \, , 
\]
which now matches the format of the right side of \eqref{eq:wbp_doubling_single_1D}. 

	  Notice that $\{K_{\mathbf{I}} ^*\}$ and $\{K_{\mathbf{I}} \}$ are disjoint subcollections of cubes within $\mathcal{D} \left (J^* \right)$ and $\mathcal{D} \left (J \right )$, respectively. Since $\mathbf{K}_{\mathbf{I}}$ is a cube, then $\ell \left ( K_{\mathbf{I}} ^* \right ) =  \ell \left ( K_{\mathbf{I}}  \right )$, which are both $\approx \operatorname{dist} \left (K_{\mathbf{I}} ^*, K_{\mathbf{I}} \right )$ by    nature of the Whitney decomposition. Re-indexing the collections $\{K_{\mathbf{I}} ^*\}$ and $\{K_{\mathbf{I}} \}$ yields the sought-after collections $\{J_k ^* \}$ and $\{J_k \}$ respectively. 
\end{proof}

 Theorem \ref{thm:T1_thm_original_vector} is actually the case $\lambda = 0$ of the $\lambda$-fractional result in \cite[Theorem 1, part (2)]{SaWi}. Lemma \ref{lem:wbp_doubling} easily extends to the $\lambda$-fractional case, and hence the $\lambda$-fractional analogue of Theorems \ref{thm:new_testing_thm_Lp_vector} and \ref{thm:new_testing_thm_Lp} holds, provided $\lambda \neq 1$. See also \cite{LaWi} for similar limitations, and see \cite{SaShUr9} where the case $\lambda=1$ can be handled by the proof strategies there by differentiating $\ln |x|$ instead of $|x|^0$ in \cite[(54)]{SaShUr9} when $\beta =0$.

  \section{A weight pair on \texorpdfstring{$[0,1)$}{[0,1)} which satisfies \texorpdfstring{$A_p$}{Ap} but fails \texorpdfstring{$A_{p} ^{\ell^2, \operatorname{local}}$}{quadratic Ap}}\label{section:Eric_example_weights}

Given a collection of cubes $\mathcal{S}$ closed under the operation of taking dyadic children, a measure $\mu$ is \emph{doubling on $\mathcal{S}$} if 
  \[
	  \sup\limits_{J \in \mathcal{S}} \sup\limits_{J_1, J_2 ~:~ \pi J_i = J} \frac{|J_1|_{\mu}}{|J_2|_{\mu}} < \infty \, . 
\]
We call the left side the \emph{doubling ratio of $\mu$ over $\mathcal{S}$}.
Our interest is when $\mathcal{S}$ equals  
 $\mathcal{D}(I)$ or $\mathcal{D}$, in which case we will also refer to the doubling ratio over $\mathcal{S}$ as the \emph{dyadic doubling ratio}.

Given a collection of cubes $\mathcal{S}$ and two measures $\sigma$ and $\omega$, define the Muckenhoupt characteristic over $\mathcal{S}$
\[
  A_p ^{\mathcal{S}} (\sigma, \omega) \equiv \sup\limits_{Q \in \mathcal{S}} \left ( E_Q \sigma \right )^{\frac{1}{p'}} \left ( E_Q \omega \right ) ^{\frac{1}{p}} \, , 
\]
We will also need to keep track of the step size of step functions  and whether they adhere to a grid.
\begin{definition}\label{def:dyadic_step_fn}
A function $f: \mathbb{R}^n \to \mathbb{C}$ is a $\mathcal{D}$-dyadic step function with step size $2^{-t}$ if  $f$ is constant on all dyadic cubes in $\mathcal{D}_t$.

Given $a \in \mathbb{R}\setminus \{0\}$ and $b \in \mathbb{R}^n$, let $\mathcal{D}'$ denote the image of the grid $\mathcal{D}$ under the affine transformation $x \mapsto ax+b$. Then $f: \mathbb{R}^n \to \mathbb{C}$ is a $\mathcal{D}'$-dyadic step function of step size $\Delta$ if $f(ax+b)$ is a $\mathcal{D}$-dyadic step function of step size $\Delta a^{-1}$. 
\end{definition}

Using the example from \cite[Section 7.1]{SaWi}, we obtain the following.

  \begin{proposition}[ {\cite[Section 7.1]{SaWi}}] \label{prop:start_ref_weights}
    Let $\Gamma > 1$, and $p \in (1, \infty) \setminus \{2\}$. There exist positive $\mathcal{D}$-dyadic step functions $\sigma, \omega$ on $[0,1) \subset \mathbb{R}$ with finite doubling ratio over $\mathcal{D}([0,1))$, such that 
    \[
      A_{p} ^{\mathcal{D}([0,1))} (\sigma, \omega) \lesssim 1 \text{ but } A_{p} ^{\ell^2, \operatorname{local}, \mathcal{D} ([0,1))} (\sigma, \omega) > \Gamma \, .
    \]
  \end{proposition}
  \begin{proof}
    We prove the case when $1<p<2$; the case $2<p< \infty$ then follows from the duality Lemma \ref{lem:duality_quadratic}.\footnote{To apply Lemma \ref{lem:duality_quadratic}, use the fact that $M_{\mu} ^{\mathcal{D} ([0,1))} f \leq M_{\mu} ^{\mathcal{D}} f$ and the Fefferman-Stein inequality Lemma \ref{lem:FS} to get $M_{\sigma} ^{\mathcal{D} ([0,1))}: L^p ( \ell^2, \sigma ) \to L^p (\ell^2, \sigma )$ and $M_{\omega} ^{\mathcal{D} ([0,1))}: L^{p'} (\ell^2, \omega) \to L^{p'} (\ell^2, \omega)$.} 
    Following the exact computations as in \cite[Section 7.1]{SaWi}, fix $\alpha \in (0,1)$ and consider 
    \begin{align*}
      \omega (x) \equiv \left [ x \left ( \ln \frac{1}{x} \right)^{\alpha} \right ] ^{p-1} \mathbf{1}_{[0, \frac{1}{2})} (x) \, , \qquad \sigma (x) \equiv \frac{1}{ x \left ( \ln \frac{1}{x} \right)^{1+\alpha} } \mathbf{1}_{[0, \frac{1}{2})} (x) \, .
    \end{align*}
    While $\omega$ and $\sigma$ fail to be dyadically doubling on $\mathcal{D}([0,\frac 1 2))$, the computations in  \cite[Section 7.1, middle of p.72 till middle of p.75]{SaWi} show that 
    \begin{equation}\label{eq:scalar_Ap_vs_quad_init}
      A_p ^{\mathcal{D} ([0, \frac 1 2))} (\sigma, \omega) \lesssim 1 \text{ and } A_{p} ^{\ell^2, \operatorname{local}, \mathcal{D}\left ( [0,\frac 1 2) \right )} (\sigma, \omega) = + \infty \, . 
    \end{equation}
  To see the last equality, set $  I_k \equiv [0, 2^{-k})$, $ a_k \equiv k^{\eta}$ for some $\eta> 0$ to be determined, and consider whether
  \begin{align}\label{eq:Ap_quad_demo}
	  \left \| \left ( \sum\limits_{k=1}^{\infty} \left | a_k \frac{|I_k|_{\sigma}}{|I_k|} \right | ^2 \mathbf{1}_{I_k} \right ) ^ \frac{1}{2}\right \|_{L^p (\omega)} \leq C \left \| \left ( \sum\limits_{k=1}^{\infty} \left | a_k  \right | ^2 \mathbf{1}_{I_k} \right )^{\frac{1}{2}}\right \|_{L^p (\sigma)} \, . 
  \end{align}
   Then as shown precisely in \cite[Section 7.1, middle of p.73 till middle of p.75]{SaWi}, the following holds:
  \begin{enumerate}
      \item the right side of \eqref{eq:Ap_quad_demo} is finite if and only if $\epsilon > 0$, where $\epsilon$ is defined by  
      \[
      2 \eta + 1 \equiv (\alpha-\epsilon) \frac{2}{p} \, .
      \]
      \item the left side of \eqref{eq:Ap_quad_demo} is infinite if 
      \[
      \eta p - \alpha  > -1 \, .
      \]
      \item the first two conditions hold simultaneously if and only if 
      \[
      0 < \epsilon < \frac{2-p}{2} \, .
      \]
  \end{enumerate}
  Taking $\epsilon \equiv \frac{2-p}{4}$ and $\{I_k\}$ and $\{a_k\}$ as above,  the left side of \eqref{eq:Ap_quad_demo} is infinite, while the right side is finite, meaning \eqref{eq:Ap_quad_demo} only holds if $C= +\infty$, i.e., \eqref{eq:scalar_Ap_vs_quad_init} holds. And then, for the same choice of $I_k$ and $a_k$ as above, by the monotone convergence theorem, there exists $M$ sufficiently large so that
  \begin{align}\label{eq:Ap_quad_demo_finite}
	  \left \| \left ( \sum\limits_{k=1}^{M} \left | a_k \frac{|I_k|_{\sigma}}{|I_k|} \right | ^2 \mathbf{1}_{I_k} \right )^{\frac{1}{2}}\right \|_{L^p (\omega)}  > \Gamma \left \| \left ( \sum\limits_{k=1}^{M} \left | a_k  \right | ^2 \mathbf{1}_{I_k} \right )^{\frac{1}{2}}\right \|_{L^p (\sigma)} \, .
  \end{align}
      To get dyadically doubling weights, fix this $M$, and now replace $\sigma$ and $\omega$ by the $M$th-level martingales $\mathbb{E}_M \sigma$ and $\mathbb{E}_{M} \omega$, where
      the $t^{\text{th}}$-level martingale approximation of a locally integrable function $f$ is the $\mathcal{D}$-dyadic step function \[
  \mathbb{E}_{t} f (x) \equiv \sum\limits_{I \in \mathcal{D}_t} \left ( E_I f \right) \mathbf{1}_{I}(x) \, , 
\]
and $\mathcal{D}_t$ denotes the dyadic cubes in $\mathcal{D}$ of sidelength $2^{-t}$. 
      This new weight pair will have a finite doubling ratio on $\mathcal{D}([0,\frac{1}{2}))$ (which diverges to infinity as $M$ diverges to infinity) and will satisfy
\[
  A_p ^{\mathcal{D}([0,\frac{1}{2}))} (\sigma, \omega) \lesssim 1 \text{ and } A_{p} ^{\ell^2, \operatorname{local}, \mathcal{D}([0,\frac{1}{2}))} (\sigma, \omega) > \Gamma \, .
\]
Finally, by rescaling, we can adapt these weights to $[0, 1)$.
\end{proof}

\section{Modifying \texorpdfstring{$\sigma$ and $\omega$}{our measures} to get dyadic doubling ratio close to 1}\label{section:KaTr}

A measure $\mu$ is \emph{$\tau$-doubling}  on a collection $\mathcal{S}$ of cubes if its doubling ratio on $\mathcal{S}$ is at most $1 + \tau$. If $\mathcal{S}$ is not specified, we take $\mathcal{S}$ to be the set of all cubes in $\mathbb{R}^n$. Note that a $\tau$-doubling measure $\mu$ has doubling constant approaching doubling constant of Lebesgue measure on $\mathbb{R}^n$ as $\tau$ approaches $0$, i.e.,
\[
  \lim\limits_{\tau \to 0} |C_{\mu} ^{\operatorname{doub}} - 2^n| = 0 \, .
\]
Our goal in this section is the following improvement of Proposition \ref{prop:start_ref_weights}, where the weights are now $\tau$-doubling.
\begin{proposition} \label{prop:weights_small_dyadic_doubling}
  Let $\Gamma \in (1, \infty)$, $\tau \in \left ( 0, \frac{1}{2} \right )$ and let $p \in (1, \infty) \setminus \{2\}$. There exist positive $\mathcal{D}$-dyadic step functions $\tilde{\sigma}, \tilde{\omega}$ on $[0,1) \subseteq \mathbb{R}$ that are $\tau$-doubling on $\mathcal{D}([0,1))$ and satisfy 
    \[
      A_{p} ^{\mathcal{D}([0,1))} (\tilde{\sigma}, \tilde{\omega}) \lesssim 1 \, , \qquad A_{p} ^{\ell^2, \operatorname{local}, \mathcal{D}([0,1))} (\tilde{\sigma}, \tilde{\omega}) > \Gamma \, .
    \]
  \end{proposition}
The proof will follow from the ``small step'' construction of \cite[Section 4]{KaTr}.

\subsection{The construction of \texorpdfstring{\cite{KaTr}}{Kakaroumpas-Treil}}

We recall the ``small step'' construction of Kakaroumpas-Treil, whose key properties are precisely proved in \cite[Section 4]{KaTr}. Therefore, we do not prove all the claims in this subsection. Futhermore, we will only need the statement of Proposition \ref{prop:weights_small_dyadic_doubling} in later subsections.

Given an interval $I$, we first describe a  ``random'' walk on the tree of intervals $\mathcal{D}(I)$. Consider first the standard unbiased random walk $\{\Sigma_k\}_k$ on the integers $\mathbb{Z}$, where $\Sigma_k$ denotes the value of the random walk after $k$ steps. The probability space on which $\Sigma$ is defined is set of infinite sequences taking values in the symbol set $\{-, +\}$. We can naturally identify each sequence $\{m_j\}_{k =j}^{\infty} $ in this space with a tower of dyadic intervals in $\mathcal{D}(I)$, i.e., a sequence $\{I_j\}_{\{j \geq 0\}}$ such that $I_0 = I$ and $\pi I_{j+1} = I_j$, where $\pi$ denotes the natural dyadic parent operator: indeed, start with $I_0 = I$, and recursively define $I_{j+1} = (I_j)_{m_j}$, where for an interval $J$, we let $J_{-}$ and $J_+$ denote the left and right dyadic subchildren of $J$, respectively. Furthermore, each tower has a unique point of $I$ in its intersection, and each point has a unique such tower associated to it. Rewriting the random walk $\Sigma_k$ as a function on $I$, we obtain       
\[
  \Sigma_{I, k} (x) \equiv \sum\limits_{j=0}^{k-1} \sum\limits_{I' \in \mathcal{D}_{j} \left (I \right )} \mathbf{1}_{I_+} - \mathbf{1}_{I_-} \, .
\]
 These three perspectives on a random walk will be helpful as we define stopping intervals.

\begin{figure}[ht]
 \centering
	\fbox{\begin{subfigure}[b]{0.3\textwidth}
         \centering
         \includegraphics[width=\textwidth]{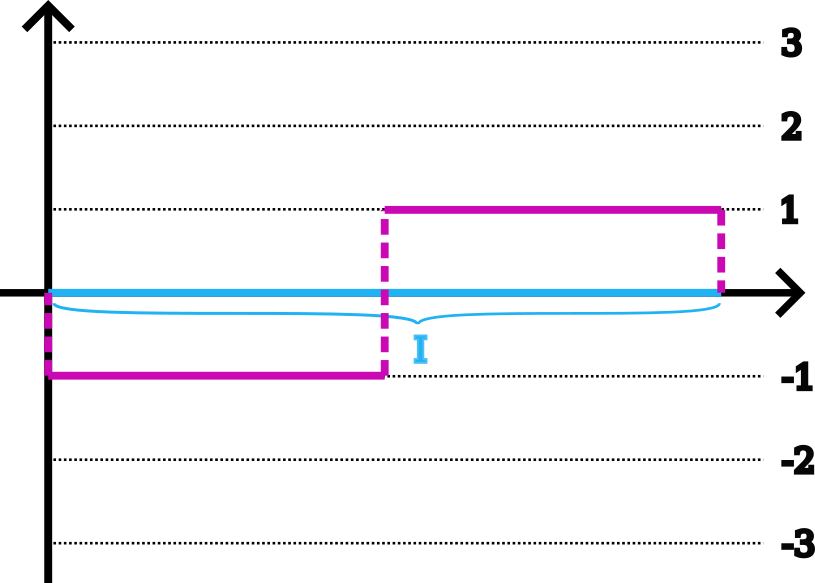}
	     \caption{Graph of $\Sigma_{I,1}$}
         \label{fig:y equals x}
	\end{subfigure}}
     \hfill
	\fbox{\begin{subfigure}[b]{0.3\textwidth}
         \centering
         \includegraphics[width=\textwidth]{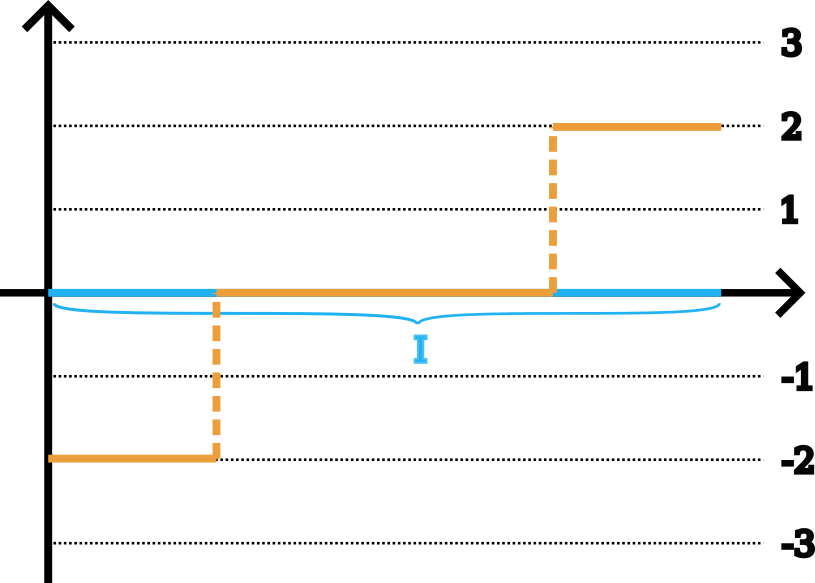}
         \caption{Graph of $\Sigma_{I,2}$}
         \label{fig:three sin x}
	\end{subfigure}}
     \hfill
	\fbox{\begin{subfigure}[b]{0.3\textwidth}
         \centering
         \includegraphics[width=\textwidth]{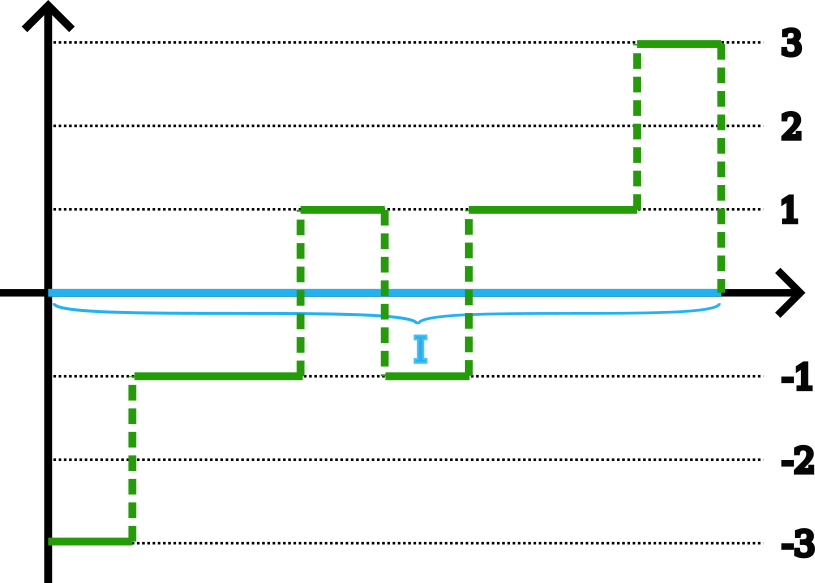}
         \caption{Graph of $\Sigma_{I,3}$}
         \label{fig:five over x}
	\end{subfigure}}
  \caption{Graphs of the truncated random walks $\Sigma_{I,1}$, $\Sigma_{I,2}$ and $\Sigma_{I,3}$}
\label{fig:KaTr_graph_Sigma}
\end{figure}

We are interested in the first times $k$ when the random walk $\Sigma_k$ escapes the interval $[-d, d]$, where $d \geq 1$ is a large integer. In terms of the function $\Sigma_{I, k}$, we can equivalently consider the stopping times 
\[
 k_{I} ^{-} (x) \equiv \inf \left \{ k \geq 1 ~:~ \Sigma_{I,k} (x) \leq -d \right \}\, , \qquad 	k_{I} ^{+} (x) \equiv \inf \left \{ k \geq 1 ~:~ \Sigma_{I,k} (x) \geq d \right \} \, , 
\]
with the understanding that $\inf \emptyset = +\infty$. And then associate to them the sets
\[
	 \Omega ^{-} _I (x) \equiv \{ x \in I ~:~ k_{I} ^{-} (x) <\infty \text{ and } k_{I} ^{-} (x) < \Sigma_{I} ^{+} (x) \} \, , \qquad \Omega ^{+} _I (x) \equiv \{ x \in I ~:~ k_{I} ^{+} (x) <\infty \text{ and } k_{I} ^{+} (x) < \Sigma_{I} ^{-} (x) \} \, .
\]
Define $\stopp_{-} \left (I \right )$ and $\stopp_{+} \left (I \right )$ to be the maximal dyadic intervals $J \in \mathcal{D} \left (I \right )$ such that $J \subseteq \Omega^{-} _I$ and $J \subseteq \Omega^{+} _I$, respectively. These are stopping times from the towers perspective. Define the family of stopping intervals for $I$ by 
\[
\stopp \left (I \right ) \equiv \stopp_{-} \left (I \right ) \cup \stopp_{+} \left (I \right ) \, .
\]
Then $\stopp (I)$ equals the family of all maximal dyadic subintervals $J \in \mathcal{D}\left (I \right)$ such that 
\[
  \left | \sum\limits_{I' \in \mathcal{D} (I) ~:~ J \subsetneq I' } \mathbf{1}_{I '_+} - \mathbf{1}_{I '_-} \right | = d \, , 
\]
and $\stopp_{\pm} (I)$ is the family of all intervals $J$ in $\stopp (I)$ for which the sum in-between the absolute value signs above equals $\pm d$. By the classical theory of random walks, we have $ I \setminus \left ( \bigcup\limits_{J \in \stopp (I)} J \right )$ has Lebesgue measure zero and 
\begin{equation}\label{eq:measure_pres_stop_times_1}
\left |\bigcup\limits_{J \in \mathcal{S}_{+} \left ( I \right )} J \right | = \left |\bigcup\limits_{J \in \mathcal{S}_{-} \left ( I \right )} J \right | = \frac{1}{2} \left | I \right | \, ;
\end{equation}
see \cite[Lemma 4.1]{KaTr}.

The weights $\tilde \sigma$ and $\tilde \omega$ in Proposition \ref{prop:weights_small_dyadic_doubling} will be obtained from the weights $\sigma$ and $\omega$ in Proposition \ref{prop:start_ref_weights} by application of the map $ G \mapsto \tilde G$, described first informally as follows. Given $G \in L^{\infty} ([0,1))$ , we first ``transplant,'' or ``copy-paste,'' the graph of $G_0 \equiv G$ on $[0,1)_{-}$ (or $[0,1)_{+}$) onto each stopping interval in $\stopp_{-} ([0,1))$ (or $\stopp_{+} ([0,1))$) to get a function $G_1$. Then we will obtain a function $G_2$ by looking at $G_1$ on each stopping interval $I$ in $\stopp ([0,1))$, and again ``transplanting'' the graph of $G_1$ on the left (and right) child $I_{-}$ (and $I_+$) of $I$ onto each stopping interval in $\stopp_{-} \left ( I \right )$ (and $\stopp_{+} \left ( I \right )$ to get a function $G_2$. And repeat this process ad infinitum on each subsequent generation of stopping intervals to get a function $\widetilde{G}$.

To define $\tilde{G}$ precisely, we first need to define the set of roofs associated an element of $\mathcal{D}([0,1))$. Define the set of roofs associated to $[0,1)$ by $\roof ([0,1)) \equiv \{[0,1)\}$. Then for each $I \in \mathcal{D} ([0,1))$ for which $\roof (I)$ is defined, inductively define 
\[
\roof(I_{-} ) \equiv \bigcup\limits_{J \in \roof ( I)} \stopp_{-} (J) \, , \qquad \roof(I_{+} ) \equiv \bigcup\limits_{J \in \roof ( I)} \stopp_{+} (J) \, .
\]
 Induction combined with \eqref{eq:measure_pres_stop_times_1} shows that for all $I \in \mathcal{D} ([0,1))$, we have 
\begin{equation}\label{eq:measure_pres_stop_times_2}
|I| = \sum\limits_{J \in \roof(I)} |J| \, .
\end{equation}
In fact, if $I \subseteq K$ are both dyadic subintervals of $[0,1)$, then each roof of $\roof(I)$ contained in $L_1 \in \roof(K)$ is just a rescaling of a roof of $\roof(I)$ contained in another $L_2 \in \roof(K)$, meaning 
\[
\frac{1}{|L_1|} \sum\limits_{J \in \roof(I)~:~ J \subseteq L_1} |J| = \frac{1}{|L_2|} \sum\limits_{J \in \roof(I)~:~ J \subseteq L_2} |J|
\]
and hence for all $L \in \roof(K)$, we have
\begin{equation}\label{eq:measure_pres_stop_times_3}
\frac{|I|}{|K|} = \frac{1}{|L|} \sum\limits_{J \in \roof(I)~:~ J \subseteq L} |J| \, .
\end{equation}
 Define the supervisor of an interval  $J \in \roof(I)$ by $\supr (J) \equiv I$. See Figure \ref{fig:KaTr_stopping_times} for a picture of how supervisors relate to their roof sets: supervisors on the left side have the same color as their roofs on the right side.

And now define
\[
\tilde {G} \equiv G \circ \Phi \, ,
\]
where the measurable function 
\begin{equation}\label{eq:def_Phi}
\Phi: [0, 1) \to [0,1)
\end{equation}
is defined
almost everywhere as follows: almost every $x \in [0,1)$ is the intersection point of a unique sequence $\{J_m\}_{m=0}^{\infty}$ of dyadic intervals, where $J_m \in \roof(I_m)$ for some $I_m \in \mathcal{D}_{m} ([0,1))$. Since the sequence $\{J_m\}_{m=0}^{\infty}$ is nested, then so is the sequence $\{I_m\}_{m=0}^{\infty}$, and so it has at most one point of intersection, which we call $\Phi(x)$ if it exists. Then $\Phi(x)$ is well-defined for almost every $x$: only a countably many sequences of nested intervals $\{I_m\}_{m \geq 0}$ will have no intersection point, and the points $x$ which are the points of intersection of the sequence $\{J_m\}_{m \geq 0}$ where $J_m \equiv \supr (I_m)$ are measure $0$, so the union of all such $x$ is measure $0$.  
Note $\Phi$ is the unique map on $[0,1)$ whose preimage of any dyadic subinterval of $[0,1)$ is the union of all its roofs. More precisely, for all $I \in \mathcal{D} ([0,1))$, 
\begin{equation} 
\sum\limits_{J \in \roof (I) }  \mathbf{1}_J = \mathbf{1}_{I} \circ \Phi \, . \label{eq:indicators_Phi}
\end{equation}
 In Figure \ref{fig:KaTr_stopping_times}, the mappping $\Phi$ sends any interval $I$ on the right to the interval on the left of the same color.
\begin{proposition}\label{prop:measure_preserving_avgs}
The following properties of $\widetilde{G}$ and $\Phi$ hold.
\begin{enumerate}
  \item \label{measure-preserving} $\Phi$ is measure-preserving on $[0,1)$, i.e.,
\[
  \int\limits_{[0,1)} \widetilde{G} \, dx = \int\limits_{[0,1)} G \circ \Phi \, dx= \int\limits_{[0,1)} G \, dx\, .
\]
\item \label{averages_same} If $J \in \roof (I)$, then $E_{I} G = E_{J} \widetilde{G}$. More generally, if there exist $J_2 \in \roof(I)$ and $J_1 \in \roof(I_-) \cup \roof(I_+)$ such that $J_1 \subsetneq J \subseteq J_2$, then $ E_{J} \widetilde{G}$ is a convex combination of $E_{I_{-}} G$ and $E_{I_+} G$.
  \end{enumerate}
\end{proposition}
Property (\ref{averages_same}) means that in Figure \ref{fig:KaTr_stopping_times}, the average of $G$ on any interval on the left equals the average of $\tilde{G}$ on any interval on the right of the same color.   

\begin{proof}
  We check Property (\ref{measure-preserving}). If $G = \mathbf{1}_K$, where $K \subseteq [0,1)$ is a dyadic interval, then by  \eqref{eq:indicators_Phi} and \eqref{eq:measure_pres_stop_times_2},
  \[
 \int\limits_{[0,1)} \tilde{\mathbf{1}_K} \, dx =  \int\limits_{[0,1)} \mathbf{1}_K  \circ \Phi\, dx  = \sum\limits_{J \in \roof(K)} \int\limits_{[0,1)} \mathbf{1}_J\, dx  = \sum\limits_{J \in \roof(K)} |J| = |K| = \int\limits_{[0,1)} \mathbf{1}_K  \, .
  \]
  By linearity,  Property (\ref{measure-preserving}) holds for $\mathcal{D}$-dyadic step functions.

  Next we note that $E \subset [0,1)$ has Lebesgue measure $0$, then so does $\Phi^{-1} (E)$. Indeed, let $\{I\}_j$ be a covering of $E$ by disjoint dyadic subintervals of $[0,1)$. Using the measure-preserving property on the dyadics, we have 
  \[
  |\Phi^{-1} (E)| \leq  \int\limits_{[0,1)} \mathbf{1}_E \circ \Phi \leq \sum\limits_j \int\limits_{[0,1)} \mathbf{1}_{I_j} \circ \Phi = \sum\limits_j |I_j|  \, .
  \]
  If $E$ has measure $0$, then $\{I_j\}_j$ can be chosen so that the right side is arbitrarily small, and so $|\Phi^{-1}(E)| = 0$.

So now given $G \in L^{\infty} ([0,1))$, we consider its martingale approximation $\mathbb{E}_n G$. By Lebesgue differentiation,  
\[
\lim\limits_{n \to \infty} \mathbb{E}_n G = G
\]
pointwise outside a set $E$ of measure $0$. Similarly, 
\[
\lim\limits_{n \to \infty} \mathbb{E}_n G \circ \Phi = G \circ \Phi
\]
outside $\Phi^{-1} (E)$, which also has Lebesgue measure $0$ by the previous claim. Thus by dominated convergence and the fact that Property (\ref{measure-preserving}) holds for $\mathcal{D}$-dyadic step functions, we have
\[
\int\limits_{[0,1)} \tilde{G} \, dx = \int\limits_{[0,1)} G \circ \Phi \, dx = \lim\limits_{n \to \infty} \int\limits_{[0,1)} \mathbb{E}_n G \circ \Phi \, dx = \lim\limits_{n \to \infty} \int\limits_{[0,1)} \mathbb{E}_n G  \, dx = \int\limits_{[0,1)} G \, . 
\]
This proves Property (\ref{measure-preserving}).
  
  By a similar dyadic approximation argument as above, the first statement of  Property (\ref{averages_same}) will follow once we show it for $G = \mathbf{1}_K$ where $K$ is a dyadic interval. By \eqref{eq:indicators_Phi} we have
  \begin{equation}\label{eq:dyadic_expectation_small_step}
  E_J \tilde{\mathbf{1}_K} = \frac{1}{|J|} \int\limits_J \mathbf{1}_K \circ \Phi \, dx = \sum\limits_{L \in \roof(K)} \frac{1}{|J|} \int\limits_{J} \mathbf{1}_L \, dx \, . 
  \end{equation}
  If $I \subseteq K$, then each interval of $\roof(I)$ is contained in an interval of $\roof(K)$, and so the right integral equals 
  \[
  \frac{1}{|J|} \int\limits_J 1 = 1 = E_I \mathbf{1}_K \, , 
  \]
  where again in the last step we used that $I \subseteq K$. If on the other hand $K \subsetneq I$, then each interval of $\roof(K)$ is contained in an interval of $\roof (I)$, and so by \eqref{eq:measure_pres_stop_times_3}, the right side of \eqref{eq:dyadic_expectation_small_step} equals
  \[
  \sum\limits_{L \in \roof(K) ~:~ L \subseteq J} \frac{1}{|J|} \int\limits_{J} \mathbf{1}_L \, dx =\frac{1}{|J|} \sum\limits_{L \in \roof(K) ~:~ L \subseteq J} |L|  = \frac{|K|}{|I|} = E_I \mathbf{1}_K \, . 
  \]

  As for the second statement of Property (\ref{averages_same}), if $K$ is the unique element in $\roof \left (I \right )$ containing $J$, then $J$ decomposes uniquely into intervals from $\stopp_{-} \left (K \right)$ and $\stopp_{+} \left (K \right)$, on which $\tilde{G}$ is constructed to have average equal to $E_{I_{-}} G$ and $E_{I_+} G$ respectively. Thus,
  \[
 E_J \tilde{G} = \frac{1}{|J|} \left ( \sum\limits_{L \in \mathcal{S}_- (K)} \int\limits_{L} \tilde{G} \, dx + \sum\limits_{L \in \mathcal{S}_+ (K)} \int\limits_{L} \tilde{G} \, dx \right ) =  E_{I_{-}} G \left ( \frac{1}{|J|}   \sum\limits_{L \in \mathcal{S}_- (K)} |L| \right ) +  E_{I_{+}} G \left ( \frac{1}{|J|}\sum\limits_{L \in \mathcal{S}_+ (K)} |L| \right ) \, , 
  \]
  which is a convex combination of $E_{I_{-}} G$ and $E_{I_{+}} G$. 
\end{proof}

\begin{figure}[ht]
  \fbox{\includegraphics[width=\linewidth]{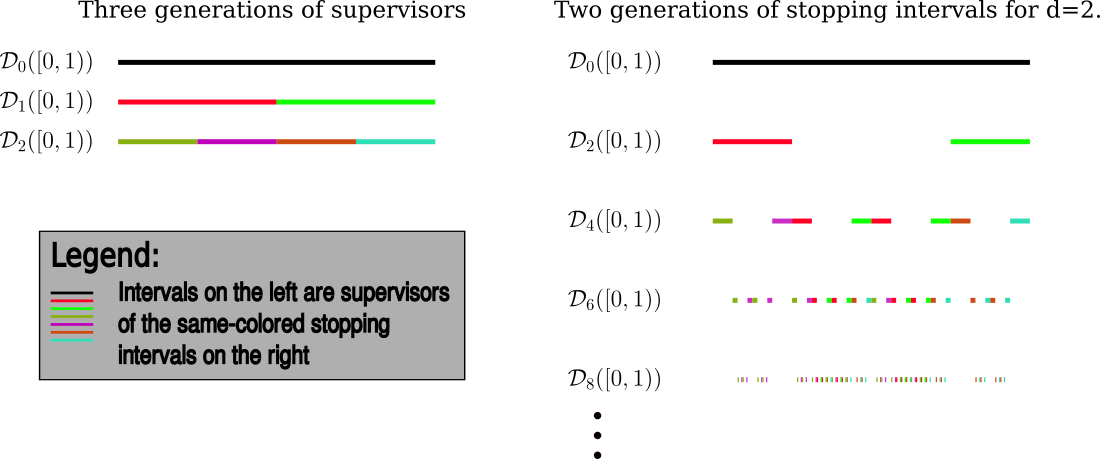}}
  \caption{Two generations of stopping intervals when $d=2$}
\label{fig:KaTr_stopping_times}
\end{figure}

  This rearrangment lets us bring the doubling ratio close to $1$. 
  \begin{lemma}[{\cite[Lemma 4.4]{KaTr}}]\label{lem:large_doubling_to_small}
  Let $\epsilon \in (0, \frac 1 2)$. If $w \in L^{\infty} ([0,1))$ is a positive function with finite dyadic doubling ratio over $\mathcal{D}([0,1))$, then there exists $d \geq 1$ such that  $\widetilde{w}$ is $\epsilon$-doubling on $\mathcal{D}([0,1))$.
  \end{lemma}
  \begin{proof}[Proof sketch using Figure \ref{fig:KaTr_stopping_times}] Let $w$ be a weight on $[0,1)$ and consider $\widetilde{w}$.  Using Figure \ref{fig:KaTr_stopping_times} as a reference, $w$ will have the same average on a dyadic interval $I$ on the left as the average of $\widetilde{w}$ over any interval $J$ on the right side with same color as $I$. Let $J$ be, say, a red interval on the right side of Figure \ref{fig:KaTr_stopping_times}, so that it has supervisor the red interval $\left [ 0, \frac{1}{2} \right)$ on the left. Then the random walk counter $\Sigma$ equals $-d$ when on $J$. If $J'$ is the dyadic sibling of $J$, then the random walk counter $\Sigma$ equals $-d+2$ on $J'$: indeed, just look at their common dyadic parent $K$, on which $\Sigma$ must equal $-d+1$. If $d\gg 1$, then $\Sigma$ on $J'$ is very close to $-d$, and so the random walk proceeding from $J'$ will hit $-d$ very often. Hence $J'$, when decomposed into red and green stopping intervals, will be mostly red, with small smudges of green. This means $E_{K_{-}} \tilde{w}$ will be very close to $E_{K_{+}} \tilde{w}$, meaning we get dyadic $\epsilon$-doubling. 
\end{proof}
  \subsection{Effects of the small-step construction on Muckenhoupt conditions} As a consequence of Proposition \ref{prop:measure_preserving_avgs} part (\ref{averages_same}), the small-step construction of Kakaroumpas-Treil also preserves the dyadic $A_p$ condition. 
  \begin{lemma}[{\cite[Lemma 4.5]{KaTr}}]\label{lem:Ap_preserved_KaTr}
    If $\sigma, \omega \in L^{\infty} ([0,1))$ are positive, then  
    \[
    A_p ^{\mathcal{D}([0,1))} (\widetilde{\sigma}, \widetilde{\omega}) \lesssim A_p ^{\mathcal{D} ([0,1))} (\sigma, \omega) \, .
    \]
  \end{lemma}

A new observation is that $A_p ^{\ell^2, \operatorname{local}, \mathcal{D} ([0,1)) } (\sigma, \omega)$ can only grow under the small step construction of \cite{KaTr}.
\begin{proposition}\label{lem:quadratic_Ap_swells_KaTr}
  If $\sigma, \omega$ are positive functions in  $L^{\infty} ([0,1))$, then  
  \[
    A_p ^{\ell^2, \operatorname{local}, \mathcal{D} ([0,1)) } (\sigma, \omega) \leq A_p ^{\ell^2, \operatorname{local}, \mathcal{D} ([0,1)) } (\widetilde{\sigma}, \widetilde{\omega})
  \]
\end{proposition}
\begin{proof}
  Let $\{a_I\}$ be a near-extremizing sequence for $A_p ^{\ell^2, \operatorname{local}, \mathcal{D} ([0,1))  } (\sigma, \omega)$, i.e., given $\epsilon \in (0,1)$, let $\{a_I\}_I$ be such that for $\mathcal{S} = \mathcal{D} ([0,1)) $ we have \eqref{eq:Ap_quadratic_S} can only hold with a constant $C > (1-\epsilon) A_{p} ^{\ell^2, \operatorname{local}, \mathcal{D}([0,1))}$. Consider the left side of \eqref{eq:Ap_quadratic_S} for the measure $(\widetilde{\sigma}, \widetilde{\omega})$, with $\mathcal{S} = \mathcal{D}([0,1))$, and with another particular choice of coefficients determined by $\{a_I\}$:
 \[ 
    \left \| \left ( \sum\limits_{I \in \mathcal{D} ([0,1))} \sum\limits_{J \in \roof (I) } \left | a_I \right |^2 \left |  E_J \widetilde{\sigma} \right |^2 \mathbf{1}_J \right ) ^{\frac{1}{2}} \right \|_{L^p (\widetilde{\omega})}
  \]
  By Proposition \ref{prop:measure_preserving_avgs} part (\ref{averages_same}), and then \eqref{eq:indicators_Phi} and the fact that $\Phi$ is measure preserving, this equals 
  \[
    \left \| \left ( \sum\limits_{I \in \mathcal{D} ([0,1))} \left | a_I E_I \sigma \right |^2 \sum\limits_{J \in \roof (I) }  \mathbf{1}_J \right ) ^{\frac{1}{2}} \right \|_{L^p (\widetilde{\omega})} =
  \left \| \left ( \sum\limits_{I \in \mathcal{D} ([0,1))} \left | a_I E_I \sigma \right |^2 \mathbf{1}_I \right ) ^{\frac{1}{2}} \right \|_{L^p (\omega)} \, . 
  \]
  Thus for our particular choice of coefficients, the left-hand side of \eqref{eq:Ap_quadratic_S} for $(\widetilde{\sigma}, \widetilde{\omega})$ reduces to the left-hand side of \eqref{eq:Ap_quadratic_S} for $(\sigma, \omega)$ for near-extremizing coefficients $\{a_I\}$. Similarly for the right side of \eqref{eq:Ap_quadratic_S}. Putting this all together yields \[
    (1-\epsilon) A_p ^{\ell^2, \operatorname{local}, \mathcal{D} ([0,1)) } (\sigma, \omega) \leq A_p ^{\ell^2, \operatorname{local}, \mathcal{D} ([0,1)) } (\widetilde{\sigma}, \widetilde{\omega}) \, .
  \]
  Letting $\epsilon \to 0$ yields the inequality sought.
\end{proof}

  \begin{proof}[Proof of Proposition \ref{prop:weights_small_dyadic_doubling}]
    The weights $\sigma, \omega$ of Proposition \ref{prop:start_ref_weights}, satisfy all the requirements of the proposition except for being $\tau$-doubling on $\mathcal{D}([0,1))$. Replace $\sigma, \omega$ by $\widetilde{\sigma}, \widetilde{\omega}$. Then they are $\tau$-doubling on $\mathcal{D}([0,1))$ by Lemma \ref{lem:large_doubling_to_small} if $d$ is chosen sufficiently large. By Lemma \ref{lem:Ap_preserved_KaTr}, $A_p ^{\mathcal{D}([0,1))} (\tilde{\sigma}, \tilde{\omega}) \lesssim 1$, and by Proposition \ref{lem:quadratic_Ap_swells_KaTr} we have $A_{p} ^{\ell^2, \operatorname{local}, \mathcal{D}([0,1))} (\tilde{\sigma}, \tilde{\omega}) > \Gamma$. Finally, we can replace $\widetilde{\sigma}$ and $\widetilde{\omega}$ by their martingale approximations $\mathbb{E}_M \widetilde{\sigma}$ and $\mathbb{E}_M \widetilde{\omega}$ for some $M$ sufficiently large to get that the weights are $\mathcal{D}$ dyadic step functions. 
  \end{proof}

	\section{A remodeling/supervisor argument to well-localize Haar support}\label{section:transplantion}
	 In this section, we explain Nazarov's remodeling \cite{NaVo, Naz2}\footnote{Again as in the introduction, \cite{Naz2} is unpublished. We refer the reader to \cite[Section 4]{NaVo} for details,}, which is a map on $L^{\infty}([0,1))$ denoted as 
    \begin{equation}\label{eq:remdel_mention_1}
    \tilde{G} \mapsto \hat{G} \, . 
    \end{equation}
    We will apply this map to the weights $\tilde{\sigma}$ and $\tilde{\omega}$ in Proposition \ref{prop:weights_small_dyadic_doubling} to obtain weights $\hat{\sigma}$ and $\hat{\omega}$ on $\mathbb{R}$ which (1) are continuously doubling, (2) still satisfy the same Muckenhoupt estimates, and (3) have Haar support on a set of sparse frequencies. This section focuses on the first two properties, while the last property is used in Section \ref{section:testing} to estimate the testing constants for a Riesz transform on the tensored measures. In \cite[Section 5.2]{KaTr}, the authors motivate a similar type of remodeling and trace  Nazarov's argument back to \cite{Bo}.

	\subsection{A simpler remodeling which fails continuous doubling}\label{subsection:transplantation_first_try}

We define here a map
\[
\tilde{G} \mapsto \overline{G} \, . 
\]
While this map is conceptually simpler than remodeling, it has the caveat that it fails to produce continuously doubling measures. But a minor tweak yields the map \eqref{eq:remdel_mention_1} and  continuous doubling in  Subsection \ref{subsection:transplantation_blue_cubes}. 
    
	Set $k_0 \equiv 0$ and let $\{k_{\ell}\}_{\ell=1}^{\infty}$ be a sequence of positive integers which will be chosen inductively later on. Define the grids 
    \[
      \mathcal{K} \equiv \bigcup\limits_{t=0}^{\infty} \mathcal{K}_t \, , \qquad \mathcal{K} (I) \equiv \bigcup\limits_{t=0}^{\infty} \mathcal{K}_t (I)  \, ,
    \]
    where
    \[
    \mathcal{K}_t \equiv \mathcal{D}_{k_0 + k_1 + \ldots + k_t} , \, \qquad \mathcal{K}_t (I ) \equiv \mathcal{D}_{k_0 + k_1 + \ldots + k_t} (I) \, .   \] 
    Each grid has a natural parent operator, denoted by $\pi_{\mathcal{K} (I)}$ and $\pi_{\mathcal{K}}$, each of  which takes an interval in the grid (except at the top level) and returns the smallest interval within the grid containing but not equal to it. The Haar support of  $\overline{G}$ will be confined to these grids.

    Given an interval $I$ and a positive integer $t$, define the $t$\textsuperscript{th} level stopping intervals 
    \[
    \stopp ^{t} (I) \equiv \mathcal{D}_{k_t} (I) \, .
    \]
    We partition  $\stopp ^t (I)$ into the set $\stopp ^t _{-} (I)$ of intervals which are the left children of their $\mathcal{D}(I)$-dyadic parent and the set $\stopp ^t _{+} (I)$ of intervals in $\stopp (I)$ which are the right children of their dyadic parent. Thus going from left to right, the intervals in $\stopp ^{t}  (I)$ alternate between membership in  $\stopp ^t _{-} (I)$ and $\stopp ^t _{+} (I)$.

Given $\tilde{G} \in L^{\infty} ([0,1))$, we will define the function $\overline{G}$ on $[0,1)$ by first ``transplanting,'' or ``copy-pasting,'' the graph of $G_0 \equiv \tilde{G}$ on $[0,1)_{-}$ (or $[0,1)_{+}$)  onto each stopping interval in $\stopp_{-} ^1 ([0,1))$ (or $\stopp_{+} ^1 ([0,1))$) to get a function $G_1$. Then we will obtain a function $G_2$ by looking at $G_1$ on each stopping interval $I$ in $\stopp ^1 ([0,1))$, and again transplanting the graph of $G_1$ on the left (and right) child $I_{-}$ (and $I_+$) of $I$ onto each stopping interval in $\stopp_{-} ^2 \left ( I \right )$ (and $\stopp_{+} ^2 \left ( I \right )$ to get a function $G_2$. And from $G_2$, we look at each stopping interval $I$ in $\stopp ^2 ([0,1))$, and again transplant the graph of $G_2$ on $I_{\pm}$ onto each stopping interval in $\stopp_{\pm} ^2 \left ( I \right )$ to get a function $G_3$. And repeat this process ad infinitum on each subsequent generation of stopping intervals to get a function $\overline{G}$.

More precisely, we again define  roofs associated to elements of $\mathcal{D}(([0,1))$. First set $\roof ([0,1)) \equiv \{[0,1)\}$. Then for each $I \in \mathcal{D}_t ([0,1))$ for which $\roof (I)$ is defined, inductively define $\roof(I_{\pm} ) \equiv \bigcup\limits_{J \in \roof ( I)} \stopp_{\pm} ^{t+1} (J)$.  And given $J \in \mathcal{R}(I)$, define the \emph{supervisor} of $J$ to be $\supr (J) \equiv I$. The parent and supervisor  operations  $\pi$ and $\supr$ commute, i.e., 
\[
\supr\left (\pi_{\mathcal{K}} J \right ) = \pi_{\mathcal{D}} \supr\left ( J \right ) \, .
\]
See Figure \ref{fig:remodeling} for a picture of how supervisors relate to their roof sets:  supervisors on the left  have the same color as all their roofs on the right side.

We now define
\[
\overline{G} \equiv \tilde{G} \circ \Phi \, ,
\]
where the measurable function 
\[
\Phi: [0, 1) \to [0,1)
\]
is defined
almost everywhere as follows: almost every $x \in [0,1)$ is the intersection point of a unique sequence $\{J_m\}_{m=0}^{\infty}$ of dyadic intervals, where $J_m$ belongs to $\roof(I_m)$ for some $I_m \in \mathcal{D}_{m} ([0,1))$. Since the sequence $\{J_m\}_{m=0}^{\infty}$ is nested, then so is the $\{I_m\}_{m=0}^{\infty}$, and so for almost every point $x$ has a unique intersection point, which we call $\Phi(x)$. Note $\Phi$ is the unique map whose preimage of any dyadic subinterval of $[0,1)$ is the union of all its roofs, i.e., for all $I \in \mathcal{D}_t([0,1))$,
\begin{equation} 
 \mathbf{1}_{I} \circ \Phi= \sum\limits_{J \in \roof (I) }  \mathbf{1}_J =  \sum\limits_{J \in \mathcal{K}_t ([0,1)) ~:~ \supr(J) = I }  \mathbf{1}_J \, . \label{eq:indicators_Phi_2}
\end{equation}
 Pictorially, any interval of a given color on the left of Figure \ref{fig:remodeling} has preimage the union of all the intervals of the same color on the right of Figure \ref{fig:remodeling}. 
The operator $G \mapsto \bar{G}$ is known as ``remodeling'' \cite{KaTr, Naz2}. It was also referred to as ``transplantation'' in \cite{AlLuSaUr}, because this map ``transplants'' the averages of the original function $\tilde{G}$ onto new intervals; see, e.g., Figure \ref{fig:remodeling}. 
\begin{figure}[ht]
  \fbox{\includegraphics[width=\linewidth]{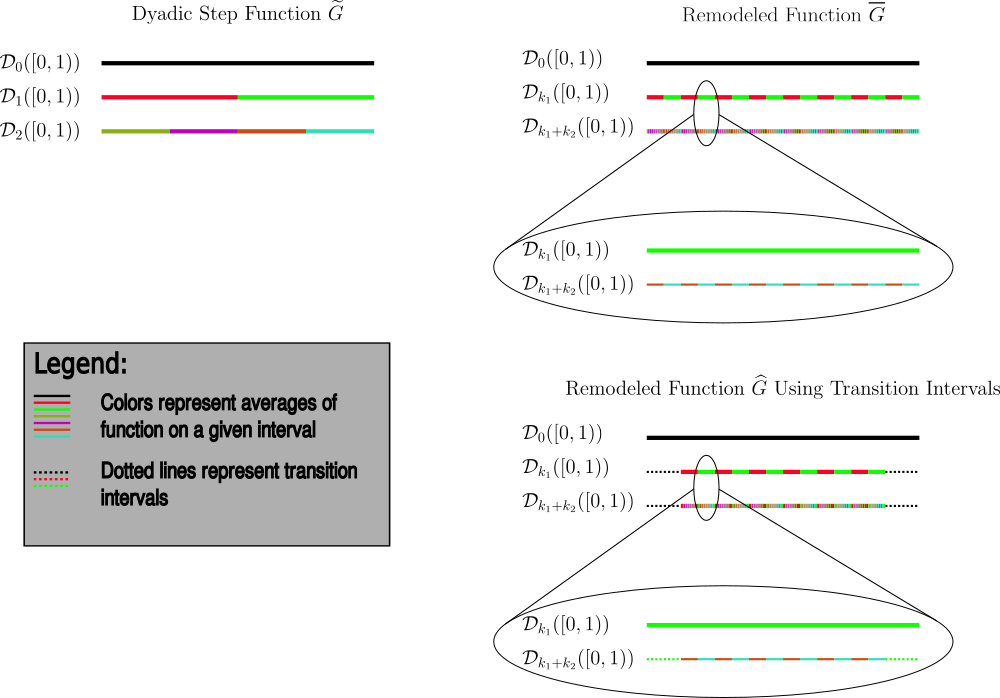}}
  \caption{Remodeling of a dyadic step function with step size $\frac{1}{4}$}
\label{fig:remodeling}
\end{figure}

\begin{proposition}\label{prop:measure_preserving_avgs_2}
The following properties of $\overline{G}$ and $\Phi$ hold.
\begin{enumerate}
  \item \label{measure-preserving_2} $\Phi$ is measure-preserving on $[0,1)$, i.e.,
\[
  \int\limits_{[0,1)} \overline{G} \, dx = \int\limits_{[0,1)} \tilde{G} \circ \Phi \, dx= \int\limits_{[0,1)} \tilde{G} \, dx\, .
\]
\item \label{averages_same_2} If $J \in \mathcal{D}([0,1))$ and $K$ is the smallest interval in $\mathcal{K}$ containing $J$, then
\begin{equation}\label{eq:avg_remodeling_JK}
E_{J} \overline{G} = E_{K} \overline{G} = E_{\supr(K)} \widetilde{G} \, . 
\end{equation}
  \end{enumerate}
\end{proposition}
\begin{proof}
The proof of Property \eqref{measure-preserving_2} and the second equality in \eqref{eq:avg_remodeling_JK} is the same as the proof of Proposition \ref{prop:measure_preserving_avgs}. As for the first equality in \eqref{eq:avg_remodeling_JK}, 
it follows by noting that if $K \in \mathcal{K}_t$, then half of $J$ is covered by intervals in $\stopp^{t+1} _{-} (K)$, on which $\overline{G}$ has average $E_{\supr(K)_{-}} \tilde{G}$, and the other half is covered by intervals in $\stopp^{t+1} _{+} (K)$, on which $\overline{G}$ has average $E_{\supr(K)_{+}} \tilde{G}$.
\end{proof}

Now given a function $\tilde{G} \in L^{\infty} ([0,1))$ and a nonnegative integer $t$, define the $\mathcal{D}_t ([0,1))$ martingales 
    \begin{align*}
      \tilde{G}_t (x) &\equiv \mathbb{E}_t \tilde{G} (x) = \sum\limits_{I \in \mathcal{D}_t ([0,1))} \left ( E_I \tilde{G} \right ) \mathbf{1}_I \, .
          \end{align*} Then we may write the martingale differences as a linear combination of Haar functions, i.e.,\ \begin{align}\label{eq:normal_discrepency}
  \tilde{G}_{t+1} (x) - \tilde{G}_t (x) =\sum\limits_{I \in \mathcal{D}_t ([0,1))} \langle \tilde{G}, h_I  \rangle h_I  (x) \, , 
\end{align}
where the $L^2 (I)$-normalized Haar function $h_I$ is defined as
\begin{equation}\label{eq:Haar_defn}
  h_I \equiv \frac{1}{\sqrt{\left | I \right | }} \left (- \mathbf{1}_{I_-} + \mathbf{1}_{I_+} \right ) \, .
\end{equation}
Then for each integer $t$, define $\bar{G}_t \equiv \tilde{G}_t \circ \Phi$. By \eqref{eq:indicators_Phi_2}, 
we have
\[
  \bar{G}_t \equiv \sum\limits_{I \in \mathcal{D}_t ([0,1))} \left ( E_{I} \tilde{G} \right ) \sum\limits_{J \in \mathcal{K}_t ([0,1)) ~:~ \supr(J) = I} \mathbf{1}_J   \, . 
\]
And by \eqref{eq:Haar_defn}, \eqref{eq:indicators_Phi_2} and \eqref{eq:normal_discrepency}  the martingale differences are given by
\begin{align}\label{eq:init_discrepency}
  \bar{G}_{t+1} -\bar{G}_t \equiv \sum\limits_{I \in \mathcal{D}_t ([0,1))}  \langle \tilde{G}, h_{I} \rangle \frac{1}{\sqrt{ \left | I \right |}} \sum\limits_{J \in \mathcal{K}_t ([0,1)) ~:~ \supr (J) = I} s_{k_{t+1}} ^{J}  \, ,
\end{align}
where the function $s_{k_{t+1}} ^{J}$ vanishes outside of $J$, is constant on each interval in $\mathcal{D}_{2^{-k_{t+1}}} (J)$, and equals $\pm 1$ on each interval in $\mathcal{S}_{\pm} ^{t+1} (J)$. By telescoping, we have
\begin{equation}\label{eq:barG_defn_s}
  \bar{G}_{\ell} = \left(  E_{[0,1)} \tilde{G} \right)  \mathbf{1}_{[0,1)} +\sum\limits_{t=0}^{\ell-1}%
  \sum\limits_{R \in \mathcal{D}_t ([0,1))} \langle \tilde{G},h_{R}\rangle\frac{1}{\sqrt{\left\vert R \right\vert }} \sum\limits_{Q\in\mathcal{K}_{t}([0,1)) ~:~ \supr (Q) = R} s_{k_{t+1}}^{Q}\, .
\end{equation}
Note that if $\tilde{G}$ is a dyadic step function of step size $2^{-m}$, then $\bar{G} \equiv \bar{G}_{m+\ell}$ for all $\ell \geq 0$, in which case by \eqref{eq:barG_defn_s}, $\bar{G}$ has Haar support on the grid $\mathcal{K}([0,1))$. 

This simplified remodeling will disarrange our weights so that their Haar decompositions will consist of wavelets with frequencies that are far and well-separated from each other, yielding a lot of cancellation in their cross interactions, even when when acted on an operator with transverse singularity. We will use this to estimate the testing constants for $R_2$ for the tensors of our weights in Section \ref{section:testing} following \cite{AlLuSaUr}.

If $\widetilde{G}$ is $\tau$-doubling on $\mathcal{D}([0,1))$, then so is $\bar{G}$. But to get continuous doubling, we will need to modify the remodeling construction as in \cite{NaVo, Naz2} (see also \cite{KaTr} for an alternate ``iterated remodeling'' not needed here).
        
\subsection{Nazarov's trick of using transition cubes to get continuous doubling} \label{subsection:transplantation_blue_cubes}

Define the grid $\widehat{\mathcal{K}}$ from $\mathcal{K}$ inductively as
follows. Firstly, $[0,1)$ is an element of $\widehat{\mathcal{K}}$.  And given
$Q\in\widehat{\mathcal{K}}$, every $\mathcal{K}$-child of $Q$ belongs to  $\widehat{\mathcal{K}}$ if it is not one of the two left-most or two right-most $\mathcal{K}$-children of $Q$. For $Q \in \widehat{\mathcal{K}}$, its two left-most and two right-most $\mathcal{K}$-children are called  
\textit{transition intervals} for $Q$. Define
$\widehat{\mathcal{K}}_{t+1}$ to consist of all intervals 
\[
P \in \widehat{\mathcal{K}} \cap\mathcal{K}_{t+1} \, .
\]
And let $\mathcal{T}_{\ell}$ be the set of transition interval in $\mathcal{K}_{\ell}$.  One may define $\hat{\mathcal{K}}(I)$ and $\mathcal{T}_{\ell} \left (I \right )$ from $\mathcal{K}(I)$ similarly.  Note transition intervals are pairwise disjoint.  And if two adjacent intervals in $\widehat
{\mathcal{K}}(I)$ have equal lengths, then they have the same $\mathcal{K}(I)$-parent. In fact, adjacent transition intervals are close in
tree distance. 

\begin{lemma}\label{lemma:adj_trans}
	Let $T_1$ and $T_2$ be adjacent transition intervals. Then we are in one of following two cases. 
	\begin{enumerate}
		\item The pair $T_1$ and $T_2$ have a common $\mathcal{K}$-ancestor $K$, and $K$ is at most two levels away from $T_1$ and $T_2$ in $\mathcal{K}$, meaning if $T_1 \in \mathcal{K}_{t_1}$, then $K \in \mathcal{K}_{t_1 -1 }$ or $\mathcal{K}_{t_1 -2}$, and similarly for $T_2$.
		\item The pair $T_1$ and $T_2$ have no common $\mathcal{K}$-ancestor, and $T_1 , T_2$ are both elements of $\mathcal{K}_1$.
	\end{enumerate}
\end{lemma}

\begin{proof}
	Let $T_1$ and $T_2$ be two adjacent transition intervals.
We do a case analysis.

	If $T_1$ and $T_2$ are the same size, then $\pi_{\mathcal{K}} T_1$ and $\pi_{\mathcal{K}} T_2$ are either equal, meaning we are in Case (1), or they are adjacent in $\hat{\mathcal{K}}$. In the latter case, both $\pi_{\mathcal{K}} T_1$ and $\pi_{\mathcal{K}} T_2$ 
    are either at the top level of $\mathcal{K}$, which is Case (2) because they cannot have common $\mathcal{K}$-ancestor, or they 
    have common $\mathcal{K}$-parent, which falls into Case (1). 

	Say now $T_2$ is larger than $T_1$. Then $T_2 \in \mathcal{K}$ and $\pi_{\mathcal{K}} T_1 \in \hat{\mathcal{K}}$ must be adjacent intervals of equal size. And again by structure of $\hat{\mathcal{K}}$,  these intervals' $\mathcal{K}$-parents must then be the same, which lands us in Case (2).
\end{proof}

Now given $Q\in\widehat{\mathcal{K}}_{t}([0,1))$, define
\begin{equation}\label{eq:s_k_transition}
  \hat{s}_{k_{t+1}}^{Q}(x)\equiv%
\begin{cases}
s_{k_{t+1}}^{Q}(x) & \text{ if }x\text{ is not
contained in a transition interval for }Q\\
0 & \text{ if }x\text{ is 
	contained in a transition interval for }Q \text{ or is outside of } Q\\

\end{cases}
\,.
\end{equation}
With this new function, we modify the definition \eqref{eq:barG_defn_s} and define
\begin{align}\label{eq:trans_modification}
  \hat{G}_{\ell}  &  \equiv\left(  E_{[0,1)} \tilde{G} \right)  \mathbf{1}_{[0,1)}
	+\sum\limits_{t=0}^{\ell-1}\sum\limits_{R \in \mathcal{D}_t ([0,1))} \langle
\widetilde{G},h_{R}\rangle\frac{1}{\sqrt
	{\left\vert R \right\vert }} \sum\limits_{Q\in\widehat{\mathcal{K}}_{t}([0,1)) ~:~ \supr(Q) = R} \hat{s}_{k_{t+1}}%
^{Q} \, .
\end{align}
 We may also view $\widehat{G}_{\ell}$ as a version of the the previous construction which terminates at the transition intervals instead of proceeding ad infinitum: given $x\in [0,1)$ and $\ell \geq 0$,  define
\[
t_{\ell} (x)\equiv%
\begin{cases}
t\text{ if }x\text{ is contained in a transition interval belonging to
}\mathcal{K}_{t} \text{ for some } t< \ell\\
\ell\text{ otherwise }%
\end{cases}
.
\]
Then pointwise we have
\begin{align*}
  \hat{G}_{\ell} (x)&=\left(  E_{[0,1)} \tilde{G}  \right)  \mathbf{1}_{[0,1)}%
+\sum\limits_{t=0}^{t_{\ell} (x)-1}\sum\limits_{R \in \mathcal{D}_t ([0,1))} \langle
\widetilde{G},h_{R}\rangle\frac{1}{\sqrt
	{\left\vert R \right\vert }} \sum\limits_{Q\in\widehat{\mathcal{K}}_{t}([0,1)) ~:~ \supr(Q) = R} \hat{s}_{k_{t+1}}%
^{Q}(x) \\
&= \bar{G}_0 (x) + \sum\limits_{t=0}^{t_{\ell} (x) - 1} \left \{ \bar{G}_{t+1} (x) - \bar{G}_t (x) \right \} \,  .
\end{align*}
In particular, if $x$ does not belong to any transition interval, then $t_{\ell} (x) = \ell$ for all $\ell \geq 0$ and so 
\[
\lim\limits_{\ell \to \infty} \widehat{G}_{\ell} (x) = \lim\limits_{\ell \to \infty} \overline{G}_{\ell} (x) = \overline{G}(x) \, ,
\]
and if $x$ belongs in a transition interval in $\mathcal{K}_t$, then
\[
\lim\limits_{\ell \to \infty} \widehat{G}_{\ell} (x) = \overline{G_{t}} (x) \, .
\]
Thus we may define
\[
\widehat{G} (x) \equiv \lim\limits_{\ell \to \infty} \widehat{G}_{\ell} (x) 
\]
pointwise everywhere on $[0,1)$.
As in Figure \ref{fig:remodeling}, the function $\hat{G}$ is ``nearly'' a transplantation of $\tilde{G}$. We note that if $\widetilde{G}$ is a $\mathcal{D}$-dyadic step function with step size $2^{-k}$, then $\widehat{G}_{k+\ell} =\widehat{G}_{k}$ for all $\ell \geq 0$.

Supervisors are \textit{a priori} only defined for intervals in  $\mathcal{K} ([0,1))$. We extend the notion of supervisor to any interval in  the grid $\mathcal{K}$, which is the periodic extension of $\mathcal{K} ([0,1))$ to $\mathbb{R}$. We do this by defining
\[
	\supr(J) \equiv \supr(J \operatorname{mod} 1) \, ,
\]
where $J \operatorname{mod} 1$ is the integer translate of $J$ with left endpoint in $[0,1)$, and so $J \operatorname{mod} 1$ is contained in $[0,1)$. Using this extension of supervisor, the formula
\begin{align}\label{eq:trans_modification_periodic}
  \hat{G}_{\ell}  &  \equiv\left(  E_{[0,1)} \tilde{G} \right)  \mathbf{1}_{[0,1)}
	+\sum\limits_{t=0}^{\ell-1}\sum\limits_{R \in \mathcal{D}_t ([0,1))} \langle
\widetilde{G},h_{R}\rangle\frac{1}{\sqrt
	{\left\vert R \right\vert }} \sum\limits_{Q\in\widehat{\mathcal{K}}_{t} ~:~ \supr(Q) = R} \hat{s}_{k_{t+1}}%
^{Q}
\end{align}
extends \eqref{eq:trans_modification}  periodically to $\mathbb{R}$. This also extends $\hat{G}= \lim\limits_{\ell \to \infty} \hat{G}_{\ell}$ periodically to $\mathbb{R}$. Finally, given $I \in \mathcal{D}_t ([0,1))$, define the periodic \textbf{oscillatory} function 
\begin{align} \label{eq:osc_1d_def} 
  \operatorname{Osc}_{k_{t+1}} ^I \equiv \sum\limits_{J \in \mathcal{K}_t ~:~ \supr (J) = I} \hat{s}_{k_{t+1}} ^{J} \, ,
\end{align}
so that by \eqref{eq:trans_modification_periodic} and \eqref{eq:osc_1d_def},  the (modified) remodeled martingale differences are given by
\begin{align}\label{eq:remodeled_martingale_diff_transition}
  \hat{G}_{t+1} -\hat{G}_t = \sum\limits_{ I \in \mathcal{D}_t ([0,1)) }  \langle \tilde{G}, h_{I} \rangle \frac{1}{\sqrt{ \left | I \right |}} \operatorname{Osc}_{k_{t+1}} ^I  \, ;
\end{align}
see Figure \ref{fig:Osc_func} below for a depiction of the function $\operatorname{Osc}_{k_{t+1}} ^I$. We also periodically extend the grid $\widehat{K}$ and transition intervals in the natural way. 
\begin{figure}[ht]
  \fbox{\includegraphics[width=\linewidth]{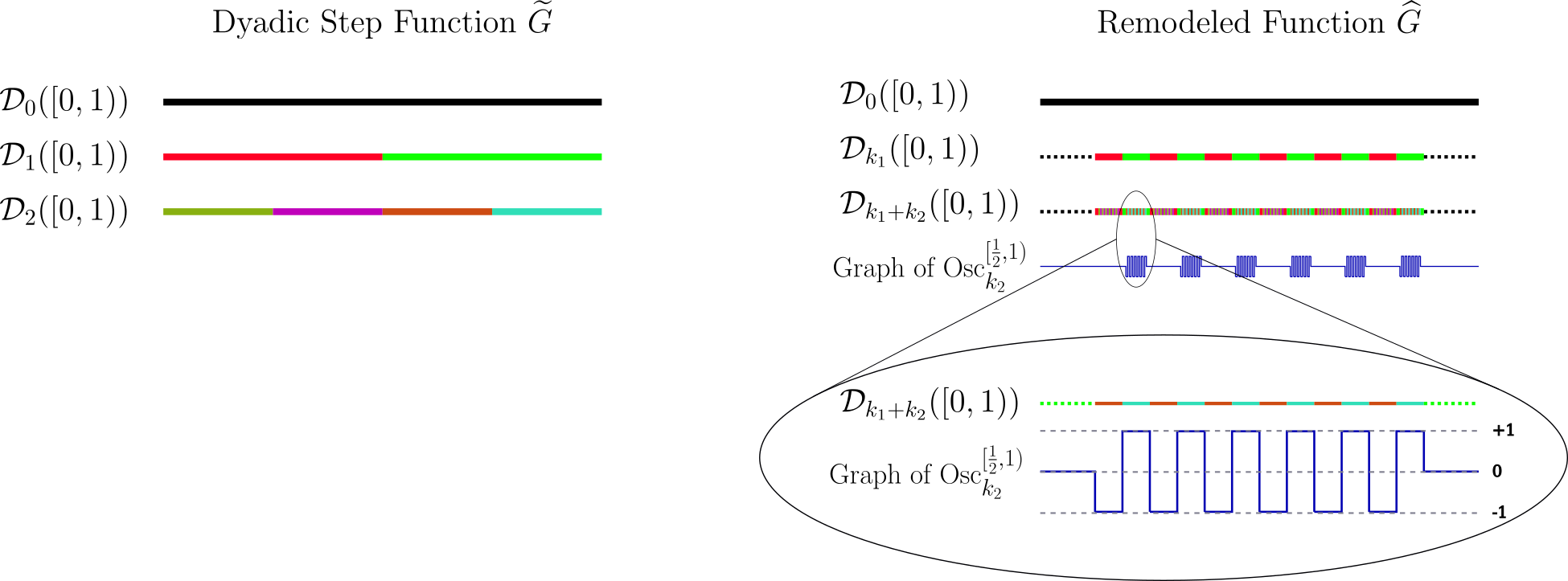}}
  \caption{Graph of $\operatorname{Osc}^{\left [\frac{1}{2},1 \right )}_{k_2}$ relative to stopping and transition intervals.}
\label{fig:Osc_func}
\end{figure}

Remodeling with transition intervals yields continuous doubling, as proved in \cite[Section A.4]{AlLuSaUr} and \cite{Naz2}, claimed in \cite[Section 4]{NaVo}, and proved for an analogous construction in \cite[Section 6.1.2]{KaTr}. The proof below however involves case-chasing: we recommend the reader draw a picture.  

\begin{lemma}[{\cite{NaVo,KaTr,AlLuSaUr, Naz2}}]\label{lem:doubling_remodeling}
  Let $\epsilon \in (0, \frac 1 2)$. There exists $\tau(\epsilon) \in (0, \frac 1 2)$ such that if a positive function  $\tilde{G} \in L^{\infty} ([0,1))$ is $\tau$-doubling on $\mathcal{D}([0,1))$, then $\hat{G}$ is $\epsilon$-doubling.
\end{lemma}
\begin{proof}
    It suffices to show that for every $\epsilon \in (0, \frac 1 2)$, if a positive function  $\tilde{G} \in L^{\infty} ([0,1))$ is $\tau$-doubling on $\mathcal{D}([0,1))$ for $\tau = \frac{\epsilon}{4}$, then for all pairs of adjacent dyadic intervals $S_{-}$ and $S_{+}$ of same length, we have 
   \begin{equation}\label{eq:to_check_doubling}
   (1+\epsilon)^{-1} E_{S_{+}} \widehat{G}\leq E_{S_{-}} \widehat{G}  \leq (1 + \epsilon) E_{S_{+}} \widehat{G} \, . 
   \end{equation}
If $\ell (S_{\pm}) \geq  1$, then
   \[
   E_{S_{\pm}} \widehat{G} = E_{[0,1)} \widehat{G} = E_{[0,1)} \tilde{G} \, , 
   \]
   and so \eqref{eq:to_check_doubling} holds trivially. So assume now $\ell (S_{\pm}) <  1$. 

If both $S_{-}$ and $S_{+}$ belong to $\widehat{\mathcal{K}}_t$ for some $t \geq 1$, then 
\[
\pi_{\mathcal{D}} \supr (S_{-}) = \pi_{\mathcal{D}} \supr (S_{+}) \, .
\]
By $\tau$-doubling of $\tilde{G}$, we then have 
\[
E_{S_{-}} \widehat{G} = E_{\supr(S_{-})} \tilde{G} \leq (1+ \tau) E_{\pi_{\mathcal{D}} \supr(S_{-})} \tilde{G} \leq (1 + \tau)^2 E_{ \supr(S_{+})} \tilde{G} = (1 + \tau)^2 E_{ S_{+}} \widetilde{G} \, , 
\]
and similarly
\[
 E_{ S_{+}} \widetilde{G}\leq (1 + \tau)^2 E_{ S_{-}} \widetilde{G} \, .
\]
If $\tau = \frac{\epsilon}{4}$, then \eqref{eq:to_check_doubling} holds.

If only one of the two intervals $S_{-}$ and $S_{+}$ belongs to $\widehat{\mathcal{K}}_t$, say $S_{-}$, then $S_+$ must be a transition interval. Then $S_{-}$ and $S_+$ must have the same $\mathcal{K}$-parent, and so by $\tau$-doubling of $\widetilde{G}$,
\[
E_{S_{-}} \widehat{G} = E_{\supr(S_{-})} \tilde{G} \leq (1+\tau) E_{\pi_{\mathcal{D}}\supr(S_{-})} \tilde{G} \leq (1+\tau) E_{\supr( \pi_{\mathcal{K}}S_{-})} \tilde{G} = (1+\tau) E_{\supr( \pi_{\mathcal{K}}S_{+})} \tilde{G}  = (1+ \tau) E_{S_+} \widehat{G} \, ,
\]
and similarly
\[
E_{S_+} \widehat{G} \leq (1+\tau) E_{S_{-}} \widehat{G} \, , 
\]
and so \eqref{eq:to_check_doubling} holds if $\tau = \frac \epsilon 4$.

If both $S_{-}$ and $S_+$ belong to $\mathcal{K}_t \setminus \widehat{\mathcal{K}}_t$, then 
 $S_{-}$ and $S_{+}$ are contained within transition intervals $T_{-}$ and $T_{+}$. By Lemma \ref{lemma:adj_trans}, one of two things can happen. If $T_{-}$ and $T_+$ belong to $\mathcal{K}_1$, then by periodicity of $\widehat{G}$,  
 \[
E_{S_{\pm}} \widehat{G} =  E_{T_{\pm}} \widehat{G} = E_{[0,1)} \widehat{G}
 \]
and so \eqref{eq:to_check_doubling} holds. If on the other hand 
 $T_{-}$ and $T_+$ have a common $\mathcal{K}$-ancestor $K$,  it is at most $2$ levels above either one with respect to $\mathcal{K}$, and so
\[
E_{S_{-}} \widehat{G} = E_{\pi_{\mathcal{K}} T_{-}} \widehat{G} = E_{\supr ( \pi_{\mathcal{K}} T_{-})} \tilde{G} \leq (1 + \tau) E_{\supr ( K) } \tilde{G}  \leq (1+ \tau)^2 E_{ \supr ( \pi_{\mathcal{K}} T_{+})} \tilde{G} = E_{\pi_{\mathcal{K} T_{+}}} \widehat{G} = E_{S_{+}} \widehat{G} \, . 
\]
Similarly, we have
\[
E_{S_{+}} \widehat{G} \leq (1 + \tau)^2 E_{S_{-}} \widehat{G} \, , 
\]
and so \eqref{eq:to_check_doubling} holds for $\tau = \frac \epsilon 4$.

Thus \eqref{eq:to_check_doubling} holds whenever $S_{-}$ or $S_+$ belongs to $\mathcal{K}$, as well as whenever $\ell (S_{\pm}) \geq 0$. If neither $S_{-}$ nor $S_+$ belong to $\mathcal{K}$ and both have sidelength less than $1$, then let $t$ be the smallest nonnegative integer for which
\[
\ell (S_{\pm}) \leq 2^{-(k_0 + \ldots +k_t)} \, .
\]
Let $K_{-}$ and $K_+$ denote smallest intervals in $\mathcal{K}_t$ containing $S_{-}$ and $S_+$, respectively. Then either $K_{-} = K_+$, or they are adjacent. Using the fact that
\[
E_{S_{-}} \widehat{G} = E_{K_{-}} \widehat{G} \, , \qquad E_{S_{+}} \widehat{G} = E_{K_{+}} \widehat{G}
\]
and that \eqref{eq:to_check_doubling} holds for $K_{-}$ and $K_+$ then yields \eqref{eq:to_check_doubling} for $S_{-}$ and $S_{+}$. 
\end{proof}

The following says that $\hat{G}_t$ is essentially constant at scales near its step size if $\tilde{G}$ is dyadically doubling. 
\begin{lemma}\label{lem:remodeling_loc_cst}
Let $\tau \in (0, \frac 1 2)$.
	If $Q$ is an interval and $t$ is smallest nonnegative integer such that 
    \begin{equation}\label{eq:size_cond_Q_0}
    2^{-(k_0 + k_{1}+k_{2}+...+k_{t})}\leq\ell\left(  Q\right) \, ,
    \end{equation}
    then there exists an interval $Q^*  \in \bigcup\limits_{\ell=0}^t\mathcal{K}_{\ell}$, such that for all positive $\widetilde{G} \in L^{\infty} ([0,1))$ that are $\tau$-doubling on $\mathcal{D}([0,1))$ and for all $x \in Q$, we have
  \[
    \hat{G}_t (x) \approx E_{\supr (Q^*)} \tilde{G} \, . 
  \]
\end{lemma}

\begin{proof}

	\textbf{Case $t=0$:} if $t=0$, then $\hat{G}_0 = E_{[0,1)} \tilde{G}$ is constant, so setting $Q^* \equiv [0,1)$ suffices.
    
   \textbf{Case $t=1$:} if $t = 1$, then $\hat{G}_t$ takes values in the set 
   \[
   \left \{ E_{[0,1)} \tilde{G}, E_{\left [ 0,\frac{1}{2}\right )} \tilde{G}, E_{\left [ \frac{1}{2},1\right )} \tilde{G} \right \} \, .
   \]All elements of this set are $\approx E_{[0,1)} \tilde{G}$ by dyadic $\tau$-doubling of $\tilde{G}$. So set $Q^* \equiv [0,1)$.

	\textbf{Case $t \geq 2$:} we suggest Figure \ref{fig:locally_constant} as a visual aid for the reasoning for this case, as it illustrates some of the subcases we consider when $t=3$. Here, dotted green intervals belong in $\mathcal{T}_2$ and dotted black intervals belong in $\mathcal{T}_1$.  
	\begin{enumerate}
		\item \textbf{Subcase where $Q$ does not intersect any interval in $\bigcup\limits_{\ell=1}^{t-1} \mathcal{T}_{\ell}$:} then $Q$ is contained in some $Q^* \in \widehat{\mathcal{K}}_{t-2}$. But $\hat{G}_t$ is constant on all $\mathcal{K}$-grandchildren of $Q^*$, taking the values 
        \begin{equation}\label{eq:values_G}
        E_{\supr (Q^*)} \tilde{G} \, , \quad E_{\supr (Q^*)_{\pm} } \tilde{G} \, , \quad E_{\supr (Q^*)_{\pm \pm} } \tilde{G} 
        \end{equation}
        on the transition intervals for $Q^*$, on the transition intervals for children of $Q^*$ within $\hat{\mathcal{K}}_t$, and on the grandchildren of $Q^*$ in $\hat{\mathcal{K}}_t$, respectively. By dyadic doubling of $\tilde{G}$, these values are $\approx E_{\supr (Q^*)} \tilde{G}$.
		\item \textbf{Subcase where $Q$ intersects only one interval in $\bigcup\limits_{\ell=1}^{t-1} \mathcal{T}_{\ell}$:} call this interval $T$.
			\begin{enumerate}
				\item \textbf{Subsubcase where $Q \subseteq T$:} since $\hat{G}_t$ has constant value $E_{ \pi_{\mathcal{K}} T} \tilde{G}$ on $T$, then take $Q^* \equiv \pi_{\mathcal{K}} T$.
				\item \textbf{Subsubcase where $Q \not \subseteq T$:} then $Q$ also intersects some other transition interval $T'$ in $\bigcup\limits_{\ell=1}^{t} \mathcal{T}_{\ell}$ that is adjacent to $T$. By assumption, $Q$ does not intersect any other interval in $\bigcup\limits_{\ell=1}^{t-1} \mathcal{T}_{\ell}$, and so $T' \in \mathcal{T}_{t}$, and  $Q^* \equiv \pi_{\mathcal{K}} T$  contains $Q$. On $Q^*$, the function $\hat{G}_t$  can only take the values specified by  \eqref{eq:values_G}, all of which are  $\approx E_{\supr (Q^*)} \tilde{G}$ by dyadic doubling.
 
			\end{enumerate}
		\item \textbf{Subcase where $Q$ intersects two intervals in $\bigcup\limits_{\ell=1}^{t-1} \mathcal{T}_{\ell}$:}  call these intervals $T_1$ and $T_2$, which must be adjacent. Then $Q \subseteq T_1 \cup T_2$ by the size condition \eqref{eq:size_cond_Q_0}. By Lemma \ref{lemma:adj_trans}, we have the following subsubcases. 
			\begin{enumerate}
				\item \textbf{Subsubcase where the common $\mathcal{K}$-ancestor of $T_1$ and $T_2$ exists:} let $Q^*$ be this common ancestor. By Lemma \ref{lemma:adj_trans} and then dyadic doubling of $\tilde{G}$, for all $x \in T_1$ we have 
                \[
                \hat{G}_t(x) = E_{\supr (\pi_{\mathcal{K}} K_1 )} \tilde{G}\approx E_{\supr (Q^*)} \tilde{G} \, ,
                \]
                similarly for $T_2$.
				\item \textbf{Subsubcase no common $\mathcal{K}$-ancestor of $T_1$ and $T_2$ exists:}  set $Q^* \equiv [0,1)$. By Lemma \ref{lemma:adj_trans}, we have $T_1, T_2 \in \mathcal{K}_1$. Thus for $x \in T_1 \cup T_2$, we have
                \[
                \hat{G}_t (x) = E_{[0,1)} \tilde{G} = E_{\supr (Q^*)} \tilde{G} \, . \qedhere
                \]
			\end{enumerate}
	\end{enumerate}
	\end{proof}

\begin{figure}[ht]
  \fbox{\includegraphics[width=\linewidth]{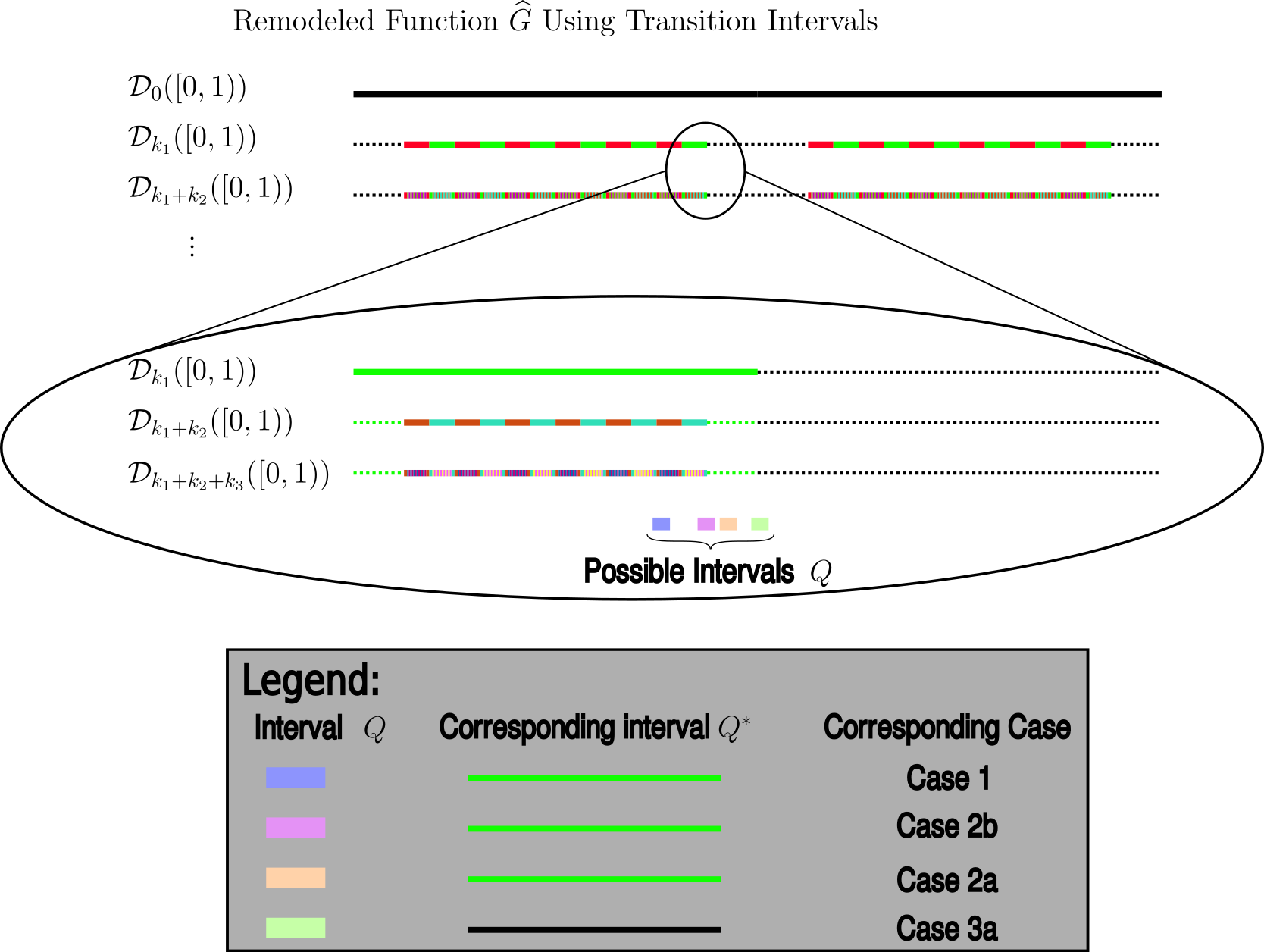}}
  \caption{Some of the cases occuring in the proof of Lemma \ref{lem:remodeling_loc_cst}, when $t=3$}
\label{fig:locally_constant}
\end{figure}

\begin{remark}
  In Lemma \ref{lem:remodeling_loc_cst} the condition $x \in Q$ is easily replaced by $x \in 3Q$ or $x \in 5Q$, or really $x$ in any fixed constant dilate of $Q$. Furthermore, the choice of $Q^*$ in Lemma \ref{lem:remodeling_loc_cst} is neither unique nor canonical.
\end{remark}

\subsection{Effects of remodeling on Muckenhoupt conditions}\label{subsection:remodeling_Muckenhoupt}
If $\tilde{u} \in L^{\infty} ([0,1))$, then $\bar{u}$ denotes the remodeled function \emph{without}  transition intervals, and $\widehat{u}$ denotes the periodic extension of the remodeled function \emph{with} transition intervals. The result below says that remodeling preserves the $A_p$ condition, as noted when $p=2$ in \cite[Section 4]{NaVo}, \cite[Remark A.16 combined with doubling]{AlLuSaUr}, \cite{Naz2} and for general $p$ in \cite[Section 6.1.1 combined with doubling in Section 6.1.2]{KaTr}.

\begin{lemma}[{\cite{NaVo, KaTr,AlLuSaUr, Naz2}}] \label{lem:remodeling_Ap}
  Let $\tau \in \left ( 0, \frac{1}{2} \right )$. If $\tilde{\sigma}$ and $\tilde{\omega}$ are $\tau$-doubling on $\mathcal{D}([0,1))$, then 
  \[
   A_{p} (\hat{\sigma}, \hat{\omega})\lesssim A_{p} ^{\mathcal{D}([0,1))} (\tilde{\sigma}, \tilde{\omega}) \, . 
  \] 
\end{lemma}

\begin{proof}
  Since $\hat{\sigma}$ and $\hat{\omega}$ are doubling by Lemma \ref{lem:doubling_remodeling} and periodic, we have 
\[
  A_p  (\hat{\sigma}, \hat{\omega}) \approx A_p ^{\mathcal{D} ([0,1))} (\hat{\sigma}, \hat{\omega}) \, .
\]
By Proposition \ref{prop:measure_preserving_avgs_2} part (\ref{averages_same_2}) and dyadic doubling,
\[
A_p ^{\mathcal{D}([0,1))} (\hat{\sigma}, \hat{\omega})\lesssim A_p ^{\mathcal{K} ([0,1))} (\hat{\sigma}, \hat{\omega}) \, .
\]
Let $Q \in \mathcal{K} ( [0,1))$. If $Q \in \hat{\mathcal{K}}$, then  
\[
E_Q \hat{\sigma}  = E_{\supr (Q)} \widetilde{\sigma} \, , \qquad E_Q \hat{\omega} = E_{\supr (T)} \tilde{\omega} \, .
\]
And if $Q \not\in \hat{\mathcal{K}}$, then $Q$ must be contained in a transition cube $T$, and so on $Q$ we have 
\[
\hat{\sigma} = E_{\supr (T)} \widetilde{\sigma} \, , \qquad \hat{\omega} = E_{\supr (T)} \widetilde{\omega} \, .
\]
Hence
\[
  A_p ^{\mathcal{K} ([0,1))} (\hat{\sigma}, \hat{\omega}) \leq \sup\limits_{K \in \mathcal{K}([0,1))]} E_{\supr(K)} \tilde{\sigma} E_{\supr(K)} \tilde{\omega}  \leq A_p ^{\mathcal{D} ([0,1))} (\tilde{\sigma}, \tilde{\omega}) \, . \qedhere 
\]
\end{proof}

In contrast, the following new contribution of this paper  says that under remodeling, the $A_{p} ^{\ell^2, \operatorname{local}, \mathcal{D}([0,1))}$ characteristic can only get larger. Here, the condition on $\{k_j\}$ is just that each $k_j$ has to be sufficiently large depending on the weights and $\epsilon$. 
\begin{proposition} \label{prop:remodeling_quad_Muck}
	If $\tilde{\sigma}$ and $\tilde{\omega}$ are weights in $L^{\infty} ([0,1))$, then 
    \[
      A_{p} ^{\ell^2, \operatorname{local}, \mathcal{D}([0,1))} (\tilde{\sigma}, \tilde{\omega}) \leq A_{p} ^{\ell^2, \operatorname{local}, \mathcal{D}([0,1))} (\bar{\sigma}, \bar{\omega}) \, . 
  \]
	Furthermore, there exists a sequence $\{F_j\}_{j \geq 1}$ of functions such that for any $\epsilon > 0$, if $k_j \geq F_j (\tilde{\sigma}, \tilde{\omega}, \epsilon)$, 
 then 
  \begin{align}\label{eq:Ap_quad_inc_remodeling}
  (1- \epsilon) A_{p} ^{\ell^2, \operatorname{local}, \mathcal{D}([0,1))} (\tilde{\sigma}, \tilde{\omega}) \leq A_{p} ^{\ell^2, \operatorname{local}, \mathcal{D}([0,1))} (\hat{\sigma}, \hat{\omega}) \, . 
  \end{align}
\end{proposition}
\begin{proof}
  Let $\epsilon > 0$, and let $\{a_J\}_{J \in \mathcal{D} ( [0,1))}$ be an almost-extremizing sequence for $A_p ^{\ell^2, \operatorname{local}, \mathcal{D}([0,1))} (\widetilde{\sigma}, \tilde{\omega})$, i.e., let $\{a_J\}_{J \in \mathcal{D} ([0,1))}$ be a sequence with only finitely many nonzero terms such that
\[
  \left \| \left \{ \sum\limits_{J \in \mathcal{D}([0,1))} a_J ^2 \left ( E_{J} \tilde{\sigma} \right )^2 \mathbf{1}_J \right \} ^{\frac{1}{2}} \right \|_{L^p (\tilde{\omega})} > \left ( 1- \frac{\epsilon}{2} \right ) A_{p} ^{\ell^2, \operatorname{local}, \mathcal{D}([0,1))} (\tilde{\sigma}, \tilde{\omega}) \left \| \left \{ \sum\limits_{J \in \mathcal{D} ([0,1))} a_J ^2 \mathbf{1}_J \right \} ^{\frac{1}{2}} \right \|_{L^p (\tilde{\sigma})} \, .
\]
Since $\{a_J\}$ only has finitely many nonzero terms, there exists a positive integer $m = m(\epsilon, \tilde{\sigma}, \tilde{\omega})$ for which
\[
  \left \|  \left \{ \sum\limits_{t=0}^{m} \sum\limits_{J \in \mathcal{D}_t ([0,1))} a_J ^2 \left ( E_{J} \tilde{\sigma} \right )^2 \mathbf{1}_J \right \} ^{\frac{1}{2}} \right \|_{L^p (\tilde{\omega})} > \left ( 1- \frac{\epsilon}{2} \right ) A_{p} ^{\ell^2, \operatorname{local}, \mathcal{D}([0,1))} (\tilde{\sigma}, \tilde{\omega}) \left \| \left \{ \sum\limits_{t=0}^{m} \sum\limits_{J \in \mathcal{D}_t ([0,1))} a_J ^2 \mathbf{1}_J \right \} ^{\frac{1}{2}} \right \|_{L^p (\tilde{\sigma})} \, .
\]

We first show, using this sequence $\{a_J\}_{J \in \mathcal{D} ([0,1)}$, that
\[
  A_{p} ^{\ell^2, \operatorname{local}, \mathcal{D} ([0,1))} (\bar{\sigma}, \bar{\omega} ) > \left ( 1- \frac{\epsilon}{2} \right ) A_{p} ^{\ell^2, \operatorname{local}, \mathcal{D}([0,1))} (\tilde{\sigma}, \tilde{\omega}) \, ,
\]
which, by taking $\epsilon \to 0$, will prove the first statement in Proposition \ref{prop:remodeling_quad_Muck}. By \eqref{eq:indicators_Phi_2}, we write
\begin{multline*}
\left \| \left \{ \sum\limits_{t=0}^{m} \sum\limits_{J \in \mathcal{D}_t ( [0,1)) } \sum\limits_{K \in \mathcal{K}_t ( [0,1)) ~:~ \supr (K) = J} a_J ^2 \left ( E_{K} \bar{\sigma}\right )^2 \mathbf{1}_K \right \} ^{\frac{1}{2}} \right \|_{L^p ( \bar{\omega})} \\
  = \left \| \left \{  \sum\limits_{t=0}^{m} \sum\limits_{J \in \mathcal{D}_t ([0,1)) }  a_J ^2 \left ( E_{J} \tilde{\sigma} \right )^2 \mathbf{1}_{J} \circ \Phi \right \} ^{\frac{1}{2}} \right \|_{L^p (\tilde{\omega} \circ \Phi)} = \left \| \left \{ \sum\limits_{t=0}^{m} \sum\limits_{J \in \mathcal{D}_t ([0,1)) }  a_J ^2 \left ( E_{J} \tilde{\sigma} \right )^2 \mathbf{1}_{J}\right \} ^{\frac{1}{2}} \right \|_{L^p (\tilde{\omega})} \, ,
\end{multline*}
where in the last step we used that $\Phi$ is measure-preserving. Similarly, we have
\begin{align*}
  &\left \| \left \{ \sum\limits_{t=0}^{m}  \sum\limits_{J \in \mathcal{D}_t ( [0,1)) } \sum\limits_{K \in \mathcal{K}_t ( [0,1)) ~:~ \supr (K) = J} a_J ^2 \mathbf{1}_K \right \} ^{\frac{1}{2}} \right \|_{L^p ( \bar{\sigma})} =\left \| \left \{ \sum\limits_{t=0}^{m}  \sum\limits_{J \in \mathcal{D}_t ([0,1))} a_J ^2 \mathbf{1}_J \right \} ^{\frac{1}{2}} \right \|_{L^p (\tilde{\sigma})} \ ,
\end{align*}
and so
\[
  A_{p} ^{\ell^2, \operatorname{local}, \mathcal{D} ([0,1))} (\bar{\sigma}, \bar{\omega}) > \left ( 1- \frac{\epsilon}{2} \right ) A_{p} ^{\ell^2, \operatorname{local}, \mathcal{D}([0,1))} (\tilde{\sigma}, \tilde{\omega}) \, .
\]

Now let us see how to get the desired result for the weights $\hat{\sigma}$ and $\hat{\omega}$; to do so we will use the same $\{a_J\}_J$ above and the computations we did for $\bar{\sigma}$ and $\bar{\omega}$. Consider
\[
  \left \| \left \{ \sum\limits_{t=0}^{m} \sum\limits_{J \in \mathcal{D}_t ( [0,1)) } \sum\limits_{K \in \hat{\mathcal{K}}_t ( [0,1)) ~:~ \supr (K) = J} a_J ^2 \left ( E_{K} \hat{\sigma}\right )^2 \mathbf{1}_K \right \} ^{\frac{1}{2}} \right \|_{L^p ( \hat{\omega})} \, .
\]
Since the integrand 
\[
\left \{ \sum\limits_{t=0}^{m} \sum\limits_{J \in \mathcal{D}_t ( [0,1)) } \sum\limits_{K \in \hat{\mathcal{K}}_t ( [0,1)) ~:~ \supr (K) = J} a_J ^2 \left ( E_{K} \hat{\sigma}\right )^2 \mathbf{1}_K \right \} ^{\frac{1}{2}} 
\]
is constant on cubes in the set $\left ( \bigcup\limits_{\ell=0}^m \mathcal{T}_{\ell} \right ) \cup \hat{\mathcal{K}}_m$, on which $\hat{\omega}$ has the same averages as $\bar{\omega}$, and similary for $\hat{\sigma}$ and $\bar{\sigma}$, we get 
\begin{align*}
  &\left \| \left \{ \sum\limits_{t=0}^{m} \sum\limits_{J \in \mathcal{D}_t ( [0,1)) } \sum\limits_{K \in \hat{\mathcal{K}}_t ( [0,1)) ~:~ \supr (K) = J} a_J ^2 \left ( E_{K} \hat{\sigma}\right )^2 \mathbf{1}_K \right \} ^{\frac{1}{2}} \right \|_{L^p ( \hat{\omega})} \\
= &\left \| \left \{ \sum\limits_{t=0}^{m} \sum\limits_{J \in \mathcal{D}_t ( [0,1)) } a_J ^2 \left ( E_{J} \tilde{\sigma}\right )^2 \sum\limits_{K \in \hat{\mathcal{K}}_t ( [0,1)) ~:~ \supr (K) = J}  \mathbf{1}_K \right \} ^{\frac{1}{2}} \right \|_{L^p ( \hat{\omega})} \\
= &\left \| \left \{ \sum\limits_{t=0}^{m} \sum\limits_{J \in \mathcal{D}_t ( [0,1)) } a_J ^2 \left ( E_{J} \tilde{\sigma}\right )^2 \sum\limits_{K \in \hat{\mathcal{K}}_t ( [0,1)) ~:~ \supr (K) = J}  \mathbf{1}_K \right \} ^{\frac{1}{2}} \right \|_{L^p ( \bar{\omega})} \\
= &\left \| \left \{ \sum\limits_{t=0}^{m} \sum\limits_{J \in \mathcal{D}_t ( [0,1)) } a_J ^2  \sum\limits_{K \in \hat{\mathcal{K}}_t ( [0,1)) ~:~ \supr (K) = J}  \left ( E_{K} \bar{\sigma}\right )^2\mathbf{1}_K \right \} ^{\frac{1}{2}} \right \|_{L^p ( \bar{\omega})} .
 \end{align*}
 Compare this last term to the same quantity where the inner-most sum is over $\mathcal{K}_t$ instead of $\hat{\mathcal{K}}_t$, and using the reverse triangle inequality for the mixed $L^p (\ell^2, \bar{\omega})$-norm, we get that the above is at least  
 \begin{align*}
&\left \| \left \{ \sum\limits_{t=0}^m \sum\limits_{J \in \mathcal{D}_t ( [0,1)) } \sum\limits_{K \in \mathcal{K}_t ( [0,1)) ~:~ \supr (K) = J} a_J ^2 \left ( E_{K} \bar{\sigma}\right )^2 \mathbf{1}_K \right \} ^{\frac{1}{2}} \right \|_{L^p ( \bar{\omega})}\\
-&\left \| \left \{ \sum\limits_{t=0}^{m} \sum\limits_{J \in \mathcal{D}_t ( [0,1)) } a_J ^2  \sum\limits_{K \in \left \{ \mathcal{K}_t ( [0,1)) \setminus \hat{\mathcal{K}}_t ( [0,1)) \right \} ~:~ \supr (K) = J}  \left ( E_{K} \bar{\sigma}\right )^2\mathbf{1}_K \right \} ^{\frac{1}{2}} \right \|_{L^p ( \bar{\omega})} \, . 
 \end{align*}
 Since we already showed that the first term equals
 \[
   \left \| \left \{ \sum\limits_{t=0}^m  \sum\limits_{J \in \mathcal{D}_t ([0,1)) }  a_J ^2 \left ( E_{J} \tilde{\sigma} \right )^2 \mathbf{1}_{J}\right \} ^{\frac{1}{2}} \right \|_{L^p (\tilde{\omega})} \, ,
 \]
  we now show the second term is small if  $k_1, \ldots, k_m$ are all sufficiently large. Indeed, we use the fact that $\|\bar{\sigma}\|_{\infty} \leq \|\tilde{\sigma}\|_{\infty}$, and similary for $\bar{\omega}$, and that $|a_J| \leq C(\epsilon, \tilde{\sigma}, \tilde{\omega})$, to say the second term is at most
 \begin{align*}
   &C(\epsilon, \tilde{\sigma}, \tilde{\omega}) \left \| \left \{ \sum\limits_{t=0}^{m} \sum\limits_{J \in \mathcal{D}_t ( [0,1)) } \sum\limits_{K \in \left \{ \mathcal{K}_t ( [0,1)) \setminus \hat{\mathcal{K}}_t ( [0,1)) \right \} ~:~ \supr (K) = J}  \mathbf{1}_K \right \} ^{\frac{1}{2}} \right \|_{L^p (dx)} \\
   \leq &C(\epsilon, \tilde{\sigma}, \tilde{\omega}) \left \| \sum\limits_{t=0}^{m} \sum\limits_{J \in \mathcal{D}_t ( [0,1)) } \sum\limits_{K \in \left \{ \mathcal{K}_t ( [0,1)) \setminus \hat{\mathcal{K}}_t ( [0,1)) \right \} ~:~ \supr (K) = J}  \mathbf{1}_K  \right \|_{L^p (dx)} \, .
 \end{align*}
 Since the integrand has $L^{\infty}$-norm at most $m = m(\epsilon, \tilde{\sigma}, \tilde{\omega})$, then by H\"older's inequality it suffices to show the support of the integrand can be made arbitrarily small. The support of the function has Lebesgue measure at most
 \[
   \sum\limits_{\ell=1}^m \left | T_{\ell} ([0,1)) \right |\leq \sum\limits_{\ell=1}^m 2^{-k_{\ell}} \prod\limits_{j=1}^{\ell-1} \left ( 1 - 2 \cdot 2^{-k_j} \right ) \leq \sum\limits_{\ell=1}^m 2^{-k_{\ell}} \, . 
 \]
 By choosing $k_1, \ldots, k_m$ sufficiently large, the Lebesgue measure of the support can be made arbitrarily small. Thus we have shown that
 \[
   \left \| \left \{ \sum\limits_{t=0}^{m} \sum\limits_{J \in \mathcal{D}_t ( [0,1)) } \sum\limits_{K \in \hat{\mathcal{K}}_t ( [0,1)) ~:~ \supr (K) = J} a_J ^2 \left ( E_{K} \hat{\sigma}\right )^2 \mathbf{1}_K \right \} ^{\frac{1}{2}} \right \|_{L^p ( \hat{\omega})}
 \]
is at most 
\[
   \left \| \left \{ \sum\limits_{t=0}^m  \sum\limits_{J \in \mathcal{D}_t ([0,1)) }  a_J ^2 \left ( E_{J} \tilde{\sigma} \right )^2 \mathbf{1}_{J}\right \} ^{\frac{1}{2}} \right \|_{L^p (\tilde{\omega})} \, ,
 \]
 minus some term which goes to $0$ as $k_1, \ldots, k_m$ are all chosen sufficiently large. Similarly, we have
\[
   \left \| \left \{ \sum\limits_{t=0}^{m} \sum\limits_{J \in \mathcal{D}_t ( [0,1)) } \sum\limits_{K \in \hat{\mathcal{K}}_t ( [0,1)^2) ~:~ \supr (K) = J} a_J ^2  \mathbf{1}_K \right \} ^{\frac{1}{2}} \right \|_{L^p ( \hat{\sigma})}
 \]
is at least 
\[
   \left \| \left \{ \sum\limits_{t=0}^m  \sum\limits_{J \in \mathcal{D}_t ([0,1)) }  a_J ^2  \mathbf{1}_{J}\right \} ^{\frac{1}{2}} \right \|_{L^p (\tilde{\sigma})} \, ,
 \]
 plus some term which goes to $0$ as $k_1, \ldots, k_m$ are all chosen sufficiently large. Thus we can choose $k_1, \ldots, k_m$ sufficiently large so that
 \begin{align*}
    \left \| \left \{ \sum\limits_{t=0}^{m} \sum\limits_{J \in \mathcal{D}_t ( [0,1)) } \sum\limits_{K \in \hat{\mathcal{K}}_t ( [0,1)) ~:~ \supr (K) = J} a_J ^2 \left ( E_{K} \hat{\sigma}\right )^2 \mathbf{1}_K \right \} ^{\frac{1}{2}} \right \|_{L^p ( \hat{\omega})}& \\
    > (1- \epsilon) A_{p} ^{\ell^2, \operatorname{local}, \mathcal{D}([0,1))} (\tilde{\sigma}, \tilde{\omega}) & \left \| \left \{ \sum\limits_{t=0}^{m} \sum\limits_{J \in \mathcal{D}_t ( [0,1)) } \sum\limits_{K \in \hat{\mathcal{K}}_t ( [0,1)) ~:~ \supr (K) = J} a_J ^2  \mathbf{1}_K \right \} ^{\frac{1}{2}} \right \|_{L^p ( \hat{\sigma})}
 \, ,
 \end{align*}
 and so we conclude that $A_p ^{\ell^2, \operatorname{local}, \mathcal{D}([0,1))} (\hat{\sigma}, \hat{\omega}) > (1- \epsilon) A_{p} ^{\ell^2, \operatorname{local}, \mathcal{D}([0,1))} (\tilde{\sigma}, \tilde{\omega})$.
\end{proof}

\section{The quadratic Muckenhoupt condition in the plane and the norm inequality for \texorpdfstring{$R_2$}{R2}}\label{section:Ap_rectangles}

We recall that for doubling measures and Stein elliptic operators like the second Riesz transform $R_2$, the quadratic Muckenhoupt conditions follow from the norm inequality.
\begin{proposition}[{\cite[paragraph preceding Section 4.1]{SaWi}}] \label{lem:norm_rect_quad_Muck}
	Let $1<p<\infty$, let $\sigma$ and $\omega$ be doubling measures on $\mathbb{R}^n$, and let $T$ be Stein elliptic. Then 
	\begin{equation}\label{eq:Ap_quad_necessary}
		A_p ^{\ell^2, \operatorname{local}} (\sigma, \omega) \lesssim \mathfrak{N}_{T, p} \left (\sigma, \omega \right ) \, .
	\end{equation}
\end{proposition}
\begin{proof}
Let $\{J\}$ be a countable sequence of cubes, and let $\{a_J\}$  be a sequence of nonnegative coefficients.

   For any operator $T$, it is well-known that
   \begin{equation}\label{eq:norm_vector_valued}
   \left \| \left \{ \sum\limits_{k} |T_{\sigma} f_k|^2 \right \} ^{\frac 1 2} \right \|_{L^p (\omega)} \leq \mathfrak{N}_{T, p} (\sigma, \omega) \left \| \left \{ \sum\limits_{k} | f_k|^2 \right \} ^{\frac 1 2} \right \|_{L^p (\sigma)}
   \end{equation}
   holds across all sequences of functions $\{f_k\}$.

    Replacing the index $k$ by $J$, and then inserting $\mathbf{1}_J$ outisde the operator, we get \eqref{eq:norm_vector_valued} implies
   \begin{equation}\label{eq:norm_weaker_quad}
   \left \| \left \{ \sum\limits_{J} |T_{\sigma} f_J|^2 \mathbf{1}_J \right \} ^{\frac 1 2} \right \|_{L^p (\omega)} \leq \mathfrak{N}_{T, p} (\sigma, \omega) \left \| \left \{ \sum\limits_{J} | f_J|^2 \right \} ^{\frac 1 2} \right \|_{L^p (\sigma)} \, .
   \end{equation}

  Because $T$ is Stein-elliptic, then for each cube $J^*$ satisfying
  \[
\ell (J) = \ell(J^*) \, , \qquad \ell (J) \approx \operatorname{dist} (J, J^*) \, ,
  \]
  and such that
  \[
|T_{\sigma} \mathbf{1}_{J^*} (x) | \gtrsim E_J \sigma
  \]
  for all $x \in J$.
   Setting
   \[
   f_J = a_J \mathbf{1}_{J^*} \, , 
   \]
   from \eqref{eq:norm_weaker_quad} we then obtain
   \[
    \left \| \left \{ \sum\limits_{J} |a_J|^2 \left ( E_J \sigma\right )^2 \mathbf{1}_J \right \} ^{\frac 1 2} \right \|_{L^p (\omega)} \lesssim \mathfrak{N}_{T, p} (\sigma, \omega) \left \| \left \{ \sum\limits_{J} | a_J|^2 \mathbf{1}_{J^*} \right \} ^{\frac 1 2} \right \|_{L^p (\sigma)} \, .
   \]
Finally, bounding the right side using \eqref{eq:FS_key_application_piece}, the Fefferman-Stein inequality Lemma \ref{lem:FS} and the fact that $\sigma$ and $\omega$ are doubling yields
\[
    \left \| \left \{ \sum\limits_{J} |a_J|^2 \left ( E_J \sigma\right )^2 \mathbf{1}_J \right \} ^{\frac 1 2} \right \|_{L^p (\omega)} \lesssim \mathfrak{N}_{T, p} (\sigma, \omega) \left \| \left \{ \sum\limits_{J} | a_J|^2 \mathbf{1}_{J} \right \} ^{\frac 1 2} \right \|_{L^p (\sigma)} \, .
   \]
   Thus \eqref{eq:Ap_quad_necessary} holds.
\end{proof}

If our measures in $\mathbb{R}^2$ only depend on one of the two coordinates, then the full-dimensional dyadic quadratic Muckenhoupt condition must be at least as large as the one-dimensional one.

\begin{lemma}\label{lem:quad_Muck_tensor}
  Let $1< p < \infty$. If $\hat{\sigma}$ and $\hat{\omega}$ are weights on $\mathbb{R}$ and $\sigma(x_1, x_2) = \hat{\sigma} (x_1)$, $\omega(x_1, x_2) = \hat{\omega} (x_1)$ are weights on $\mathbb{R}^2$, then  
  \[
	  A_{p} ^{\ell^2, \operatorname{local}, \mathcal{D} \left ( \left [0,1 \right )^2 \right) } (\sigma, \omega) \geq A_{p} ^{\ell^2, \operatorname{local}, \mathcal{D} \left ( \left [0,1 \right ) \right)} (\hat{\sigma}, \hat{\omega}) \, . 
  \]
\end{lemma}
\begin{proof}
Let $\epsilon > 0$, and let $\{a_J\}_{J \in \mathcal{D} ( [0,1))}$ be an almost-extremizing sequence for $A_p ^{\ell^2, \operatorname{local}, \mathcal{D}([0,1))} (\hat{\sigma}, \hat{\omega})$, i.e.\ let $\{a_J\}_{J \in \mathcal{D} ([0,1))}$ be a sequence with only finitely many nonzero terms such that
	\begin{align}\label{eq:almost_extreme_Ap_quad_dim_1}
  \left \| \left \{ \sum\limits_{J \in \mathcal{D}([0,1)^2)} a_J ^2 \left ( E_{J} \hat{\sigma} \right )^2 \mathbf{1}_J \right \} ^{\frac{1}{2}} \right \|_{L^p (\hat{\omega})} > \left ( 1- \epsilon \right ) A_{p} ^{\ell^2, \operatorname{local}, \mathcal{D}([0,1))} (\hat{\sigma}, \hat{\omega}) \left \| \left \{ \sum\limits_{J \in \mathcal{D} ([0,1))} a_J ^2 \mathbf{1}_J \right \} ^{\frac{1}{2}} \right \|_{L^p (\hat{\sigma})} \, .
	\end{align}
We show, using this $\{a_J\}_{J \in \mathcal{D} ([0,1)}$, that
	\begin{align*}
  A_{p} ^{\ell^2, \operatorname{local}, \mathcal{D} ([0,1)^2)} (\sigma, \omega ) > \left ( 1- \epsilon \right ) A_{p} ^{\ell^2, \operatorname{local}, \mathcal{D}([0,1))} (\hat{\sigma}, \hat{\omega}) \, ,
	\end{align*}
which, by taking $\epsilon \to 0$, will prove Lemma \ref{lem:quad_Muck_tensor}.

	Given an interval $J \in \mathcal{D} \left ( \left [0,1 \right ) \right )$, define the square $J_k$ by  
	\[
J_k \equiv J \times \left [ k \ell \left (J \right ) , \left ( k+1 \right ) \ell \left (J \right ) \right ) \, , 
	\]
	and note that $\mathcal{D} \left ([0,1)^2 \right ) = \bigcup\limits_{J \in \mathcal{D} \left ( [0,1) \right )} \left \{J_k \right \}_{k=0}^{\frac{1}{\ell \left ( J \right )} -1}$. Define $b_{J_k} \equiv a_J$, and consider
	\begin{multline*}
		\left \| \left \{ \sum\limits_{J \in \mathcal{D}([0,1))} \sum\limits_{k=0}^{ \frac{1}{\ell \left ( J \right )} -1 } b_{J_k} ^2 \left ( E_{J_k } \sigma \right )^2 \mathbf{1}_{J_k} \right \} ^{\frac{1}{2}} \right \|_{L^p (\omega)} = \left \| \left \{ \sum\limits_{J \in \mathcal{D}([0,1))} a_J ^2 \left ( E_{J} \hat{\sigma} \right )^2 \sum\limits_{k=0}^{ \frac{1}{\ell \left ( J \right )} -1 }  \mathbf{1}_{J_k} \right \} ^{\frac{1}{2}} \right \|_{L^p (\omega)} \\
		= \left \| \left \{ \sum\limits_{J \in \mathcal{D}([0,1))} a_J ^2 \left ( E_{J} \hat{\sigma} \right )^2 \mathbf{1}_{J} \left (x_1 \right ) \mathbf{1}_{[0,1)} (x_2) \right \} ^{\frac{1}{2}} \right \|_{L^p (\omega)} = \left \| \left \{ \sum\limits_{J \in \mathcal{D}([0,1))} a_J ^2 \left ( E_{J} \hat{\sigma} \right )^2 \mathbf{1}_{J}  \right \} ^{\frac{1}{2}} \right \|_{L^p (\hat{\omega})} \, .
	\end{multline*}
Similarly, we get
	\[
		\left \| \left \{ \sum\limits_{J \in \mathcal{D}([0,1))} \sum\limits_{k=0}^{ \frac{1}{\ell \left ( J \right )} -1 } b_{J_k} ^2 \mathbf{1}_{J_k} \right \} ^{\frac{1}{2}} \right \|_{L^p (\sigma)} = \left \| \left \{ \sum\limits_{J \in \mathcal{D}([0,1))} a_J ^2 \mathbf{1}_{J}  \right \} ^{\frac{1}{2}} \right \|_{L^p (\hat{\sigma} ))} \, , 
	\]
	and so by \eqref{eq:almost_extreme_Ap_quad_dim_1} we have
	\[
		A_{p} ^{\ell^2, \operatorname{local}, \mathcal{D} ([0,1)^2)} (\sigma, \omega ) \geq \frac{ \left \| \left \{ \sum\limits_{J \in \mathcal{D}([0,1))} \sum\limits_{k=0}^{ \frac{1}{\ell \left ( J \right )} -1 } b_{J_k} ^2 \left ( E_{J_k } \sigma \right )^2 \mathbf{1}_{J_k} \right \} ^{\frac{1}{2}} \right \|_{L^p (\omega)} }{ \left \| \left \{ \sum\limits_{J \in \mathcal{D}([0,1))} \sum\limits_{k=0}^{ \frac{1}{\ell \left ( J \right )} -1 } b_{J_k} ^2 \mathbf{1}_{J_k} \right \} ^{\frac{1}{2}} \right \|_{L^p (\sigma)}} > \left ( 1- \epsilon \right ) A_{p} ^{\ell^2, \operatorname{local}, \mathcal{D}([0,1))} (\hat{\sigma}, \hat{\omega}) \, . \qedhere
	\]
\end{proof}

\section{Proof of Theorem \ref{thm:counter_example}: Scalar Triple Testing for \texorpdfstring{$R_2$}{}}\label{section:testing}

We are ready to prove Theorem \ref{thm:counter_example}. Let $\Gamma > 1$, $\tau \in \left ( 0, \frac{1}{2} \right )$. By duality, it suffices to prove Theorem \ref{thm:counter_example} when $p > 2$, which we assume here onwards. Let $\epsilon  \in (0, \frac{1}{2})$ be fixed shortly. Begin with $\tilde{\sigma}$ and $\tilde{\omega}$ from Proposition \ref{prop:weights_small_dyadic_doubling}, which are $\epsilon$-doubling on $\mathcal{D}([0,1))$, and say both are $\mathcal{D}$ dyadic step functions with step size $2^{-m}$. Then let $\{k_j\}_{j \geq 0}$ be a sequence to be chosen later with $k_0 = 0$. Apply Nazarov's remodeling of Section \ref{section:transplantion} to obtain periodic weights $\hat{\sigma}_t$ and $\hat{\omega}_t$ on $\mathbb{R}$ for all  $t\geq 0$, which by Lemma \ref{lem:doubling_remodeling}, are $\tau$-doubling on $\mathbb{R}$ after fixing $\epsilon =\epsilon(\tau)$. 

Next, we extend $\hat{\sigma}_t$ and $\hat{\omega}_t$ to $\tau$-'doubling measures $\sigma_t$ and $\omega_t$ on the plane $\mathbb{R}^2$ by setting
  \begin{equation}\label{eq:defn_sigma_omega_testing}
    \sigma_t (x_1, x_2) \equiv \hat{\sigma}_t (x_1) \, , \qquad \omega_t (x_1, x_2) \equiv \hat{\omega}_t (x_1) \, .
  \end{equation}
And with $m$ as above, set 
\begin{align} \label{eq:defn_sigma_omega_testing_2}
  \hat{\sigma} \equiv \hat{\sigma}_m \, , &\qquad \hat{\omega} \equiv \hat{\omega}_m \, , \\
  \sigma \equiv \sigma_m \, , &\qquad \omega \equiv \omega_m \, . \label{eq:defn_sigma_omega_testing_3} 
\end{align}
 For convenience, define $A\equiv\left\vert \left[  0,1\right )
\right\vert _{\hat{\sigma}}$ and $B=\left\vert \left[  0,1\right )
\right\vert _{\hat{\omega}}$. Note that 
\[
A^{p-1} B\leq A_{p}^{\mathcal{D}([0,1))} (\widetilde{\sigma},\widetilde{\omega}) ^p \lesssim 1 \, .
\]
Recalling Definitions \eqref{eq:s_k_transition}, \eqref{eq:osc_1d_def} and  \eqref{eq:remodeled_martingale_diff_transition}, we also define
\[
  \operatorname{Osc}_{k_{t+1}} ^{R, \operatorname{horizontal}}(x_1, x_2) \equiv \operatorname{Osc}_{k_{t+1}} ^{R}(x_1) \, ,
\]
yielding the formulas
\begin{align}
  \label{eq:level_0}\sigma_0 &= A \, ,\\
  \sigma_{t+1} -\sigma_t &\label{eq:disc_sigma}=\sum\limits_{ R \in \mathcal{D}_t ([0,1)) }  \langle \tilde{\sigma}, h_{R} \rangle \frac{1}{\sqrt{ \left | R \right |}} \operatorname{Osc}_{k_{t+1}} ^{R, \operatorname{horizontal}} \, , 
\end{align}
and 
\begin{align}
  \omega_0 &= B \, ,\\
  \omega_{t+1} -\omega_t &=\sum\limits_{ R \in \mathcal{D}_t ([0,1)) }  \langle \tilde{\omega}, h_{R} \rangle \frac{1}{\sqrt{ \left | R \right |}} \operatorname{Osc}_{k_{t+1}} ^{R, \operatorname{horizontal}} \label{eq:disc_omega}\, . 
\end{align}
Because $\tilde{\sigma}$ and $\tilde{\omega}$ were dyadic step functions with step size $2^{-m}$, then for all $\ell \geq 0$ we have  
\[
  \sigma = \sigma_{m+ \ell} = \sigma_m \, , \qquad \omega = \omega_{m+ \ell} = \omega_m
\]

By Lemma \ref{lem:remodeling_Ap}, we have 
\begin{equation}\label{eq:Ap_final}
A_p (\sigma, \omega) \lesssim 1 \, .
\end{equation}
By Remark \ref{rmk:no_tails} and by taking $\tau$ sufficiently small, the same holds for the two-tailed characteristics, i.e., 
\[
\mathcal{A}_p (\sigma, \omega) \lesssim 1 \, , \qquad \mathcal{A}_{p'} (\omega, \sigma) \lesssim 1 \, .
\] By Lemma \ref{prop:remodeling_quad_Muck}, for any choice of $\{k_j\}$ satisfying $k_j \geq F_j (\Gamma, p)$ for some positive function $F_j$,  we have 
\begin{equation}\label{eq:Ap_quad_big2}
	A_{p} ^{\ell^2, \operatorname{local}, \mathcal{D} ([0,1))} (\hat{\sigma}, \hat{\omega}) > \Gamma \, .
\end{equation}
By Proposition \ref{lem:norm_rect_quad_Muck} and Lemma \ref{lem:quad_Muck_tensor}, and Stein-ellipticity of the second Riesz transform $R_2$, we also have 
\[
\mathfrak{N}_{R_2, p} (\sigma, \omega) \gtrsim \Gamma \, .
\]
Thus, Theorem \ref{thm:counter_example} will follow once we show boundedness of the full-testing characteristics. By Remark \ref{rmk:no_tails}, it suffices to show this for the triple testing characteristics, i.e.,
\begin{equation} \label{eq:triple_testing_bounded}
	\mathfrak{T}_{R_2, p} ^{\operatorname{triple}} (\sigma, \omega) \lesssim 1 \, , \qquad \mathfrak{T}_{R_2, p'} ^{\operatorname{triple}} ( \omega, \sigma) \lesssim 1 \, .
\end{equation}
The remainder of this section is dedicated to this goal, which is finally achieved in Section \ref{subsection:triple_testing} below. Our strategy follows    \cite[Sections 4-5]{AlLuSaUr}, where we also remodeled weights with oscillation only along the $e_1$ axis and showed boundedness of the \emph{dyadic} local scalar testing constants for $R_2$.

\subsection{Riesz transforms of rectangles and horizontally oscillating functions}\label{subsection:weak_convergence_properties}

A key lemma of \cite{AlLuSaUr} was the following, in which given a square $P = P_1 \times P_2$ in $\mathbb{R}^2$,  
we define
\[
\widehat{s}_k ^{P , \operatorname{horizonal}} (x_1, x_2) \equiv \widehat{s}_k ^{P_1} (x_1) \mathbf{1}_{P_2} (x_2)  \, .
\]
\begin{lemma}[{\cite[Lemma 30 part (3)]{AlLuSaUr}}]
\label{reduction}Suppose $q\in(1,\infty)$. Given a square $P\subset\mathbb{R}^{2}$, we have 
\[
 \lim\limits_{k \to \infty} \int \left | R_{2} \widehat{s}_k ^{P, \operatorname{horizontal}} (x) \right |^q dx  \to 0 \, .
\]
\end{lemma}
To show testing across \emph{all} cubes, and not just dyadic cubes, we generalize Lemma \ref{reduction} to  Lemma \ref{reduction_extra} below.

\begin{lemma}\label{lem:approx_dependence} Let $1< q< \infty$ and consider the Riesz transform $R_2$ on the plane $\mathbb{R}^2$.
  \begin{enumerate}
    \item \label{lem_part:Riesz_indicator} If $R$ is an axis-parallel rectangle, then $R_2 \mathbf{1}_{R}$ has only logarithmic singularities at the boundary of $R$.

    \item  \label{lem_part:Riesz_sum_indicators} Suppose $f = \sum\limits_{R} c_R \mathbf{1}_R$, where the sum is among a finite number of rectangles $R$ in $[0,1)^2$ of the form 
    \[
    R = R^1 \times [-1,1)  \, .
    \]
    Then there exists a nonnegative function $F$ such that, for every $\epsilon > 0$ and every translate $\mathcal{D}'$ of $\alpha \mathcal{D} \equiv \{\alpha 2^{-k} (j+ [0,1)^2))\}_{j \in \mathbb{Z}^2, k \in \mathbb{Z}}$ in $\mathbb{R}^2$, if a nonnegative integer $k$ satisfies
	    \[
		    \alpha 2^{-k} \leq F(q, M, \sup\limits_{R} \ell (R^1), \sup\limits_{R} c_R, \epsilon) \, , 
	    \]
		  then there exists a $\mathcal{D}'$ dyadic step function $f_k$ with step size $\alpha 2^{-k}$ for which
\[
  \int \left | \left | R_2 f \right |^q  - f_k \right | dx < \epsilon \, . 
      \]
		    \end{enumerate}
\end{lemma}
\begin{proof}
 For Part (\ref{lem_part:Riesz_indicator}), write  
  \[R = [a,b) \times [c,d) \, .
  \]
  A standard computation yields \begin{equation}\label{eq:R2_indicator}
	R_2 \mathbf{1}_{[a,b) \times [c,d)} (x_1, x_2) = C \{ ( \arcsinh z_{NW} -  \arcsinh  z_{NE}) -( \arcsinh  z_{SW} -  \arcsinh  z_{SE}) \} \, ,
      \end{equation}
      where $z_{NW} = \frac{x_1 - a}{|x_2 - d|}$, $z_{NE} = \frac{x_1 - b}{|x_2 - d|}$,  $z_{SW} = \frac{x_1 - a}{|x_2 - c|}$, $z_{SE} = \frac{x_1 - b}{|x_2 - d|}$, and $C > 0$ is an absolute constant. The identity $\arcsinh z = \log (\sqrt{z^2 + 1} + z)$ shows that \eqref{eq:R2_indicator} has only logarithmic singularities on $\partial R$.

  As for Part (\ref{lem_part:Riesz_sum_indicators}), because $R_2 f$ decays like $\frac{1}{|x|^2}$, then let $K$ be a large finite union of cubes in $\mathcal{D}'$ such that 
  \[
    \int\limits_{K^c} |R_2 f|^p dx < \frac{\epsilon}{3} 
  \]
	Since there are finitely many rectangles $R$ appearing in the formula for $f$, choose $k_0$ large enough so that there exists a finite collection of cubes $\{Q\}$ in $\mathcal{D}_{k_0} '$ for which $\bigcup\limits_R \partial R$ is contained in the interior of $\bigcup Q$ and 
  \[
    \int\limits_{K \cap \left (\bigcup Q \right )} \left | R_2 f \right |^p < \frac{\epsilon}{3} \, .
  \]
  Then on $K \cap (\bigcup Q)^c$, approximate $\left | R_2 f \right |^p$ uniformly by a $\mathcal{D}'$-dyadic step function $f_k$ of step size $\alpha 2^{-k}$ to control the last integral by $\frac{\epsilon}{3}$. This can be done since $\left |\nabla R_2 f (x) \right | \leq C (q, M, \sup\limits_{R} c_R , \epsilon)$ on $(\bigcup Q)^c$. 
\end{proof}

\begin{definition}
  A function $g_{\Delta}: \mathbb{R} \to \mathbb{C}$ is a  square wave with step-size $\Delta$ if 
  \[
  g_{\Delta} = s \sum\limits_{j=1}^{M} (-1)^j \mathbf{1}_{I_j} \, , 
  \] where  $\{I_j\}_{j=1}^M$ is a collection  of consecutive intervals of size $\Delta$ and $s \in \{-1,1\}$ is a number. 
  
\end{definition}
If $R \in \mathcal{D}_t \left ( [0,1) \right )$ then $\operatorname{Osc}_{k_{t+1}} ^R \mathbf{1}_{[0,1)}$ is a sum of  $2^{k_1 + \ldots + k_t}$ square waves, with pairwise disjoint support and step-size $2^{- \left ( k_1 + \ldots + k_{t+1} \right )}$. We will also need to track the changes in monotonicity of a function. 

\begin{definition}
  We say that a function $f:[a,b) \to \mathbb{C}$ is $M$\emph{-piecewise
monotone} if there is a partition 
\[
  \left\{  a=t_{0}<t_{1}<...<t_{M}=b\right\}
\] such that $f$ is monotone and of one sign on each subinterval $\left(
t_{k},t_{k+1}\right)  $, $0\leq k<M$.
\end{definition}

\begin{lemma}[Alternating Series Test, Lemma 23 in \cite{AlLuSaUr}]
  \label{Alt_series_test} Let $g_{\Delta}$ be a square wave with step size $\Delta$ supported on an interval $I$. If $b$ is $M$-piecewise monotone and bounded on $I$, then%
\[
  \left\vert \int b(x)  g_{\Delta}dx\right\vert \lesssim M \Delta \left\Vert
b\right\Vert _{\infty} \, .
\]
\end{lemma}

\begin{proof}
First suppose $I = [0,1)$ and 
\[
g_{\Delta}(x) = \sum\limits_{j=1}^{\frac{1}{\Delta}} \left (-1 \right)^j \mathbf{1}_{\left [ (j-1) \Delta, j \Delta \right )} (x) \, .
\]
Say $b$ is monotone decreasing on $[0,1)$. Then by the alternating series test we have
\begin{align}
  \label{alt_series_estimate}\left\vert \int b(x)g_{\Delta}(x)dx\right\vert
  =\left\vert \sum_{j=1}^{\frac{1}{\Delta} }\left(  -1\right)  ^{j}\int_{(j-1) \Delta } ^{j \Delta }b(x)dx\right\vert \leq\int_{0}^{\Delta}\left\vert
  b(x)\right\vert dx\leq \Delta \left\Vert b\right\Vert _{\infty} \, .
\end{align}
In general, if $I$ is arbitrary and $b$ is $M$-monotone on $I$ with partition $ t_0 < t_1 < \ldots < t_M $, we can apply this argument to
the subinterval $\left[  t_{m-1},t_{m}\right]  $ if the endpoints lie in
$\left\{  j\Delta \right\}  _{j=0}^{\frac{1}{\Delta}}$, the points of change in sign of
$g_{\Delta}$. For general $b$, if $ j_{m-1} \Delta$ (or $j_m \Delta$) denotes the leftmost (or rightmost) point
of the form $j \Delta$ in $[t_{m-1}, t_{m}]$, then the integrals over
each of the intervals $[t_{m-1}, j_{m-1} \Delta ]$, $[j_{m-1} \Delta, j_{m} \Delta]$, and $[j_{m} \Delta, z_{m}]$ all satisfy the same bound as (\ref{alt_series_estimate}).
\end{proof}

For simplicity, we only present the following lemmas in the plane. See \cite[Section 4]{AlLuSaUr} for how these results extend to higher dimensions. 
In what follows, given $f: \mathbb{R} \to \mathbb{C}$, we define the $\mathbb{R}^2 \to \mathbb{C}$ function
\[
f \otimes \mathbf{1}_{[-1,1)} (x_1,x_2) \equiv f(x_1) \mathbf{1}_{[-1,1)} (x_2) \, . 
\]
\begin{lemma}
  \label{reduction_extra}Let  $q\in(1,\infty)$.  If 
  \[
  g_{\Delta} = \sum\limits_{j=1}^B g_{\Delta, j}
  \]
  where $\{g_{\Delta, j}\}_{j=1}^B$ is a collection of square waves supported in $[-1,1)$ with step size $\Delta$ and with pairwise disjoint support, then 
\begin{equation}\label{eq:alt_series_est_applied}
  \left \|R_2 ( g_{\Delta} \otimes \mathbf{1}_{[-1,1)}) \right \|_{L^q (dx)} \lesssim (B \Delta)^{\frac{q(2q-1)}{1 + 3q + 12q^2}} \, .
\end{equation}
The implicit constant only depends on $q$.
\end{lemma}

\begin{proof}
For convenience define
\[
P_r \equiv [-r,r)^2  \, .
\]
With 
\[
G \equiv R_2 (g_{\Delta} \otimes \mathbf{1}_{[-1,1)} )
\]
and $r \geq 10$,
we write 
\begin{equation}\label{eq:split_Riesz_osc}
\int\limits_{\mathbb{R}^2} |G(x)|^q \, dx = \int\limits_{\mathbb{R}^2 \setminus P_r} |G(x)|^q \, dx + \int\limits_{H_{\delta} ^r} |G(x)|^q \, dx +  \int\limits_{P_r \setminus H_{\delta} ^r} |G(x)|^q \, dx \, ,
\end{equation}
where for $\delta \in (0,1)$, we define 
\[
H_{\delta}^{r}\equiv\left\{  x\in P_r : | |x_{2}| - 1 |<\delta \right\}  \, .
\]
We will estimate all three terms on the right of \eqref{eq:split_Riesz_osc} in terms of $r, \delta, q, \Delta, B$, and then optimize in the choice of $r$ and $\delta$.

Using the Calder\'on-Zygmund size estimates and the fact that $|g_{\Delta} \otimes \mathbf{1}_{[-1,1)}| \leq \mathbf{1}_P$, and then integrating using polar coordinates, the first term on the right of \eqref{eq:split_Riesz_osc} equals
\begin{equation}\label{eq:tail_bd_G}
\int\limits_{\mathbb{R}^{2}\setminus P_r}\left\vert G(x)\right\vert ^{q}dx\lesssim \int\limits_{\mathbb{R}%
^{2}\setminus P_r}\left(  \int\limits_{P_1}\frac{1}{\left\vert x-y\right\vert ^{2}%
}dy\right)  ^{q}dx\lesssim \int\limits_{\mathbb{R}^{2}\setminus P_r}\left(
\frac{\left\vert P_1\right\vert }{\left\vert \operatorname*{dist}\left(
x,P_1\right)  \right\vert ^{2}}\right)  ^{q}dx \lesssim \int\limits_{r} ^{\infty} s^{-2q} \, ds \lesssim r^{1-2q } \, . 
\end{equation}

As for the last two terms on the right side of \eqref{eq:split_Riesz_osc}, we write
\begin{align*}
G(x) = R_{2}\left (g_{\Delta} \otimes \mathbf{1}_{[-1,1)} \right )\left(  x\right)   &  = c \int_{-1}%
^{1}\int_{-1}^{1}\frac{\left(  x_{2}-y_{2}\right)
g_{\Delta} \left(  y_{1}\right)  }{\left[  \left(
	x_{1}-y_{1}\right)  ^{2}+\left ( x_2-y_2 \right )
^{2}\right]  ^{\frac{3}{2}}}dy_{1}dy_2 =\int_{-1}^{1}F(x,y_{1})g_{\Delta}\left(  y_{1}\right)
dy_{1} \, ,
\end{align*}
where
\[
	F(x,y_{1}) \equiv c \int\limits_{-1} ^1 \frac{ x_{2}-y_{2}}{\left[  \left(
	x_{1}-y_{1}\right)  ^{2}+\left ( x_2 - y_2 \right )
^{2}\right]  ^{\frac{3}{2}}}dy_2 \, . 
\]

For the before-last term of \eqref{eq:split_Riesz_osc}, we write using a change of variable and the oddness of the kernel, 
	\begin{align} \label{eq:Fy_odd}
-F(x,y_{1})  &= c \int_{x_{2}+1}^{x_{2}-1}\frac
{t}{\left[  \left(  x_{1}-y_{1}\right)  ^{2}+t^{2}\right]  ^{\frac{n+1}{2}}}dt  = c\int_{|x_{2}+1|}^{|x_{2}-1|}\frac{t}{\left[  \left(  x_{1}%
-y_{1}\right)  ^{2}+t^{2}\right]  ^{\frac{3}{2}}}dt \, .
\end{align}
Let
\begin{align*}
A_{x}  &  \equiv\left\{  y_{1}\in\lbrack-1,1]%
:\left\vert x_{1}-y_{1}\right\vert
>\left\vert 1-x_{2}\right\vert \right\}  \, , \qquad B_{x}   \equiv\left\{  y_{1}\in\lbrack
-1,1]:\left\vert x_{1}-y_{1} \right\vert <\left\vert 1-x_{2}\right\vert \right\}  \, .
\end{align*}
Assume without loss of generality that $|x_{2}-1|\leq|x_{2}+1|$. For $x\in
H_{\delta}^{r}$, we have%
\begin{align*}
&  \int_{-1}^{1}\left\vert F(x,y_{1})\right\vert dy_{1}\lesssim\int_{-1}%
^{1}\left\{  \int_{\min\{|x_{2}+1|,|x_{2}-1|\}}^{\infty
}\frac{t}{\left[  \left(  x_{1}-y_{1}\right)  ^{2}+t^{2}\right]  ^{\frac{3}{2}}%
}dt\right\} dy_{1}\\
&  =\frac{1}{2}\int_{-1}^{1}\frac{1}{\left[  \left(
x_{1}-y_{1}\right)  ^{2}+\left(  1-x_{2}\right)  ^{2}\right]  ^{\frac{1}{2}}%
}dy_{1} \leq \int \left\{  \mathbf{1}_{A_{x}}+\mathbf{1}_{B_{x}}\right\}  \frac{1}{\left[  \left(
x_{1}-y_{1}\right)  ^{2}+\left(  1-x_{2}\right)  ^{2}\right]  ^{\frac{1}{2}}%
}dy_{1}\\
	&  \leq\int_{A_{x}}\frac{1}{\left|  x_{1}-y_{1}\right|} d y_{1}+\int_{B_{x}}\frac{1}{\left\vert
1-x_{2}\right\vert} d y_{1}.
\end{align*}
By a crude estimate the second integral is bounded by
\[
\int_{B_{x}}\frac{1}{\left\vert 1-x_{2}\right\vert }d y_{1}
\leq |B_{x}|\frac{1}{\left\vert 1-x_{2}\right\vert}\lesssim 1 .
\]
Meanwhile the first integral clearly is bounded by
\[
c\ln\frac{c}{\left\vert 1-x_{2}\right\vert } \, . 
\]
Similar estimates hold when $\left\vert x_{2}+1\right\vert <\left\vert
x_{2}-1\right\vert $ and $x\in H_{\delta}^{r}$. Thus, if $x\in H_{\delta}^{r}$, then  
\[
|G(x)|^q \lesssim  1 + \left ( \ln \frac{1}{|1-x_2|} \right )^q+ \left ( \ln \frac{1}{|1+x_2|} \right )^q \lesssim 1 + |1-x_2|^{-\frac 1 2 } + |1+x_2|^{-\frac 1 2 } \, .
\]
Integrating over $H_{\delta} ^{r}$ yields
\begin{equation}\label{eq:halo_bd_G}
\int\limits_{H_{\delta} ^{r}} |G(x)|^q \lesssim  \delta ^{\frac 1 2} r \, .
\end{equation}

For the last term of \eqref{eq:split_Riesz_osc}, we claim that for all for all $x \in Q \setminus H_{\delta} ^{r}$, we have
\begin{equation}\label{eq:bound_F_interior}
|F(x,y_1)| \lesssim r \delta^{-3} \, .
\end{equation}
and the function 
\[
y \mapsto F(x,y_1)
\]
is $2$-piecewise monotone. Given these claims, 
then the alternating series test Lemma \ref{Alt_series_test} and the triangle inequality yield that for $x \in P_r \setminus H_{\delta} ^{r}$, 
\[
|G(x)| \lesssim r \delta^{-3} B \Delta \, , 
\]
and so raising to the $q$-th power and integrating across $ P_r \setminus H_{\delta} ^{r}$ yields
\begin{equation}\label{eq:Q_bd_G}
\int\limits_{Q \setminus H_{\delta} ^{r}} |G(x)|^q \, dx \lesssim  (r \delta^{-3} B \Delta)^q |Q| \lesssim r^{q + 2} \delta^{-3q} (B \Delta)^q \, .
\end{equation}
To obtain \eqref{eq:bound_F_interior}, we use \eqref{eq:Fy_odd} to estimate for $x \in P_r \setminus H_{\delta} ^{r}$, 
\begin{align*}
\left\vert F\left(  x,y_{1}\right)  \right\vert  \lesssim \int_{\min\{ |x_{2} + 1|, |x_{2} - 1| \}}^{\max\{
|x_{2} + 1|, |x_{2} - 1| \}}\frac{t}{\left[  \left(  x_{1}-y_{1}\right)
^{2}+t^{2}\right]  ^{\frac{3}{2}}}dt  \lesssim   \int_{\min\{ |x_{2} + 1|, |x_{2} -
1| \}}^{\max\{ |x_{2} + 1|, |x_{2} - 1| \}}\frac{t}{\delta^{3}}dt \lesssim  r\frac{1}{\delta^{3}} \, ,
\end{align*}
since if $x\in P_r\setminus H_{\delta}^{r}$ then by separation from the boundary we have 
\[
  t \geq \min\{ |x_{2} + 1|, |x_{2} -1| \} >\delta \, .
\]
As for the 2-piecewise monotonicity, by \eqref{eq:Fy_odd} we have
\begin{align*}
-F(x,y_1)=c \frac{1}{\left[  \left(  x_{1}-y_{1}\right)  ^{2}+t^{2}\right]  ^{\frac{1}{2}}%
}\Biggm|_{\left\vert x_{2}+1\right\vert }^{\left\vert x_{2}-1\right\vert } =c \frac{1}{\left[  \left(  x_{1}-y_{1}\right)  ^{2}+\left(  x_{2}-1\right)
^{2}\right]
^{\frac{1}{2}}}-c \frac{1}{\left[  \left(  x_{1}-y_{1}\right)  ^{2}+\left(  x_{2}+1\right)
^{2}\right]
^{\frac{1}{2}}} \, .
\end{align*}
Differentiating the above with respect to $y_1$ yields 
\begin{align*}
 -\frac{\partial}{ \partial y_1} F\left(  x,y_{1}\right)  = &c (x_1 - y_1)\left \{  \frac{1}{\left[  \left(  x_{1}-y_{1}\right)  ^{2}+\left(  x_{2}-1\right)
^{2}\right]
^{\frac{3}{2}}}-\frac{1}{\left[  \left(  x_{1}-y_{1}\right)  ^{2}+\left(  x_{2}+1\right)
^{2}\right]
^{\frac{3}{2}}} \right \} \\
= & \frac{ c (x_1 - y_1)}{\left[  \left(  x_{1}-y_{1}\right)  ^{2}+\left(  x_{2}-1\right)
^{2}\right]
^{\frac{3}{2}} \left[  \left(  x_{1}-y_{1}\right)  ^{2}+\left(  x_{2}+1\right)
^{2}\right]
^{\frac{3}{2}} } \\
&\times \left \{ \left[  \left(  x_{1}-y_{1}\right)  ^{2}+\left(  x_{2}+1\right)
^{2}\right]
^{\frac{3}{2}}-\left[  \left(  x_{1}-y_{1}\right)  ^{2}+\left(  x_{2}-1\right)
^{2}\right]
^{\frac{3}{2}} \right \} \, , 
\end{align*}
which vanishes when 
\begin{align*}
(x_1 - y_1) \left \{ \left[  \left(  x_{1}-y_{1}\right)  ^{2}+\left(  x_{2}+1\right)
^{2}\right]
^{\frac{3}{2}}-\left[  \left(  x_{1}-y_{1}\right)  ^{2}+\left(  x_{2}-1\right)
^{2}\right]
^{\frac{3}{2}} \right \}  
\end{align*}
vanishes. Multiplying and dividing by conjugates, we see this happens precisely when the following vanishes,
\begin{align*}
 (x_1 - y_1)\left \{ \left[  \left(  x_{1}-y_{1}\right)  ^{2}+\left(  x_{2}+1\right)
^{2}\right]
^3-\left[  \left(  x_{1}-y_{1}\right)  ^{2}+\left(  x_{2}-1\right)
^{2}\right]
^3 \right \} \, . 
\end{align*}
 By the identity $a^3 - b^3 = (a-b)(a^2 + ab + b^2)$, the last expression above vanishes if and only if $y_1 = x_1$. Thus $\frac{\partial}{ \partial y_1} F \left(  x,y_{1}\right)$, as a function of $y_1$, vanishes at most once, and so $y_1 \mapsto F(x, y_1)$ is $2$-piecewise monotone. 

Combining \eqref{eq:split_Riesz_osc} with \eqref{eq:tail_bd_G}, \eqref{eq:halo_bd_G} and \eqref{eq:Q_bd_G} yields
\[
\int\limits_{\mathbb{R}^2} |G(x)|^q \, dx \lesssim r^{1-2q}+\delta^{\frac 1 2} r + r^{q+2} \delta^{-3q} (B \Delta)^q \, . 
\]
Setting
\[
\delta \equiv r^{-4q} \, , \qquad r \equiv (B \Delta)^{ -\frac{q}{1+3 q +12 q^2}} 
\]
yields \eqref{eq:alt_series_est_applied}.
\end{proof}

\subsection{Choosing \texorpdfstring{$\{k_j\}$}{the sequence of k's} inductively to make all oscillatory terms small}\label{subsection:choose_k_inductively}
Recall that $\sigma$ and $\omega$, as defined in \eqref{eq:defn_sigma_omega_testing_3}, depend on the fixed parameters $\Gamma$, $p$ and the yet to be chosen sequence $\{k_j\}_{j \geq 0}$, the latter of which is only mildly constrained by Proposition \ref{prop:remodeling_quad_Muck}. We now explain how to choose $\{k_j\}$ inductively so that all oscillatory terms are small. Given $\epsilon > 0$, an integer $t \geq 0$ and a sequence $\{k_j\}_{j \geq 0}$, consider the following statement $\mathbf{U_t (\{k_j\}_{j \geq 0}, \epsilon, \Gamma, p)}$.
	\begin{statement*} If $Q \subset \mathbb{R}^2$ is a square and $t$ is the smallest nonnegative integer satisfying 
    \begin{equation}\label{eq:size_cond_Q}
    \ell(Q) \geq 2^{-(k_0 + k_1 \ldots +k_t)} \, , 
    \end{equation}
    then for all integers $\ell \geq t$ and intervals $R \in \mathcal{D}_{\ell} ([0,1))$, we have \begin{align}
    \label{eq:strong_conv}\frac{1}{\left | Q \right |}\int\limits \left | R_2 \mathbf{1}_{Q} \operatorname{Osc}_{k_{\ell+1}} ^{R, \operatorname{horizontal}} (x) \right |^p dx &< \epsilon \, ,\\
\label{eq:basic_weak_conv}\frac{1}{\left | Q \right |}\int\limits_{3Q} \operatorname{Osc}_{k_{\ell+1}} ^{R, \operatorname{horizontal}} (x) dx &<  \epsilon \, , \\
    \label{eq:weak_conv}\frac{1}{\left | Q \right |}\int\limits_{3Q} \left | R_2 \sigma_t \mathbf{1}_{Q} (x)  \right |^p \operatorname{Osc}_{k_{\ell+1}} ^{R, \operatorname{horizontal}} (x) dx &< \epsilon \text{ and }\frac{1}{\left | Q \right |}\int\limits_{3Q} \left | R_2 \omega_t \mathbf{1}_{Q} (x)  \right |^{p'} \operatorname{Osc}_{k_{\ell+1}} ^{R, \operatorname{horizontal}} (x) dx < \epsilon \, . 
  \end{align}  
	\end{statement*}

\begin{lemma}\label{lem:inductive_convergence} Let $\Gamma > 1$, and $p >2$. There exist nonnegative functions $\{F_j\}_{j \geq 0}$ such that for all $\epsilon > 0$ and all $t \geq 0$, if any sequence $\{k_j\}_{j \geq 0}$ satisfies $k_0 = 0$ and 
	\begin{equation}\label{eq:constraint_ks}
		k_{j} \geq F_t (k_0, \ldots, k_t, \epsilon, \Gamma, p) \text{ for all } j \geq t+1 \, ,
	\end{equation}
	then $\mathbf{U_t (\{k_j\}_{j \geq 0}, \epsilon, \Gamma, p)}$ holds. 
\end{lemma}

\begin{proof} 
Fix $\Gamma > 1$, $p > 2$, let $\epsilon >0$ and fix $t \geq 0$. 
Our goal is to define $\{F_j\}$ for $j \geq t+1$ so that if \eqref{eq:constraint_ks} holds, then  $\mathbf{U_t (\{k_j\}_{j \geq 0}, \epsilon, \Gamma, p)}$. Set
\[
k_0 \equiv 0 \, .
\]
 Because \eqref{eq:constraint_ks} is of the form of a constraint on $k_j$ for $j \geq t+1$, then we may assume we are given 
\[
k_0, k_1, \ldots, k_t \, .
\]
Now assume the assumptions of $\mathbf{U_t (\{k_j\}_{j \geq 0}, \epsilon, \Gamma, p)}$, i.e., 
let
\[
Q = Q_1 \times Q_2 \subset \mathbb{R}^2
\]
be a square satisfying \eqref{eq:size_cond_Q}. Let $\ell \geq t$ and let $R \in \mathcal{D}_{\ell} ([0,1))$. We must now show we can find $\{F_j\}_{j \geq t+1}$ such that \eqref{eq:strong_conv}, \eqref{eq:basic_weak_conv} and  \eqref{eq:weak_conv} all hold. 

	\textbf{Proof of \eqref{eq:strong_conv}.} Recall that for $R \in \mathcal{D}_{\ell} \left ( [0,1) \right )$ we have
\[
  \operatorname{Osc}_{k_{\ell+1}} ^{R, \operatorname{horizontal}} (x_1, x_2) = \operatorname{Osc}_{k_{\ell+1}} ^{R} (x_1) \, .
\]
	Let 
    \[
    P_1 \equiv \bigcup\limits_{d=1}^M I_d
    \]be the largest union of consecutive dyadic intervals of sidelength $ 2^{-(k_0 + k_1 + \ldots + k_{\ell+1})+1}$
    contained in $Q$ such that $\operatorname{Osc}_{k_{\ell+1}} ^{R} $ is not identically $0$ on $I_1$ or $I_M$; if no such collection $\{I_d\}_{d=1}^M$ exists, then set $P_1 \equiv \emptyset$. Then
\begin{align*}
  &\frac{1}{\left | Q \right |} \int  \left | R_2 \mathbf{1}_{Q} \operatorname{Osc}_{k_{\ell+1}} ^{R, \operatorname{horizontal}} (x) \right |^p dx \\
  \lesssim &\frac{1}{\left | Q \right |} \int  \left | R_2 \mathbf{1}_{P_1 \times Q_2} \operatorname{Osc}_{k_{\ell+1}} ^{R, \operatorname{horizontal}} (x) \right |^p dx + \frac{1}{\left | Q \right |} \int  \left | R_2 \mathbf{1}_{(Q_1 \setminus P_1) \times Q_2} \operatorname{Osc}_{k_{\ell+1}} ^{R, \operatorname{horizontal}} (x) \right |^p dx \equiv I  + II\, . 
\end{align*}
	Since $R_2$ is bounded on  $L^p (dx)$ and  $\left | \operatorname{Osc}_{k_{\ell+1}} ^{R, \operatorname{horizontal}}  \right | \leq 1$, we then have  
\[
  II \lesssim \frac{1}{\left | Q \right |} \int  \left | \operatorname{Osc}_{k_{\ell+1}} ^{R, \operatorname{horizontal}} \mathbf{1}_{(Q_1 \setminus P_1) \times Q_2} \right |^p dx = \frac{\left | \left ( Q_1 \setminus P_1 \right ) \cap \operatorname{supp} \left (\operatorname{Osc}_{k_{\ell+1}} ^{R} \right ) \right |}{|Q_1|} \lesssim \frac{2^{-(k_0 + \ldots + k_{\ell+1})}}{2^{-(k_0 + \ldots +k_{t})}} \leq 2^{-k_{t+1}} \, .  
\]

As for $I$, let 
\[
\psi(x) = (\psi_1 (x_1), \psi_2 (x_2))
\]
be the unique composition of dilations and translations mapping $[-1,1)^2$ to $Q$. A change of variables yields 
\begin{align*}
 I =  \int  \left | R_2 \left \{ \left ( \mathbf{1}_{P_1 \times Q_2} \circ \psi \right ) \left ( \operatorname{Osc}_{k_{\ell+1}} ^{R, \operatorname{horizontal}} \circ \psi \right )\right \} (x) \right |^p dx &= \int  \left | R_2 \left \{ \mathbf{1}_{[a,b) \times [-1,1)}  \left ( \operatorname{Osc}_{k_{\ell+1}} ^{R, \operatorname{horizontal}} \circ \psi \right )\right \} (x) \right |^p dx \,  \\
  &= \int  \left | R_2 \left \{ \mathbf{1}_{[a,b)}  \left (  \operatorname{Osc}_{k_{\ell+1}} ^{R} \circ \psi_1 \right ) \otimes \mathbf{1}_{[-1,1)} \right \} (x) \right |^p dx \, ,
\end{align*}
where $[a,b) \subseteq [-1,1)$.
Note that $\mathbf{1}_{[a,b)} \operatorname{Osc}_{k_{\ell+1}} ^{R} \circ \psi_1$ is a sum of  $B$ square waves with pairwise disjoint support, all supported on $[-1,1)$, and with step size $\Delta$, where 
\[
B \lesssim \ell (Q) 2^{k_0 + k_1 + \ldots + k_{\ell}} \, , \qquad \Delta \lesssim \frac{1}{\ell (Q)} 2^{-( k_1 + \ldots + k_{\ell+1})} \, .
\]
By Lemma \ref{reduction_extra},  
\[
I \lesssim (B\Delta)^{ \frac{p(2p-1)}{1+3p+12p^2}} \lesssim 2^{- \frac{p(2p-1)}{1+3p+12p^2} k_{\ell+1}  } \, .
\]
Altogether, \eqref{eq:strong_conv} holds  
if 
\[
k_{\ell+1} \geq F_{t} \left (k_1, \ldots, k_t, \epsilon, \Gamma, p \right ) \, .
\]

\textbf{Proof of  \eqref{eq:basic_weak_conv}.}
By a change of variables with the map $\psi$ again, the left side of \eqref{eq:basic_weak_conv} equals
  \[
	  \int \mathbf{1}_{[-3,3)^2} \left ( x \right ) \operatorname{Osc}_{k_{\ell+1}} ^{R, \operatorname{horizontal}} \circ \psi (x) \, dx \, .
  \]
	Let $\mathcal{D}'$ be the preimage of the standard grid of dyadic squares $\mathcal{D}$ under $\psi$. Then there exists $ n(\epsilon)$ such that if $k'$ satisfies 
    \begin{equation}\label{eq:req_k'}
    \frac{1}{\ell (Q)} 2^{-k'} \leq 2^{-n (\epsilon)} \, ,
    \end{equation}
    then there exists a $\mathcal{D}'$-dyadic step function $f$ with step size $\frac{1}{\ell (Q)} 2^{-k'}$ for which 
	\begin{equation}\label{eq:approx_indicator}
		\int \left | \mathbf{1}_{[-3,3)^2}   -  f \right |  < \epsilon \, . 
	\end{equation} 
    In particular, by \eqref{eq:size_cond_Q}, then  \eqref{eq:req_k'} will be satisfied if we set 
	\[
		k' \equiv k_0 + \ldots + k_t + n (\epsilon) \, .
	\]
    But $\operatorname{Osc}_{k_{\ell+1}} ^{R, \operatorname{horizontal}} \circ \psi$ is a square wave adapted to the grid$\mathcal{D}'$ with step size 
    \[
    \frac{1}{\ell (Q)} 2^{-(k_0 + \ldots +k_{\ell+1})} \leq 2^{-(k_{t+1} + \ldots + k_{\ell+1})} \, .
    \]
    This function will have mean $0$ on every step of $f$ if  
    \begin{equation}\label{eq:req_kt_cancellation}
    k_{t+1} \geq k'+1 = k_0 + \ldots + k_t + n(\epsilon) \, .
    \end{equation}
    Thus for such an integer $k_{t+1}$, we have 
    \begin{equation}\label{eq:perfect_cancellation_f_osc}
    \int f \left ( \operatorname{Osc}_{k_{\ell+1}} ^{R, \operatorname{horizontal}} \circ \psi \right ) dx = 0 \, .
    \end{equation}
    Combining \eqref{eq:perfect_cancellation_f_osc} and \eqref{eq:approx_indicator} and using that $|\operatorname{Osc}_{k_{\ell+1}} ^{R, \operatorname{horizontal}} \circ \psi| \leq 1$ yields the left side of \eqref{eq:basic_weak_conv} equals 
\[
	\left | \int \mathbf{1}_{[-3,3)^2} \left ( x \right ) \operatorname{Osc}_{k_{\ell+1}} ^{R, \operatorname{horizontal}} \circ \psi (x) \, dx \right | \leq  \int\limits \left | \mathbf{1}_{[-3,3)^2} \left (x\right )  - f(x) \right | dx < \epsilon \, ,
  \]
so long as $k_{t+1}$ satisfies Requirement \eqref{eq:req_kt_cancellation}.

  \textbf{Proof of  \eqref{eq:weak_conv}.}
By a change of variables using $\psi$, the left side of the first part of \eqref{eq:weak_conv} equals
  \[
    \int\limits_{[-3,3)^2} \left | R_2 \left ( \sigma_t \circ \psi \, \mathbf{1}_{[-1,1)^2} \right ) (x)  \right |^p \operatorname{Osc}_{k_{\ell+1}} ^{R, \operatorname{horizontal}} \circ \psi (x) dx \, .
  \]
    Since $\sigma_t$ is a $\mathcal{D}$-dyadic step function with step size
    \[
   2^{-(k_0 + \ldots + k_t)} \, , 
    \]
    then $\sigma_t \circ \psi$ is a $\mathcal{D}'$-dyadic step function with step size
    \[
   \frac{1}{\ell (Q)}  2^{-(k_0 + \ldots + k_t)} \leq 1 
    \]
   by \eqref{eq:size_cond_Q}. Thus $\sigma_t \circ \psi \mathbf{1}_{[-1,1)}$ is a linear combination of indicators of $M$ rectangles, whose $e_1$-sidelengths and $M$ are both bounded by a  function of  $k_0, \ldots, k_t$. Thus by Lemma \ref{lem:approx_dependence}, there exists $ m = m(k_0, \ldots, k_t, \Gamma, \epsilon, p)$ such that, if an integer $k'$ satisfies 
	\begin{equation}\label{eq:constraint_k'_2}
		\frac{1}{\ell (Q)} 2^{-k'} \leq 2^{-m} \, ,
	\end{equation}
    then there exists a $\mathcal{D}'$-dyadic step function $f$ with step size $\frac{1}{\ell (Q)} 2^{-k'}$ for which 
	\begin{equation}\label{eq:step_approx_converence}
    \int \left | \mathbf{1}_{[-3,3)^2 }\left | R_2 \left ( \sigma_t \circ \psi \, \mathbf{1}_{[-1,1)^2} \right ) (x)  \right |^p - f (x) \right |  < \epsilon \, . 
	\end{equation}
	By the size condition \eqref{eq:size_cond_Q}, then \eqref{eq:constraint_k'_2} will hold if 
    \[
    k' = 2^{-(k_0 + \ldots  + k_t  + m)}  \, .
    \]
  Then for 
  \[
  k_{t+1} \geq k'+1 = 2^{-(k_0 + \ldots  + k_t  + m)} + 1 \, , 
  \]
  we get $\operatorname{Osc}_{k_{\ell+1}} ^{R, \operatorname{horizontal}} \circ \psi$ has mean $0$ on each step of $f$ and so 
  \begin{equation}\label{eq:perfect_cancellation_f_2}
  \int |f| \left ( \operatorname{Osc}_{k_{\ell+1}} ^{R, \operatorname{horizontal}} \circ \psi \right ) dx = 0 \, ,
  \end{equation}
  Combining \eqref{eq:step_approx_converence} and \eqref{eq:perfect_cancellation_f_2} with the bound $|\operatorname{Osc}_{k_{\ell+1}} ^{R, \operatorname{horizontal}} \circ \psi| \leq 1$ yields 
\[
  \left | \int\limits_{[-3,3)^2} \left | R_2 \mathbf{1}_{[0,1)^2} (x)  \right |^p \operatorname{Osc}_{k_{\ell+1}} ^{R, \operatorname{horizontal}} \circ \psi (x) dx \right | \leq  \int\limits \left | \mathbf{1}_{[-3,3)^2} \left | R_2 \mathbf{1}_{[0,1)^2} (x)  \right |^p  - f(x) \right | \, dx < \epsilon \, . \qedhere
  \]
\end{proof}

By inductively choosing $\{k_t\}$ to satisfy our constraints, we have the following corollary.
\begin{corollary}\label{remark:choosing_sequence_1_by_1}
Let $\epsilon >0$. There exists a sequence $\{k_j\}_{ j \geq 0}$ with $k_0 =0$ such that  $\mathbf{U_t \left ( \left \{ k_j\right \}_{j \geq 0}, \epsilon, \Gamma, p \right ) }$ holds for all $t \geq 0$ and
\begin{equation} \label{eq:Ap_quad_big3}
A_{p} ^{\ell^2, \operatorname{local}, \mathcal{D}([0,1))} (\hat{\sigma}, \hat{\omega}) > \Gamma \, .
\end{equation}
	\end{corollary}
\begin{proof}
Set $k_0 = 0$. By \eqref{eq:Ap_quad_big2}, we have
\eqref{eq:Ap_quad_big3} holds if for all $j \geq 1$, we have 
\begin{equation}\label{eq:k_req_0}
k_j \geq F_j ^0 (\Gamma, p, \epsilon) \, .
\end{equation}
And by Lemma \ref{lem:inductive_convergence}, we have $\mathbf{U_t \left ( \left \{ k_j\right \}_{j \geq 0}, \epsilon, \Gamma, p \right ) }$ holds for all $t \geq 0$ if for all $t \geq 0 $, the inequality
\begin{equation} \label{eq:k_req_1}
k_{\ell} \geq F_t ^1 (k_0, \ldots, k_t, \Gamma, p, \epsilon)
\end{equation}
holds for all $\ell \geq t+1$. Combining constraints \eqref{eq:k_req_0} and \eqref{eq:k_req_1}, and suppressing the dependence on $\epsilon$, $\Gamma$ and $p$, we have the requirement that for all $t \geq 0$, the inequality
\begin{equation} \label{eq:k_req_2}
k_{\ell} \geq F_t  (k_0, \ldots, k_t)
\end{equation}
holds for all $\ell \geq t+1$.
For each $t\geq0$, define
    \[
		k_{t+1} \equiv \sup \left \{ F_0 (k_0), F_1 (k_0, k_1) , \ldots ,  F_t (k_0, \ldots, k_t) \right \} \, .
	\]
    This will then satisfy \eqref{eq:k_req_2}. 
\end{proof}

\subsection{Triple testing holds}\label{subsection:triple_testing}
Recall that at the beginning of this section, we fixed $\Gamma >1$ and $p >2$, and given any sequence $\{k_j\}_{j \geq 0}$, we defined $\sigma$ and $\omega$ as in  \eqref{eq:defn_sigma_omega_testing}. We now complete the proof of Theorem \ref{thm:counter_example} by showing there exists a choice of sequence $\{k_j\}$ for which \eqref{eq:triple_testing_bounded} holds.
We will show that for every $\epsilon \in (0,1)$, there is a choice of sequence $\{ k_j \}_{j \geq 0}$ for which we have
    \begin{equation}\label{eq:goal_bd_testing}
    \mathfrak{T}_{R_2, p} ^{\operatorname{triple}} (\sigma, \omega) + \mathfrak{T}_{R_2, p'} ^{\operatorname{triple}} (\omega, \sigma) \lesssim A_p (\sigma, \omega) + C(\Gamma, p) \epsilon \, ,
    \end{equation}
where what follows, $C(\Gamma, p)$ denotes a constant depending on $\Gamma$ and $p$ which may change instance to instance.
    Taking $\epsilon \to 0$ and then using \eqref{eq:Ap_final}will then yield \eqref{eq:triple_testing_bounded}. We focus on $\mathfrak{T}_{R_2, p'} ^{\operatorname{triple}} (\omega, \sigma)$, since $\mathfrak{T}_{R_2, p} ^{\operatorname{triple}} (\sigma, \omega)$ will be bounded similarly. 
	
    Let $\epsilon > 0$. By Corollary \ref{remark:choosing_sequence_1_by_1} , there is a choice of sequence $\{k_j\}_{j \geq 0}$ with $k_0 = 0$ for which \eqref{eq:Ap_quad_big3} holds as well as the statement $\mathbf{U_t \left ( \{k_j\}, \epsilon, \Gamma, p \right )}$ for all $t \geq 0$. 
By \eqref{eq:disc_sigma} and \eqref{eq:disc_omega}, the functions 
\[
  \eta_{\ell} \equiv \sigma_{\ell+1} - \sigma_{\ell} \, , \qquad \delta_{\ell} \equiv \omega_{\ell+1} - \omega_{\ell} \, ,  
\]
 are of the form 
\begin{equation}\label{eq:sum_type_osc}
  \sum\limits_{ R \in \mathcal{D}_{\ell} ([0,1)) }  c_R \operatorname{Osc}_{k_{\ell+1}} ^{R, \operatorname{horizontal}} \, .
\end{equation}
By \eqref{eq:defn_sigma_omega_testing_3}, this is nonzero only if $\ell \leq m$. Thus there are at most $C(\Gamma, p)$ many nonzero coefficients $c_R$.  

We must show that for every square $Q$ we have
\begin{align}
   \frac{1}{|Q|_{\omega}} \int\limits_{3Q} \left |R_2 \omega \mathbf{1}_{Q} \right |^{p'} d \sigma &\lesssim A_{p'} (\omega, \sigma) ^{p'} + C(\Gamma, p) \epsilon\, . \label{eq:backwards_testing_small} 
\end{align}
	Given a square $Q$ in $\mathbb{R}^2$, let $t=t(Q)$ be the smallest nonnegative integer such that 
    \[
    2^{-(k_0 + k_{1}+k_{2}+...+k_{t})}\leq\ell\left(  Q\right)  \, .
    \]
    We claim that for all squares $Q \subset \mathbb{R}^2$, we have
\begin{align}\label{eq:estimate_testing_1}
&  \frac{1}{\left\vert Q\right\vert _{\omega}}\int_{3Q}\left\vert R_{2}%
\mathbf{1}_{Q}\omega\left(  x\right)  \right\vert ^{p'}\sigma\left(  x\right)
dx\lesssim \frac{1}{\left\vert Q\right\vert _{\omega}}\int_{3Q}\left\vert R_{2}%
\mathbf{1}_{Q}\left(  \omega_{t}\right)  \left(  x\right)  \right\vert
^{p'}\sigma\left(  x\right)  dx+C(\Gamma, p) \epsilon  \,  .
\end{align}
	Indeed, using that $\tilde{\sigma}$ and $\tilde{\omega}$ are dyadic step functions and are dyadically doubling on $[0,1)$, then we have 
    \begin{equation}\label{eq:infinity_norm_weights}
    \sigma\left(  x\right)  \leq\left\Vert \widetilde{\sigma} \right\Vert _{\infty} = C(\Gamma, p) \, , \qquad \omega(x)^{-1} \leq \| \tilde{\omega}^{-1}\|_{\infty} =C(\Gamma, p) \, .
    \end{equation}
    Since $m$ and the finitely many coefficients $c_R$ depend only on $\Gamma$ and $p$, then by Lemma \ref{lem:inductive_convergence}, namely \eqref{eq:strong_conv}, there is a choice of $\{k_j\}$ for which we have
\begin{align*}
\frac{1}{\left\vert Q\right\vert _{\omega}}\int_{3 Q}\left\vert R_{2}\mathbf{1}%
_{Q}\left(  \sum_{s=t}^{m-1}\delta_{s}\right)  \left(  x\right)  \right\vert
^{p'}\sigma\left(  x\right)  dx    \leq C(\Gamma,p) \frac
{1}{\left\vert Q\right\vert}\int_{3 Q}\left\vert R_{2}\mathbf{1}_{Q}\left(
\sum_{s=t}^{m}\delta_{s}\right)  \left(  x\right)  \right\vert ^{p'}dx \leq C(\Gamma,p)\epsilon  \, .
\end{align*}

We decompose the first term on the right of \eqref{eq:estimate_testing_1} further as 
	\begin{align}\label{eq:estimate_testing_2}
		\frac{1}{\left\vert Q\right\vert _{\omega}}\int\limits_{3 Q} \left | R_{2}\mathbf{1}_{Q}\omega_{t} \right |^{p'}\sigma dx &= \frac{1}{\left\vert Q\right\vert _{\omega}}\int
\limits_{3 Q} \left | R_{2}\mathbf{1}_{Q}\omega_{t} \right |^{p'}%
\sigma_{t}dx+ \frac{1}{\left\vert Q\right\vert _{\omega}}\int\limits_{3 Q} \left | R_{2}\mathbf{1}_{Q}%
\omega_{t} \right |^{p'}\left ( \sigma - \sigma_t \right) dx \, .
	\end{align}
	The last term equals
	\begin{align}\label{eq:diff_sigma}
	 \frac{1}{\left\vert Q\right\vert _{\omega}}\int\limits_{3 Q} \left | R_{2}\mathbf{1}_{Q} \omega_{t} \right |^{p'}\left ( \sigma - \sigma_t \right) dx  = \sum\limits_{\ell = t}^{m-1}  \frac{1}{\left\vert Q\right\vert _{\omega}}\int\limits_{3 Q} \left | R_{2}\mathbf{1}_{Q} \omega_{t} \right |^{p'} \eta_{\ell} dx \, .
	\end{align}
	Because there are at most $m = m(\Gamma, p)$  many functions $\eta_{\ell}$, each of which decomposes as a  sum of the form \eqref{eq:sum_type_osc} with at most $C(\Gamma, p)$ many terms, each each $c_R$ depends on $\Gamma,p$, then 
    by Corollary \ref{remark:choosing_sequence_1_by_1}, and in particular the second part of  \eqref{eq:weak_conv}, we have \eqref{eq:diff_sigma} is at most $C(\Gamma, p) \epsilon$.
	
    To estimate the first term on the right of \eqref{eq:estimate_testing_2}, given 
\[
Q = Q_1 \times Q_2 \, , 
\]
let 
$Q_1 ^* \equiv \left ( Q_1 \right)^{\ast}$ be as in Lemma \ref{lem:remodeling_loc_cst} so that on $3Q$, we have the pointwise estimates 
\begin{equation}\label{eq:weight_values_to_avg}
\sigma_{t} \approx E_{\supr (Q_1^*)} \tilde{\sigma} \, , \qquad \omega_{t} \approx E_{\supr (Q_1^*)} \tilde{\omega} \, . 
\end{equation}
Then applying the
pointwise estimate \eqref{eq:weight_values_to_avg} to $\sigma_{t}$, followed by the bound 
\[
\Vert R_{2}
\mathbf{1}_{Q} \omega_{t} \Vert_{L^{p'}} \lesssim \Vert\mathbf{1}_{Q}
\omega_{t} \Vert_{L^{p'}}
\]
which follows from $L^{p'}$ boundedness of $R_{2}$, we get
\begin{align}\label{eq:Riesz_loc_cst}
	\frac{1}{\left | Q \right |_{\omega} }\int\limits_{3 Q}\left | R_{2}\mathbf{1}_{Q}\omega_{t} \right |^{p'}%
\sigma_{t}dx \approx \left(  E_{\supr(Q^{\ast} _1)}\widetilde{\sigma}\right) \frac{1}{\left | Q \right |_{\omega} }  \int\limits_{3 Q} \left | R_{2}\mathbf{1}_{Q}\omega_{t}\right |^{p'}dx
&  \lesssim \left(  E_{\supr(Q^{\ast} _1)} \widetilde{\sigma} \right) \frac{1}{\left | Q \right |_{\omega} } \int\limits_{Q}\left | \omega_{t} \right |^{p'}dx
\end{align}
And then applying the
pointwise estimate \eqref{eq:weight_values_to_avg} to $\omega_{t}$, and then using that for all dyadic intervals $I \in \mathcal{K}_t ([0,1))$, we have $E_{\supr (I)} \tilde{\omega} = E_{I} \hat{\omega}$, we then get \eqref{eq:Riesz_loc_cst} is 
\begin{align*}
	  \approx \frac{1}{|Q|_{\omega}} \left(  E_{\supr(Q^{\ast} _1)}\widetilde{\sigma} \right)  \left ( E_{\supr(Q^{\ast} _1)}\widetilde{\omega} \right )^{p'} \approx \left(  E_{\supr(Q^{\ast} _1)}\widetilde{\sigma} \right)  \left ( E_{\supr(Q^{\ast} _1)}\widetilde{\omega} \right )^{p'-1} \frac{|Q|_{\omega_t}}{ |Q|_{\omega}}\, ,
\end{align*}
Regarding the ratio in the last term, we claim that
\begin{equation}\label{eq:ratio_bd}
\frac{|Q|_{\omega_t}}{ |Q|_{\omega}} \lesssim 1 + C(\Gamma, p) \epsilon \, , 
\end{equation}
which would then yield
\[
\frac{1}{\left | Q \right |_{\omega} }\int\limits_{3 Q}\left | R_{2}\mathbf{1}_{Q}\omega_{t} \right |^{p'}%
\sigma_{t}dx \lesssim 1 + C(\Gamma, p) \epsilon \, .
\]
To see the claim, by doubling for $\omega$ and $\omega_t$ we have 
\begin{equation}\label{eq:ratio}
	\frac{|Q|_{\omega_t}}{ |Q|_{\omega}} \approx \frac{|3Q|_{\omega_t}}{ |3Q|_{\omega}}  = 1 +  \frac{1}{ |3Q|_{\omega}} \int\limits_{3Q} \omega_t - \omega \, dx = 1 + \frac{1}{ E_{3 Q} \omega} \left ( \frac{1}{|3 Q|} \int\limits_{3Q} \omega_t - \omega \, dx \right ) \, .
\end{equation}
Using \eqref{eq:infinity_norm_weights} to bound $\left ( E_Q \omega \right )^{-1} \leq  C(\Gamma, p)$, and then noting by \eqref{eq:sum_type_osc} that  
\[
\omega_t - \omega = -\sum\limits_{\ell=t}^{m-1} \delta_{\ell} = \sum\limits_{\ell=t}^{m-1} \sum\limits_{R \in \mathcal{D}_{\ell} ([0,1))} c_R  \operatorname{Osc}_{k_{\ell+1}} ^{R, \operatorname{horizontal}} \, ,
\]
 then Lemma \ref{lem:inductive_convergence}, namely \eqref{eq:basic_weak_conv}, shows that the last term of \eqref{eq:ratio} is at most $ C(\Gamma, p)\epsilon$.  

Thus altogether we get
\[
	\sup\limits_Q \frac{1}{\left | Q \right |_{\omega} }\int\limits_{3 Q}\left | R_{2}\mathbf{1}_{Q}\omega \right |^{p'}%
	\sigma dx \lesssim \sup\limits_{Q} \left(  E_{\supr(Q^{\ast} _1)}\widetilde{\sigma} \right)  (E_{\supr(Q^{\ast} _1)}\widetilde{\omega})^{p'-1} + C(\Gamma, p) \epsilon 
    \]
    \[\leq  A_{p'} (\tilde{\omega},\tilde{\sigma})^{p'}  + C(\Gamma, p) \epsilon 
	=  A_{p'} (\omega,\sigma)^{p'} + C(\Gamma, p) \epsilon\, . 
\]
Thus \eqref{eq:goal_bd_testing} holds, completing the proof.

\appendix

\section{\texorpdfstring{$A_p$}{Ap} implies two-tailed \texorpdfstring{$\mathcal{A}_p$}{Ap} for reverse doubling measures}\label{section:rev_doub_tailed_Ap}

A measure $\mu$ on $\mathbb{R}^n$ is \emph{reverse doubling} if there exists a constant $c <1 $ such that for all cubes $Q$, we have
\[
   \left |Q\right |_{\mu} \leq c \left |2 Q\right |_{\mu} \, . 
\]
Note that every doubling measure $\mu$ is reverse doubling, but the converse is false: take $\mu$ equals for instance the Hausdorff measure $dx_1 \times \delta_0$ of $\mathbb{R} \times \{0\}$ on $\mathbb{R}^2$, or the measure $\mu = \mathbf{1}_{[0, \infty)} dx$ on $\mathbb{R}$. By adapting the proof of \cite[Theorem 4]{Gr} to the two-tailed $\mathcal{A}_p$ condition \eqref{Muck conditions}, we get the following result, whose proof we include for the convenience of the reader. 

\begin{proposition}[{\cite[Theorem 4]{Gr}}]\label{prop:reverse_doubling_tailed_Ap}
  Let $1<p<\infty$. If $\sigma$ and $\omega$ are reverse doubling measures on $\mathbb{R}^n$, then 
  \[
    \mathcal{A}_{p} (\sigma, \omega) \lesssim A_p (\sigma, \omega) \, .
  \]
\end{proposition}
\begin{proof}
	For a cube $Q$, we write
	\begin{align*}
		&\frac{1}{|Q|} \int \frac{|Q|^p}{\left [\ell(Q)+ \operatorname{dist}(x,Q) \right ]^{np}} d \omega (x) \\
		= &\frac{1}{|Q|} \int\limits_{2 Q} \frac{|Q|^p}{\left [\ell(Q)+ \operatorname{dist}(x,Q) \right ]^{np}} d \omega (x) + \frac{1}{|Q|} \sum\limits_{\ell=1}^{\infty} \int\limits_{2^{\ell+1} Q \setminus 2^{\ell} Q} \frac{|Q|^p}{\left [\ell(Q)+ \operatorname{dist}(x,Q) \right ]^{np}} d \omega (x) \\ 
		\lesssim &E_{2 Q} \omega  + \frac{1}{|Q|} \sum\limits_{\ell=1}^{\infty} 2^{-\ell n p} \left | 2^{\ell+1} Q \right |_{\omega} \, .
	\end{align*}
      Therefore, using that $(\sum_k a_k)^{1/p}\leq \sum_k a_k^{1/p}$ twice, we get
	\[
		\left( \frac{1}{|Q|} \int \frac{|Q|^p}{[\ell(Q)+\operatorname{dist}(x,Q)]^{np}} d\omega(x)\right)^{1/p} \lesssim (E_{2 Q} \omega)^{1/p} + \frac{1}{|Q|^{1/p}}\sum_{\ell=1}^\infty 2^{-\ell n}|2^{\ell+1}Q|_\omega^{1/p}.
	\]
		And with obvious changes for $\sigma$ and $p'$ we obtain
	\begin{align*}
		&\left( \frac{1}{|Q|} \int \frac{|Q|^p}{[\ell(Q)+\operatorname{dist}(x,Q)]^{np}} d\omega(x)\right)^{1/p} \left( \frac{1}{|Q|} \int \frac{|Q|^{p'}}{[\ell(Q)+\operatorname{dist}(x,Q)]^{np'}} d\sigma(x)\right)^{1/p'}\\
		\lesssim & (E_{2 Q} \omega)^{1/p}(E_{2 Q} \sigma)^{1/p'} + \sum_{m=1}^\infty (E_{2 Q} \omega)^{1/p} \frac{2^{-mn}|2^{m+1}Q|_\sigma^{1/p'}}{|Q|^{1/p'}}+ \sum_{\ell=1}^\infty (E_{ 2 Q}\sigma)^{1/p'} \frac{2^{-\ell n}|2^{\ell+1}Q|_\omega^{1/p}}{|Q|^{1/p}}\\
		\quad  & + \sum_{\ell=1}^\infty  \sum_{m=1}^\infty  \frac{2^{-mn}|2^{m+1}Q|_\sigma^{1/p'}}{|Q|^{1/p'}} \frac{2^{-\ell n}|2^{\ell+1}Q|_\omega^{1/p}}{|Q|^{1/p}}\\
		=: & A+B+C+D.
	\end{align*}
	We first handle $D$, the most difficult term.
	
	First, we break up the double sum into the regimes where $m\geq \ell$ and $\ell\geq m$, i.e. 
	\[
	D=\sum_{\ell=1}^\infty\sum_{m=1}^\infty [\ldots] \leq  \sum_{\ell=1}^\infty \sum_{m=1}^\ell [\ldots]+ \sum_{m=1}^\infty \sum_{\ell=1}^m[\ldots]=:I+II.
	\]

	By reverse doubling, there exist $\delta_\sigma, \delta_\omega>0$ such that 
	\[
	|2Q|_\sigma \geq (1+\delta_\sigma)|Q|_\sigma, \qquad |2Q|_\omega \geq (1+\delta_\omega)|Q|_\omega,
	\]
	In particular, for $\ell\geq m$, we have $|2^{m+1}Q|_\omega \leq (1+\delta_\omega)^{-(\ell-m)}|2^{\ell+1}Q|_\omega$. Moreover $|2^{\ell+1}Q|=2^{(\ell+1)n}|Q|$, and so 
	\begin{align*}
	I & \leq \sum_{\ell=1}^\infty \sum_{m=1}^\ell (1+\delta_\omega)^{-(\ell-m)/p}2^{-mn+n}\dfrac{|2^{\ell+1}Q|_\sigma^{1/p'}|2^{\ell+1}Q|_\omega^{1/p}}{|2^{\ell+1}Q|}\\
	& \lesssim A_p(\sigma,\omega) \sum_{\ell=1}^\infty(1+\delta_\omega)^{-\ell/p} \sum_{m=1}^\ell\left( \frac{(1+\delta_\omega)^{1/p}}{2^n}\right)^m.
	\end{align*}
	Since the inner sum is a geometric series in $m$, it's controlled by either its first or last term, and in either case we get a convergent series in $\ell$ on the outside. 
	
	Term $II$ is handled the same, using now the reverse doubling for $\sigma$ instead.
	
	Handling $B$ and $C$ is easier, and we only need to do one of them by symmetry, so we do $B$: again by reverse doubling 
	\[
	|Q|_\omega \leq (1+\delta_\omega)^{-m}|2^{m+1}Q|_\omega,
	\]
	and therefore 
	\[
	E_Q\omega \leq (1+\delta_\omega)^{-m} \frac{|2^{m+1}Q|_\omega}{|Q|}.
	\]
	We conclude 
	\[
	B= \sum_{m=1}^\infty (E_Q\omega)^{1/p} \frac{2^{-mn}|2^{m+1}Q|_\sigma^{1/p'}}{|Q|^{1/p'}} \lesssim \sum_{m=1}^\infty (1+\delta_\omega)^{-m/p} \frac{|2^{m+1}Q|_\omega^{1/p}|2^{m+1}Q|_\sigma^{1/p'}}{|2^{m+1}Q|} \lesssim A_p(\sigma, \omega).
	\]
And clearly $A \lesssim A_p(\sigma, \omega)$
\end{proof}

\section{The \texorpdfstring{$A_p$}{Ap} condition does not imply the two-weight norm inequality for the Maximal function for doubling measures}\label{section:counter_example_maximal_fn}

Consider the uncentered Lebesgue Hardy-Littlewood maximal function on a collection of cubes $\mathcal{S}$, 
\[
  \mathcal{M}^{\mathcal{S}} f (x) \equiv \sup\limits_{J \in \mathcal{S} ~:~ x \in J} \frac{1}{\left | J \right |} \int\limits_{J} \left | f(y) \right | d y  
\]
Given a measure $\mu$, we also define
\[
\mathcal{M}_{\mu} ^{\mathcal{S}} f (x) \equiv \sup\limits_{J \in \mathcal{S} ~:~ x \in J} \frac{1}{\left | J \right |} \int\limits_{J} \left | f(y) \right | d \mu (y) \, . 
\]
If $\mathcal{S}$ is not specified, then we take $\mathcal{S}$ to be the set of all cubes.

Recall that if $w \in A_p$, then 
\[
\mathcal{M}_{} : L^p (w) \to L^p (w) \, .
\]
 Or equivalently,
 \[
 \mathcal{M}_{\sigma} :L^p (\sigma) \to L^p(\omega)
 \]
 if $A_p (\sigma, \omega) < \infty$, where 
 \[
 d \sigma (x) \equiv w (x) ^{-\frac{1}{1-p}} dx \, , \qquad d \omega (x) \equiv w (x) dx \, .
 \]
 However, this theorem fails  for general weight pairs $(\sigma, \omega)$, i.e., there are weight pairs $\left (\sigma, \omega \right )$, not both doubling, for which 
 \[
 A_p (\sigma, \omega) < \infty \, , \qquad \mathfrak{N}_{\mathcal{M},p} \left (\sigma, \omega \right ) = + \infty \, ; 
 \]
 see \cite{Muc}. In \cite{Saw1}, one of the first two-weight $T1$ theorems was also shown, namely
\[
\mathfrak{N}_{\mathcal{M},p} \left (\sigma, \omega \right ) \approx \mathfrak{T}_{\mathcal{M},p} ^{\operatorname{local}} \left (\sigma, \omega \right ) \, .
\]
To this day, this result, and its slight improvement in \cite{HyLiSa}, is the only known \emph{characterization} of the two-weight norm inequality for $\mathcal{M}$ and general weight pairs $(\sigma, \omega)$ in terms of classical quantities.

However, what if we restrict ourselves to pairs $\left (\sigma, \omega \right )$ of doubling measures? Is there a characterization of the norm inequality in terms of $A_p (\sigma, \omega)$ alone? The tools developed in \cite{KaTr} provide a nearly immediate negative answer to this question, even for absolutely continuous doubling measures.

\begin{proposition}
  Let $1 < p < \infty$, let $\tau \in \left (0 , \frac{1}{2} \right )$ and let $\Gamma > 1$. There exist $\tau$-doubling measures $\hat{\sigma}$ and $\hat{\omega}$ on $\mathbb{R}$ such that $A_p \left ( \sigma, \omega \right) \lesssim 1$ and yet
  \[
  \mathfrak{N}_{\mathcal{M} ^{\mathcal{D} \left ( [0,1) \right )},p} \left (\sigma, \omega \right ) > \Gamma \, .  \]
\end{proposition}
\begin{proof}[Sketch of proof]
  Begin with weights $\sigma$ and $\omega$ on $[0,1)$ such that
  \[
  \mathfrak{N}_{\mathcal{M} ^{\mathcal{D} \left ( [0,1) \right )},p} \left (\sigma, \omega \right ) > \Gamma \, ,  \]
	and whose doubling ratio on $\mathcal{D} \left ( [0,1) \right )$ at most $C(\Gamma)$: when $p=2$ one may approximate the weights in \cite[p.218]{Muc} by their martingale approximations with dyadic step size sufficiently small; for $p \neq 2$ one may consider the analogues of the $p=2$ case as used in Section \ref{section:Eric_example_weights}, as they are analogues of the $p=2$ weights used in \cite[p.218]{Muc} and similar computations as there show the two-weight inequality for the maximal function fails \cite[Section 7.1, p.72-73]{SaWi}. Use the small-step construction of \cite{KaTr} to get $\tilde{\sigma}$ and $\tilde{\omega}$ which are $\tau$-doubling on $\mathcal{D} \left ( [0,1) \right )$ and satisfy $A_p ^{\mathcal{D} \left ( [0,1) \right )} \left ( \sigma, \omega \right)  \lesssim 1$.

  We now claim that
  \begin{align} \label{eq:KaTr_maximal_function_inc}
\mathfrak{N}_{\mathcal{M}^{\mathcal{D} \left ( [0,1) \right )},p} \left (\tilde{\sigma}, \tilde{\omega} \right ) \geq \mathfrak{N}_{\mathcal{M}^{\mathcal{D} \left ( [0,1) \right )},p} \left (\sigma, \omega \right ) \, ,  
\end{align}
 which will then immediately imply
\[
\mathfrak{N}_{\mathcal{M}^{\mathcal{D} \left ( [0,1) \right )},p} \left (\tilde{\sigma}, \tilde{\omega} \right ) > \Gamma \, .  \]
To see the claim, let $\Phi$ be the measure-preserving transformation arising from the small-step construction. Then by \cite[Remark 4.3]{KaTr}, we have for any function $f$ supported on $[0,1)$ that
\[
	\int \left | \mathcal{M}^{\mathcal{D} \left ( [0,1) \right )} \tilde{\sigma} f \circ \Phi  \right |^p d \tilde{\omega} \geq  \int \left | \mathcal{M}^{\mathcal{D} \left ( [0,1) \right )} \sigma f   \right |^p \circ \Phi d \tilde{\omega} = \int \left (\left | \mathcal{M}^{\mathcal{D} \left ( [0,1) \right )} \sigma f   \right |^p \circ \Phi \right ) \left ( \omega \circ \Phi \right ) dx =  \int \left | \mathcal{M}^{\mathcal{D} \left ( [0,1) \right )} \sigma f   \right |^p \omega dx \, ,
\]
where in the last equality we used the fact that $\Phi$ is measure-preserving. Hence
\[
  \left \| \mathcal{M}^{\mathcal{D} \left ( [0,1) \right )} \tilde{\sigma} f \circ \Phi  \right \|_{L^p \tilde{\omega}} \geq \left \| \mathcal{M}^{\mathcal{D} \left ( [0,1) \right )} \sigma f   \right \|_{L^p \omega} \, ,
\]
and by similar reasoning
\[
  \left \| f \circ \Phi  \right \|_{L^p \left ( \tilde{\sigma} \right )} = \left \| f   \right \|_{L^p \left (\sigma \right )} \, , 
\]
and hence \eqref{eq:KaTr_maximal_function_inc} holds.

We will now remodel $\tilde{\sigma}$ and $\tilde{\omega}$ to get doubling weights $\dddot{\sigma}$ and $\dddot{\omega}$ such $A_p (\dddot{\sigma}, \dddot{\omega}) \lesssim 1$ and yet
\[
\mathfrak{N}_{\mathcal{M}^{\mathcal{D} \left ( [0,1) \right )},p} \left (\dddot{\sigma}, \dddot{\omega} \right ) > \Gamma \, .
\]
One is tempted to apply Nazarov's remodeling argument as it is outlined in this current paper and set $\dddot{\sigma} \equiv \hat{\sigma}$ and $\dddot{\omega} \equiv \hat{\omega}$. And that would probably work. However showing the required estimates would be tedious as compared to the \emph{iterated remodeling} of \cite[Section 5]{KaTr}, which yields a much simpler argument. We will not go through the full details of the iterated remodeling of Kakaroumpas-Treil, but rather we merely mention that the iterated remodeling argument of \cite{KaTr} is very similar to Nazarov's remodeling, except that Kakaroumpas-Treil take special care to iterate their construction on the transition cubes of Nazarov. The reason for the iteration is that Kakaroumpas-Treil are then able to obtain a measure-preserving map $\Psi:[0,1) \to [0,1)$ such that, given $\tau$-doubling weights $\tilde{\omega}$ and $\tilde{\sigma}$ on $[0,1)$ satisfying $A_p (\tilde{\sigma}, \tilde{\omega})\lesssim 1$, if one defines 
\[
  \dddot{\sigma} \equiv \tilde{\sigma} \circ \Psi \, , \qquad \dddot{\omega} \equiv \tilde{\omega} \circ \Psi \, ,
\]
then one can then extend $\dddot{\sigma}$ and $\dddot{\omega}$ to $\mathbb{R}$ so that both are $\tau$-doubling on $\mathbb{R}$ and satisfy $A_p (\dddot{\sigma}, \dddot{\omega}) \lesssim 1$. Then to estimate the norm inequality, we use \cite[Remark 6.4]{KaTr} to get for all $f$ supported on $[0,1)$, and for a.e.\ $x\in [0,1)$, that
\[
  \mathcal{M} \left ( \dddot{\sigma} f \circ \Psi \right ) (x) = \mathcal{M} \left ( \tilde{\sigma} f  \right ) \circ \Psi (x) \, . 
\]
And so as before we get
\begin{align}
  \mathfrak{N}_{\mathcal{M}^{\mathcal{D} \left ( [0,1) \right )},p} \left (\dddot{\sigma}, \dddot{\omega} \right ) \geq \mathfrak{N}_{\mathcal{M}^{\mathcal{D} \left ( [0,1) \right )},p} \left (\tilde{\sigma}, \tilde{\omega} \right ) \, , 
\end{align}
and hence
\[
  \mathfrak{N}_{\mathcal{M}^{\mathcal{D} \left ( [0,1) \right )},p} \left (\dddot{\sigma}, \dddot{\omega} \right ) > \Gamma \, . 
\]
\end{proof}

\end{document}